\documentclass[a4paper,reqno]{amsart}
\usepackage{tikz-cd}
\usepackage{amssymb}
\usepackage{amsmath,amscd,color}
\usepackage{combelow}
\usepackage[all]{xy}
\usepackage{enumerate}
\usepackage{mathrsfs}
\usepackage{tensor}
\usepackage{stackrel}
\usepackage{hyperref}
\usepackage{mathrsfs}
\usepackage{hyperref}

\definecolor{tocolor}{rgb}{.1,.1,.5}
\definecolor{urlcolor}{rgb}{.2,.2,.6}
\definecolor{linkcolor}{rgb}{.1,.1,.6}
\definecolor{citecolor}{rgb}{.6,.2,.1}
\hypersetup{ colorlinks=true, urlcolor=urlcolor,
  linkcolor=linkcolor, citecolor=citecolor}
\definecolor{darkgreen}{rgb}{0.0, 0.5, 0.0}

\providecommand{\U}[1]{\protect\rule{.1in}{.1in}}
\newtheorem{theorem}{Theorem}[section]
\newtheorem*{theorem*}{Theorem}
\newtheorem*{claim*}{Claim}
\newtheorem{theoremM}{Theorem}

\newtheorem{corollary}[theorem]{Corollary}
\newtheorem{corollaryM}[theoremM]{Corollary}

\newtheorem*{question}{Question}
\newtheorem{definition}[theorem]{Definition}
\newtheorem{definitionM}[theoremM]{Definition}
\newtheorem{example}[theorem]{Example}
\newtheorem{lemma}[theorem]{Lemma}
\newtheorem{proposition}[theorem]{Proposition}
\newtheorem*{proposition*}{Proposition}
\newtheorem{remark}[theorem]{Remark}
\numberwithin{equation}{section}

\newcommand{\ii}{\tilde{i}}
\newcommand{\Sat}{\mathrm{Sat}}

\newcommand{\id}{\mathrm{Id}}
\newcommand{\G}{\mathcal{G}}
\newcommand{\F}{\mathcal{F}}
\renewcommand{\U}{\mathcal{U}}
\newcommand{\V}{\mathcal{V}}

\newcommand{\K}{\mathcal{K}}
\renewcommand{\H}{\mathcal{H}}

\newcommand{\ka}{\mathfrak{k}}

\renewcommand{\gg}{\mathfrak{g}}
\newcommand{\hh}{\mathfrak{h}}
\renewcommand{\graph}{\mathrm{Graph}}
\newcommand{\Rep}{V}
\newcommand{\Ato}{\Rightarrow}

\newcommand{\al}{\alpha}                
\newcommand{\be}{\beta}                 
\newcommand{\ga}{\gamma}                

\newcommand{\s}{\mathbf{s}}             
\renewcommand{\t}{\mathbf{t}}           

\newcommand{\tto}{\rightrightarrows}    
\newcommand{\timesst}{\tensor[_\s]{\times}{_\t}} 

\renewcommand{\L}{\mathbb{L}}
\newcommand{\Lie}{\mathscr{L}}
\newcommand{\dd}{\mathrm{d}}
\renewcommand{\d}{\dd}
\newcommand{\DD}{\mathcal{D}}
\newcommand{\cl}{\mathrm{cl}}
\newcommand{\bas}{\mathrm{bas}}
\newcommand{\Der}{\operatorname{Der}}
\newcommand{\pr}{\operatorname{pr}}
\newcommand{\Aut}{\operatorname{Aut}}

\newcommand{\ad}{\operatorname{ad}}

\newcommand{\im}{\operatorname{Im}}
\newcommand{\lin}{\mathrm{lin}}
\newcommand{\mult}{\mathrm{M}}
\newcommand{\imult}{\mathrm{IM}}
\newcommand{\can}{\mathrm{can}}

\newcommand{\Ad}{\operatorname{Ad}}
\newcommand{\R}{\mathbb{R}}
\newcommand{\X}{\mathfrak{X}}
\newcommand{\diffto}{\xrightarrow{\raisebox{-0.2 em}[0pt][0pt]{\smash{\ensuremath{\sim}}}}}
\newcommand{\rmap}{\longrightarrow}

\newcommand{\mmu}{\mu}

\newcommand{\pd}[1]{\partial_{#1}} 
\newcommand{\pdd}[2]{\frac{\partial #1}{\partial #2}}

\begin{document}
\title{Poisson geometry around Poisson submanifolds}

\author{Rui Loja Fernandes}
\address{Department of Mathematics, University of Illinois at Urbana-Champaign, 1409 W. Green Street, Urbana, IL 61801 USA}
\email{ruiloja@illinois.edu}

\author{Ioan M\u{a}rcu\cb{t}}
\address{Radboud University Nijmegen, IMAPP, 6500 GL, Nijmegen, The Netherlands}
\email{i.marcut@math.ru.nl}

\thanks{RLF was partially supported by NSF grants DMS-1710884, DMS-2003223 and DMS-2303586, a Simons Fellowship in Mathematics and FCT/Portugal. IM was partially supported by the NWO VENI grant 613.009.031.}

\begin{abstract}
We construct a first order local model for Poisson manifolds around a large class of Poisson submanifolds and we give conditions under which this model is a local normal form. The resulting linearization theorem includes as special cases all the known linearization theorems for fixed points and symplectic leaves. The symplectic groupoid version of these results gives a solution to the groupoid coisotropic embedding problem.
\end{abstract}
\maketitle

\setcounter{tocdepth}{1}
\tableofcontents

\section{Introduction}

A fundamental problem in Poisson geometry is to understand the behavior of a Poisson manifold around a given symplectic leaf \cite{DuZu05,Vorobjev05,Weinstein83}. There is a well-known first order local model for a Poisson structure around a symplectic leaf due to Vorobjev \cite{Vorobjev01}. This first order local model is the cornerstone for the study of the Poisson geometry around the leaf. Using it one can study normal forms around symplectic leaves \cite{CrMa12}, stability of symplectic leaves \cite{CrFe10}, etc.

Symplectic leaves is only one class among several interesting classes of submanifolds in Poisson geometry. A local normal form around Poisson transversals has also been found in \cite{FrMa17}, and one hopes to understand the geometry around other important classes of submanifolds, such as Poisson submanifolds, coisotropic submanifolds or Poisson-Dirac submanifolds. Unlike the special case of symplectic leaves, in general, there is no first order local model around a  Poisson submanifold. In this paper we identify conditions for the existence of such a model, extending the known model for symplectic leaves, and we prove a normal formal theorem generalizing the known (smooth) linearization results for fixed points \cite{Conn85,DuZu05} and symplectic leaves \cite{CrMa12}.

In order to state our results, we introduce some notation and terminology. Let us denote by $\X^\bullet(M)$ the space of multivector fields on a manifold $M$. We have the space of Poisson structures on $M$, denoted by
\[ \Pi(M):=\{ \pi\in\X^2(M): [\pi,\pi]=0\}. \]
Given a submanifold $S\subset M$, which will always assume to be closed and embedded, we denote by $\X^\bullet_S(M)$ the subspace of multivector fields $\vartheta\in \X^\bullet(M)$ tangent to $S$, i.e., such that $\vartheta|_S\in\X^\bullet(S)$. Then we have the space of Poisson structures in $M$ for which $S$ is a Poisson submanifold, denoted by
\[ \Pi(M,S):=\{ \pi\in\X^2_S(M): [\pi,\pi]=0\}. \]
Let $I_S\subset C^\infty(M)$ denote the vanishing ideal of $S$. The space of first order jets of multivector fields tangent to $S$ can be identified with (see Section \ref{section:jets}) the quotient
\[ J^1_S\X_S^\bullet(M):=\X^\bullet_S(M)/I_S^2\cdot \X^\bullet(M). \]
The Schouten bracket descends to $J^1_S\X_S^\bullet(M)$ and we have the space of \emph{first order jets of Poisson structures at $S$}, defined as follows
\[ J^1_S\Pi(M,S):=\{\tau\in J^1_S\X_S^2(M): [\tau,\tau]=0\}. \]

We will see that to specify an element $\tau\in J^1_S\Pi(M,S)$ amounts to giving a Lie algebroid structure on the restricted cotangent bundle $T^*_S M$ for which the natural projection $\mu_S:T^*_SM\to T^*S$ is an infinitesimal multiplicative (IM) closed 2-form. We do not distinguished between these two descriptions, so we will often denote an element $\tau\in J^1_S\Pi(M,S)$ as a pair $\tau=(T^*_SM,\mu_S)$.


We have a map which to a Poisson structure tangent to $S$ associates its first order jet along $S$, denoted by
\[  J^1_S:\Pi(M,S)\to J^1_S\Pi(M,S). \]
This map, in general, is not surjective (see Example \ref{ex:non:holonomic}).


\begin{definitionM}
\label{def:local:model}
Given a class $\mathscr{C}\subset J^1_S\Pi(M,S)$, we call a splitting $\sigma:\mathscr{C}\to \Pi(M,S)$ of the map $J^1_S$ a \textbf{first order local model} for the class $\mathscr{C}$. 
\end{definitionM}

We state now a simplified version of our first main result (see Section \ref{sec:local:model} for more precise statements).

%
%

\begin{theoremM}[Existence of local models]
\label{thm:main:one}
The class $\mathscr{C}\subset J^1_S\Pi(M,S)$ of first order jets of Poisson structures satisfying the partially split condition admits a first order local model $\sigma:\mathscr{C}\to \Pi(M,S)$.
\end{theoremM}

The ``partially split'' condition in the statement of the theorem is a condition on a jet $\tau\in J^1_S\Pi(M,S)$ expressed in terms of the IM (infinitesimally multiplicative) geometry of the pair $\tau=(T^*_SM,\mu_S)$. It amounts to the existence of an IM Ehresmann connection \cite{FM22}, and we will give the precise definition later. We use the IM Ehresmann connection to construct the local model $\sigma(\tau)$ from Theorem \ref{thm:main:one} explicitly. Moreover, we show that, up to isomorphism, the result is independent of this choice. For now, we point out that the class of jets satisfying this condition is a large class that includes many examples of interest. For example, we will see that it contains the following three classes of first order jets $\tau=(T^*_SM,\mu_S)$: 
\begin{itemize}
\item $T^*_SM$ is transitive;
\item $T^*_SM$ integrates to a proper groupoid;
\item $\ker\mu_S\subset T^*_SM$ is a bundle of semi-simple Lie algebras.
\end{itemize}
\smallskip


\begin{definitionM}
Consider the map $\sigma:\mathscr{C}\to \Pi(M,S)$ from Theorem \ref{thm:main:one}. Given a Poisson structure $\pi\in \Pi(M,S)$ with $J^1_S\pi\in \mathscr{C}$, we will call $\sigma(J^1_S\pi)\in \Pi(M,S)$ the \textbf{linear approximation} for $\pi$ around $S$. If $\pi$ and $\sigma(J^1_S\pi)$ are isomorphic around $S$ via a Poisson diffeomorphism fixing $S$, then we will say that $\sigma(J^1_S\pi)$ is a \textbf{normal form} for $\pi$ around $S$.
\end{definitionM}

Our second main result is the following linearization theorem  (see Theorem \ref{Normal:form:theorem}).

\begin{theoremM}[Normal form]
\label{thm:main:two}
Let $\mathscr{C}_0\subset J^1_S\Pi(M,S)$  consist of first order jets $\tau=(T^*_SM,\mu_S)$ such that $T^*_SM$ is integrable by a compact, Hausdorff Lie groupoid whose target fibers have trivial second de Rham cohomology. Then there exists a first order local model $\sigma:\mathscr{C}_0\to \Pi(M,S)$ and $\sigma(J^1_S\pi)$ is a normal form for $\pi\in \Pi(M,S)$ around $S$ whenever $J^1_S\pi\in \mathscr{C}_0$.
\end{theoremM}

This result includes as special cases the most important smooth linearization theorems around fixed points and symplectic leaves from the literature.
\begin{enumerate}[-]
\item For a fixed point $S=\{x_0\}$, the first order jet $(T^*_SM,\mu_S)$ becomes just a Lie algebra $\gg$. The local model always exists and is just the linear Poisson structure on $\gg^*$. If $\gg$ is compact, semi-simple, then any of its integrations is a compact Lie group $G$ with $H^2(G)=\{0\}$, so satisfies the assumptions of the theorem and one recovers Conn's Linearization Theorem \cite{Conn85}. Actually, one can allow $\gg$ to be compact with 1-dimensional center, resulting in the slightly more general result due to Monnier and Zung \cite{MoZu04}.
\item For a symplectic leaf $S$, the first order jet $(T^*_SM,\mu_S)$ is a transitive algebroid over a symplectic manifold. The local model always exists and it coincides with the one constructed by Vorobjev \cite{Vorobjev01}. Theorem \ref{thm:main:two} includes the linearization theorem due to Crainic and M\u{a}rcu\cb{t} \cite{CrMa12} and its generalization due to M\u{a}rcu\cb{t} \cite{Marcut14}.
\item For a general Poisson submanifold, the local model can be regarded as a way of gluing together Vorobjev-type local models above the symplectic leaves of $(S,\pi|_S)$ (see Proposition \ref{prop:glueing:vorobjev}). For Poisson submanifolds, the main result to date was the Rigidity Theorem due to M\u{a}rcu\cb{t} \cite{Marcut14}. The following improved version of this result is a direct consequence of Theorem \ref{thm:main:two} (see Corollary \ref{thm:main:three:end}).
\end{enumerate}

\begin{corollaryM}[Neighborhood equivalence]
\label{thm:main:three}
Let $(M,\pi)$ be a Poisson manifold and $S\subset M$ a Poisson submanifold. If $T^*_S M$ is integrable by a compact Hausdorff Lie groupoid whose target fibers have trivial second de Rham cohomology, then any Poisson structure $\pi'$ on $M$ with $J^1_S\pi=J^1_S\pi'$ is locally isomorphic to $\pi$ around $S$.
\end{corollaryM}

Our proof of Theorem \ref{thm:main:two}, and hence of Corollary \ref{thm:main:three}, will make use of the Rigidity Theorem of \cite{Marcut14}. However, Corollary \ref{thm:main:three} is an improvement, since the Rigidity Theorem assumes that $(M,\pi)$ is \emph{a priori} integrable, while our result assumes that only the restriction $T^*_SM$ is integrable. This is a consequence of the existence of our local model. For example, we obtain the following statement (see Corollary \ref{cor:integrability}).

\begin{corollaryM}
\label{cor:main}
Let $(M,\pi)$ be a Poisson manifold and $S\subset M$ a Poisson submanifold. If $T^*_S M$ is integrable by a compact Hausdorff Lie groupoid whose target fibers have trivial second de Rham cohomology, then there is a neighborhood $U\subset M$ of $S$ consisting of compact symplectic leaves and such that $\pi|_U$ is integrable. 
\end{corollaryM}

The proof of the  Rigidity Theorem in \cite{Marcut14} uses a version of the Nash-Moser fast convergence method, similar to Conn's proof \cite{Conn85}. We will also give a direct proof of Theorem \ref{thm:main:two} under the additional assumption that the integration is target 1-connected, by using techniques developed by Crainic and Fernandes in \cite{CrFe11}. This then leads to a geometric proof of Theorem \ref{thm:main:three} under this additional assumption, which is independent from the results in \cite{Marcut14}.


The results above can be generalized in several directions. For example, we will consider extensions to Dirac geometry, constructing first order local models and proving normal forms results for Dirac manifolds around invariant submanifolds (i.e., manifolds saturated by presymplectic leaves). Another possible direction, is to construct local models and normal forms around other types of submanifolds. We will introduce and discuss one such class, that we call \emph{coregular submanifolds}. These include as special cases transversals submanifolds and invariant submanifolds. By passing to a saturation of a coregular submanifold we are able to extend our results to such submanifolds.

\medskip


Another goal of the paper is to give a symplectic groupoid version of the previous results. Besides being interesting on its own right, it provides a guide for both the construction of local models and the proofs of various linearization results for Poisson structures. It can be understood by applying a dictionary \emph{global} $\leftrightarrow$ \emph{infinitesimal}, whose first entries are as follows:
\medskip

\begin{center}
\begin{tabular}{ c | c }
\emph{Global side} & \emph{Infinitesimal side}\\ \hline\\
 \txt{symplectic groupoid\\ $(\G,\omega)\tto M$}  & \txt{Poisson manifold $(M,\pi)$} \\ 
 \\
 \txt{saturated submanifold\\ $S\subset M$}  & \txt{Poisson submanifold\\ $S\subset M$}  \\ 
\\
 \txt{over-symplectic groupoid\\ $(\G_S,\omega_S)\tto S$}  & \txt{first order local data\\ $\tau\in J^1_S\Pi(S,M)$} \\ 
\\
 \txt{coisotropic embedding\\ $(\G_S,\omega_S)\hookrightarrow (\G,\omega)$}  & 
\txt{A Poisson structure $\pi\in \Pi(M,S)$\\ with $\tau=J^1_S\pi$ } 
 \\ 
\end{tabular}
\end{center}
\medskip

\noindent The first two entries in this dictionary are well-known (see, e.g., \cite{Weinstein87,CrFe04}), so let us discuss briefly the other entries. Given a symplectic groupoid $(\G,\omega)\tto M$ integrating a Poisson manifold $(M,\pi)$, its restriction to a closed, embedded Poisson submanifold $S\subset M$ yields a coisotropic groupoid embedding 
\[ (\G|_S,\omega|_S)\hookrightarrow (\G,\omega). \]
The restriction $(\G|_S,\omega|_S)$ is an example of an  \emph{over-symplectic groupoid} in the sense of \cite{BCWZ04} (see Definition \ref{def:over-symplectic}).
The {\bf groupoid coisotropic embedding problem} asks:
\begin{itemize}
\item Given an over-symplectic groupoid $(\G_S,\omega_S)$, is there a groupoid coisotropic embedding $i:(\G_S,\omega_S)\hookrightarrow (\G,\omega)$ into some symplectic groupoid?
\item Given two groupoid coisotropic embeddings $i_k:(\G_S,\omega_S)\hookrightarrow (\G_k,\omega_k)$, $k=1,2$, are they locally isomorphic around $\G_S$, i.e., is there a local symplectic groupoid isomorphism $\Phi:(\G_1,\omega_1)\to (\G_2,\omega_2)$ such that $\Phi\circ i_1=i_2$?
\end{itemize}

We will obtain answers to these questions by applying a multiplicative version of the classical Gotay coisotropic embedding theorem \cite{Gotay82}. For that, given an over-symplectic groupoid $(\G_S,\omega_S)$, the kernel of $\omega_S$ yields a vector bundle (VB) subgroupoid (in the sense of \cite{Mackenzie05} or \cite{BCdH16}) of the tangent groupoid
\[ \ker\omega_S\subset T\G_S. \]
\begin{definitionM}
An over-symplectic groupoid $(\G_S,\omega_S)$ is called \textbf{partially split} if there exists a VB subgroupoid $E\subset T\G_S$ complementary to the kernel of $\omega_S$, 
\[ T\G_S=\ker\omega_S\oplus E. \]
\end{definitionM}

The distribution $E$ in this definition is an example of a \emph{multiplicative Ehresmann connection}, introduced in \cite{FM22}. The results of \cite{FM22} imply, for example, that over-symplectic groupoids that are either transitive or proper are partially split. The partially split condition for first order jets of Poisson structures, mentioned before, is the infinitesimal version of this condition under the aforementioned dictionary.

Given a partially split over-symplectic groupoid $(\G_S,\omega_S)$ we construct a local model as follows. On the one hand, the groupoid $\G_S\tto S$ acts on the bundle $\ka=(\ker\omega_S)|_S\to S$, and on its dual, and builds the action groupoid
\[ \G_S\ltimes\ka^*\tto \ka^*. \]
On the other hand, we will see that a partial splitting can be encoded by a \emph{multiplicative connection 1-form}  $\al\in\Omega^1(\G_S;\ka)$ (see Section \ref{sec:grpd}). This allows us to define on the action groupoid the closed, multiplicative, 2-form
\[ \omega_0:=\pr_{\G_S}^*\omega_S+\d\langle\al,\cdot\rangle. \]
Then $\omega_0$ is non-degenerate on an open subgroupoid $\G_0\subset \G_S\ltimes\ka ^*$ which contains $\G_S\simeq \G_S\ltimes 0_S$. So $(\G_0,\omega_0)$ is a symplectic groupoid, and we obtain the next result (see Proposition \ref{prop:partially:split:grpd:optimal}). 

\begin{theoremM}[Existence of groupoid coisotropic embeddings]
\label{thm:main:grpd:embedding}
Given a partially split over-symplectic groupoid $(\G_S,\omega_S)$ the zero section $i_0:(\G_S,\omega_S)\hookrightarrow  (\G_0,\omega_0)$ is a coisotropic embedding.
\end{theoremM}

We will show that the partially split condition is in fact equivalent to the existence of a multiplicative symplectic structure $\omega_0$ on a neighborhood of $\G_S$ in the groupoid local model $\G_S\ltimes \ka^*\tto \ka^*$ for which $i_0$ is a coisotropic embedding. Moreover, we will show that any two such symplectic structures are related by a groupoid isomorphism, defined around $\G_S$.

The specific Lie groupoid $\G_S\ltimes \ka^*\tto \ka^*$ appearing in the local model coincides with the local model around invariant submanifolds for general Lie groupoids \cite{dHFe18}. Namely, if $(\G_S,\omega_S)$ is embedded as a coisotropic subgroupoid of the symplectic groupoid $(\G,\omega)$, then we have a canonical isomorphism of groupoids \[\nu(\G_S)\tto \nu(S)\quad \simeq\quad  \G_S\ltimes \ka^*\tto \ka^*,\]
where $\nu(\G_S)\tto \nu(S)$ is the normal bundle of $\G_S$ in $\G$, which plays the role of the linear approximation of $\G$ around $\G_S$.

In the proper case, we obtain a normal form result for groupoid coisotropic embeddings, which is the global counterpart of Theorem \ref{thm:main:two} (see Corollary \ref{corollary:linearization:proper:Hausdorff} and the comments following it).

\begin{theoremM}[Normal form]
\label{thm:main:grpd:normal:form}
Let $(\G_S,\omega_S)$ be a target-proper, Hausdorff, over-sym\-plec\-tic groupoid with connected target fibers. Any groupoid coisotropic embedding into a Hausdorff symplectic groupoid $i:(\G_S,\omega_S)\hookrightarrow (\G,\omega)$ is isomorphic around $\G_S$ to the coisotropic groupoid embedding $i_0:(\G_S,\omega_S)\hookrightarrow  (\G_0,\omega_0)$ from Theorem \ref{thm:main:grpd:embedding}.
\end{theoremM}

Finally, as a consequence of this result, we have a groupoid coisotropic neighborhood equivalence theorem, which is the global counterpart of Corollary \ref{thm:main:three}.

\begin{corollaryM}[Neighborhood equivalence]
\label{thm:main:grpd:equivalence}
Let $(\G_S,\omega_S)$ be a target-proper, Hausdorff, over-sym\-plec\-tic groupoid with connected target fibers. Any two groupoid coisotropic embeddings into Hausdorff symplectic groupoids $i_k:(\G_S,\omega_S)\hookrightarrow (\G_k,\omega_k)$, $k=0,1$, are locally isomorphic around $\G_S$.
\end{corollaryM}



\noindent\textbf{Organization of the paper}. We start by introducing, in Section \ref{section:jets}, the appropriate first order data of a Poisson structure along a Poisson submanifold using the language of jets. In Section \ref{sec:oversymplectic}, we recall some results about over-symplectic groupoids and explain the relation to first order jets. 
In Section \ref{sec:grpd}, we work on the global side and consider the groupoid coisotropic embedding problem, proving Theorems \ref{thm:main:grpd:embedding} and \ref{thm:main:grpd:normal:form}. 
We turn to the infinitesimal side in Section \ref{sec:local:model}, where we introduce the first order local model and prove Theorem \ref{thm:main:one}. In Section \ref{sec:linearization:algbrds} we discuss linearization and we prove the local normal form (Theorem \ref{thm:main:two}). 
Sections \ref{sec:examples:over:sympl}, \ref{sec:examples:groupoid:local:models} and \ref{sec:examples} present many examples, are independent of the rest of the text. In Appendix \ref{appendix}, we discuss the Cartan calculus of multiplicative forms and infinitesimal multiplicative forms, and we prove the multiplicative and IM versions of Moser's theorem. Appendices \ref{sec:Dirac} and \ref{sec:coregular} contain extensions of our results to the Dirac setting and to coregular submanifolds, respectively. 

\medskip
\noindent\emph{Acknowledgments}.  
We would like to thank Olivier Brahic, Alejandro Cabrera, Marius Crainic, Matias del Hoyo, Pedro Frejlich, Stephane Geudens and Florian Zeiser, for various comments and suggestions while this work was completed.

\section{Jets of Poisson structures}
\label{section:jets}
Let $(M,\pi)$ be a Poisson manifold and $S\subset M$ a Poisson submanifold. We will always assume that $S$ is embedded and closed in $M$. Since we are interested only in the local picture around $S$, the second assumption is not very restrictive, as it always holds in a tubular neighborhood of $S$.

In describing the Poisson geometry of $(M,\pi)$ around $S$ the first order jet of $\pi$ \emph{along} $S$ plays a crucial role. The space of multivector fields tangent to $S$ is closed under the Schouten bracket, and will be denoted by
\[ \X^\bullet_S(M):=\{\vartheta\in\X^\bullet(M):\vartheta|_S\in\Gamma(\wedge^\bullet TS)\}.\]
The space of Poisson structures tangent to $S$, denoted by
\[ \Pi(M,S):=\{\pi\in \X^2_S(M):[\pi,\pi]=0 \},\]
consists precisely of the Poisson structures on $M$ for which $S$ is a Poisson submanifold. We have the obvious restriction map to Poisson structures on $S$
\[ \Pi(M,S) \to \Pi(S), \quad \pi\mapsto \pi_S:=\pi|_S. \]

On the other hand, let $I_S\subset C^{\infty}(M)$ denote the vanishing ideal of $S$. For each $k\geq 1$, $I^{k}_S\cdot \X^{\bullet}(M)$ is also a Lie ideal in $\X^{\bullet}_S(M)$. In particular the quotient
\[ J^1_S\X^{\bullet}_S(M):= \X^{\bullet}_S(M)/(I^{2}_S\cdot \X^{\bullet}(M)) \]
inherits the structure of a graded Lie algebra. 

\begin{definition}
An element of the space
\[ J^1_S\Pi(M,S):=\big\{\tau\in  J^1_S\X^{2}_S(M)\, :\, [\tau,\tau]=0\big\}, \]
is called a {\bf first order jet of a Poisson structure} at $S$. 
\end{definition}

We will give in later a more practical definition (see Remark \ref{rem:def:1st:order:jets}). For now, a first order jet of a Poisson structure at $S$ is a class of a bivector field tangent to $S$, modulo bivector fields that vanish to second order along $S$, and which satisfies the equation of being a Poisson structure up to second order
\[ J^1_S\Pi(M,S)=\frac{\{\pi\in \X^2_S(M): [\pi,\pi] \in I^2_S\cdot \X^3(M) \} }{I^2_S\cdot \X^2(M)}.\]

Notice that we have a commutative diagram
\begin{equation}
\label{eq:diagram:jets}
\vcenter{\xymatrixrowsep{0.4cm} \xymatrixcolsep{0.2cm}
\xymatrix{\Pi(M,S)\ar[rr]^{J^1_S}\ar[dr]&& J^1_S\Pi(M,S)\ar[dl]\\
&\Pi(S)&}}
\end{equation}

Although one can express any first order jet $\tau\in J^1_S\Pi(M,S)$ in the form $\tau=J^1_S\pi$ for some bivector field $\pi\in\X^2_S(M)$, in general, one may not be able to choose $\pi$ to be a Poisson structure. In other words, the map $J^1_S$ in the previous diagram, in general, is not surjective. This is illustrated in the next example.

\begin{example}[\cite{Marcut12}]
\label{ex:non:holonomic}
Let $M=\R^3$ and $S=\{(x,y,0): x,y\in \R\}\subset \R^3$. The bivector field $\pi\in\X_S^2(M)$ given by
\[ \pi= z \pd{x}\wedge\pd{y}+xz \pd{x}\wedge\pd{z}\]
satisfies
\[ [\pi,\pi]=2z^2 \pd{x}\wedge\pd{y}\wedge\pd{z}\in I_S^2\cdot \X^3(M). \]
It follows that $\tau=J^1_S(\pi)\in J^1_S\Pi(M,S)$, with induced Poisson structure $\pi_S=0$.

Any bivector field $\tilde{\pi}\in\X^2_S(M)$ such that $J^1_S(\tilde{\pi})=\tau$ takes the form
\[ \tilde{\pi}= (z+z^2f_1) \pd{x}\wedge\pd{y}+(xz+z^2f_2) \pd{x}\wedge\pd{z}+z^2f_3\pd{y}\wedge\pd{z}, \]
for some smooth functions $f_1,f_2,f_3\in C^\infty(\R^3)$. Next, one finds that
\[ [\tilde{\pi},\tilde{\pi}]=2z^2(1+O(1))\pd{x}\wedge\pd{y}\wedge\pd{z}.\]
Therefore, $[\tilde{\pi},\tilde{\pi}]\neq 0$, and so $\tau$ does not belong to the image of $J^1_S$.
\end{example}

In the rest of this section we give an alternative description of first order jets in terms of infinitesimal multiplicative (IM) closed 2-forms (see Appendix \ref{appendix} for a recollection of definitions and basic facts about IM forms). 

In order to motivate this description, recall that a Poisson manifold $(M,\pi)$ has an associated cotangent Lie algebroid $T^*M\Rightarrow M$ with anchor $\pi^\sharp$ and Lie bracket
\begin{align}\label{formula:Lie:bracket}
[\al,\be]_\pi:&=\Lie_{\pi^\sharp(\al)}\be-\Lie_{\pi^\sharp(\be)}\al-\d \pi(\al,\be).
\end{align}
It is a classical result, first observed in \cite{CDW87}, that one can actually characterize Poisson structures by such Lie algebroid structures. More precisely,  it is proved in \cite{CDW87} that there is a 1-to-1 correspondence between
\[ 
\left\{\txt{\vspace*{0.1cm}\\ Poisson structures\\ $\pi\in\X^2(M)$\\\vspace*{0.1cm} \,}\right\}
\tilde{\longleftrightarrow}
\left\{\txt{Lie algebroids $(T^*M,[\cdot,\cdot],\rho)$ with\\ $\rho:T^*M\to TM$ skew-symmetric\\ $[\Omega^1_\cl(M),\Omega^1_\cl(M)]\subset \Omega^1_\cl(M)$\,} \right\}
\]
Using the language of IM forms, this classical result can be reformulated as stating that there is a 1-to-1 correspondence (see \cite{BCWZ04} and also Corollary \ref{cor:jets:algebroids} below)
\[ 
\Pi(M)\equiv \left\{\txt{Poisson structures\\ $\pi\in\X^2(M)$ \,}\right\}
\tilde{\longleftrightarrow}
\left\{\txt{Lie algebroids $(T^*M,[\cdot,\cdot],\rho)$ with \\ $\id:T^*M\to T^*M$ a closed IM 2-form\,} \right\}
\]
Now the main result of this section is the following  submanifold analogue of this classical result. We denote by $\mu_S:T^*_SM\to T^*S$ the restriction map $\al\mapsto \al|_{TS}$.

\begin{proposition}
\label{prop:1st-order-jets:IMforms}
Given a submanifold $S\subset M$, there is a 1-to-1 correspondence
\[ 
J^1_S\Pi(M,S)=
\left\{\txt{$\tau\in J^1_S\X^2_S(M)$ with \\$[\tau,\tau]=0$ \\ \,}\right\}
\tilde{\longleftrightarrow}
\left\{\txt{Lie algebroids $(T^*_SM,[\cdot,\cdot],\rho)$\\ with $\mu_S:T^*_SM\to T^*S$\\ a closed IM 2-form\,} \right\}.
\]
\end{proposition}

\begin{remark}[Working definition of first order jet of a Poisson structure]
\label{rem:def:1st:order:jets}
This description of first order jets turns out to be much more practical and we will adopt it as our working definition. Hence, we will often denote an element $\tau\in J^1_S\Pi(M,S)$ as a pair $\tau=(T^*_SM,\mu_S)$.
\end{remark}

\begin{example}
\label{ex:non:holonomic:cont}
In Example \ref{ex:non:holonomic}, the first order jet $\tau\in J^1_S\Pi(M,S)$ corresponds to the pair $(T^*_SM,\mu_S)$, where $T^*_SM$ is the Lie algebra bundle (i.e., $\rho=0$)
\[ \R^2\times\R^3\to \R^2 \] 
with Lie bracket given on the basis $e_1=(\d x)|_S$, $e_2=(\d y)|_S$, $e_3=(\d z)|_S$ by:
\[ [e_1,e_2]=e_3,\quad [e_1,e_3]=xe_3,\quad [e_2,e_3]=0, \]
and closed IM 2-form $\mu_S:T^*_SM\to T^*S$ given by:
\begin{equation}
    \label{eq:IM:form:non-split}
    \mu_S(e_1)=\d x,\quad \mu_S(e_2)=\d y,\quad \mu_S(e_3)=0.
\end{equation}
\end{example}

The rest of this section is dedicated to proving the correspondence above as well as relating it with the usual theory of jet bundles. The reader who is willing to accept this identification can skip the remainder of the section altogether. 

\begin{remark}
There is yet another approach to first order jets along a submanifold. It starts with the observation that if $S$ is a Poisson submanifold of $(M,\pi)$ then $I_S$ is a Poisson ideal of $(C^\infty(M),\{\cdot,\cdot\})$. The quotient $C^\infty(M)/I_S$ is a Poisson algebra canonically isomorphic to $(C^\infty(S),\{\cdot,\cdot\}_S)$, and the higher order quotients $C^\infty(M)/I_S^{k+1}$ can be interpreted as $k$th-order jets of $\pi$ at $S$. In fact, given any manifold $M$ with a closed embedded submanifold $S\subset M$, it is not hard to see that there is 1-to-1 correspondence between Poisson algebra structures on $C^\infty(M)/I_S^{2}$ such that $I_S/I_S^2$ is a Lie ideal and first order jets of a Poisson structure at $S$, as defined above (see \cite{MarcutPhD} for a proof, and \cite{Vor20} for further applications). 
\end{remark}

\subsection{Jets of multivector fields and derivations}
\label{subsection:jets:derivations}


We denote by $J^nE\to M$ the $n$-th order jet bundle of a fiber bundle $E\to M$. A section $s\in\Gamma(J^nE)$ is holonomic if it is the $n$th-order jet of a section $f\in\Gamma(E)$, i.e., if $s=J^nf$. We will be interested in first order jets of the bundle $E=\wedge^2 TM$. This is motivated by the following observation. The Lie bracket \eqref{formula:Lie:bracket} of the cotangent algebroid of a Poisson manifold $(M,\pi)$ depends pointwise on the values and the first order partial derivatives of $\pi$, i.e., on the values of the section $J^1\pi \in \Gamma(J^1(\wedge^2 TM))$. Moreover, we will see that the first jet encodes precisely the data of the cotangent Lie algebroid. More generally, we will see which (possibly non-holonomic) sections $\tau\in\Gamma(J^1(\wedge^2TM))$ give rise to Lie algebroids on $T^*M$. 
%
%
For this, we recall briefly the approach to Lie algebroid structures using multiderivations of a vector bundle developed in \cite{CrMo08}. 

A derivation of degree $d$ of a vector bundle $E\to M$ is an $\R$-multilinear, alternating map 
\[ D:\underbrace{\Gamma(E)\times\dots\times \Gamma(E)}_{d+1\text{-times}}\to \Gamma(E), \]
for which there exists a bundle map $\sigma_D:\wedge^{d}E\to TM$, called the symbol, such that
\[ D(s_0,\dots,fs_{d})=f D(s_0,\dots,s_{d})+(\Lie_{\sigma_D(s_0,\dots,s_{d-1})} f) s_{d}, \]
for any $s_0,\dots,s_{d}\in\Gamma(E)$ and $f\in C^\infty(M)$. The space of derivations of degree $d$ is denoted by $\Der^d(E)$. It is the space of sections of a vector bundle $\DD^d E\to M$, which fits into a short exact sequence of vector bundles
\[
\xymatrix{ 0\ar[r] & \wedge^{d+1} E^*\otimes E\ar[r] & \DD^{d}E \ar[r]^---{\sigma_D} &  \wedge^{d} E^*\otimes TM\ar[r] & 0}
\]
There is a graded Lie bracket
\[  [~,~]: \Der^{k}(E)\times \Der^{l}(E) \to \Der^{k+l}(E), \]
and it is proved in \cite{CrMo08} that specifying a Lie algebroid structure on $E\to M$ amounts to giving a degree 1 derivation $D\in\Der^1(E)$ such that $[D,D]=0$. The derivation $D$ defines the Lie bracket by $[\al,\be]:=D(\al,\be)$, 
while the anchor coincides with its symbol: $\rho:=\sigma_D$.
\smallskip

In our case of interest $E=T^*M$, we can generalize formula \eqref{formula:Lie:bracket} and obtain a map from jets of multivector fields to multiderivations.

\begin{proposition}
\label{prop:jets:derivations}
\begin{enumerate}[(a)]\mbox{}
\item We have a map of graded Lie algebras
\[i:(\mathfrak{X}^{d+1}(M),[~,~])\to (\Der^{d}(T^*M),[~,~]),\]
\[i(\vartheta)(\al_0,\ldots,\al_{d})=\d(\vartheta(\al_0,\dots,\al_d))+\sum_{i=0}^d (-1)^{d-i} i_{\vartheta^\sharp(\al_0,\dots,\widehat{\al_i},\dots,\al_d)} \d\al_i,
\]
\item
The map in (a) induces an injective map of vector bundles 
\[\ii: J^1(\wedge^{d+1} TM)\hookrightarrow \DD^d(T^*M),\]
whose image consists of elements $D\in  \DD^d(T^*M)$ with skew-symmetric symbol $\sigma_D\in\wedge^{d+1} TM\subset \wedge^{d}TM\otimes TM$.
\end{enumerate}
\end{proposition}
\begin{proof}
Note that $i(\vartheta)$ is determined by the relations
\[i(\vartheta)(\d f_0,\ldots, \d f_d)=\d (\vartheta(f_0,\ldots, f_d)), \quad \sigma_{i(\vartheta)}(\d f_1,\ldots, \d f_{d})=\vartheta^{\sharp}(\d f_1,\ldots, \d f_d),\]
for all $f_0,\ldots, f_d\in C^{\infty}(M)$.
Both the Schouten bracket and the bracket on derivations can be described as graded commutators. Using these, one obtains the relation $i([\vartheta,\tau])=[i(\vartheta),i(\tau)]$ on exact 1-forms. Using the formula for the symbol of the commutator of derivations given in the proof of \cite[Proposition 1]{CrMo08}, one obtains that $\sigma_{[i(\vartheta),i(\tau)]}=\sigma_{i([\vartheta,\tau])}$. This implies item $(a)$. Item $(b)$ is straightforward. 
\end{proof}

\begin{remark}The derivations with skew-symmetric symbol do not form a subalgebra of all derivations. Therefore, one cannot use the map from item (b) of the proposition to obtain a bracket on first jets. 
\end{remark}
%

This leads to the following interpretation of jets as Lie algebroids:

\begin{corollary}
\label{cor:jets:algebroids}
There is a 1-to-1 correspondence
\[ 
\left\{\txt{$\tau\in\Gamma(J^1(\wedge^2TM))$ with \\$[\ii(\tau),\ii(\tau)]=0$ \\ \,}\right\}
\tilde{\longleftrightarrow}
\left\{\txt{Lie algebroid structures on $T^*M$\\ with skew-symmetric anchor\\ $\rho:T^*M\to TM$\\ \,} \right\}
\]
Under this correspondence, holonomic sections (i.e., Poisson structures) corresponds to Lie algebroid structures for which the identity map $\id:T^*M\to T^*M$ is a closed IM 2-form.
\end{corollary}

\begin{proof}
The first part of the corollary follows from the proposition and the results in \cite{CrMo08}. For the second part, consider a jet $\tau$ satisfying $[\ii(\tau),\ii(\tau)]=0$. The anchor of the corresponding Lie algebroid bracket $[~,~]_{\tau}$ is given by $\rho=\pi^{\sharp}$, where $\pi\in\Gamma(\wedge^2TM)$ is the bivector field covered by $\tau$. By
\eqref{eq:mult:form} the identity map $\id:T^*M\to T^*M$ is a closed IM form if and only if 
\[[\al,\be]_{\tau}=\Lie_{\pi^{\sharp}(\al)}\mu(\be)-i_{\pi^{\sharp}(\be)}\d \mu(\al),\]
which is equivalent to $\ii(\tau)=i(\pi)=\ii(J^1\pi)$. Since $\ii$ is injective, the conclusion follows. 
\end{proof}


\subsection{Jets along a submanifold}

%

We can describe first order jets of Poisson structures as sections of a certain jet bundle satisfying a PDE. Namely, the restriction $J^1(\wedge^\bullet TM)|_S\to S$ has the subbundle of jets tangent to $S$, denoted by
\[ J^1(\wedge^\bullet TM: \wedge^{\bullet} TS):=\{\tau\in J^1(\wedge^\bullet TM)|_S: \pr_{\wedge^\bullet T_SM}(\tau)\in\wedge^\bullet TS\}.\]
This subbundle has an induced projection
%
\[\pr_{J^1(\wedge^{\bullet}TS)}:J^1(\wedge^\bullet TM: \wedge^\bullet TS)\longrightarrow J^1(\wedge^\bullet TS).\]
and we call a section $\tau\in\Gamma(J^1(\wedge^\bullet TM: \wedge^\bullet TS))$ \emph{holonomic} if its image under this projection is holonomic, i.e., if there $\gamma\in \X^{\bullet}(S)$ such that
\[\pr_{J^1(\wedge^{\bullet}TS)}(\tau)=J^1\gamma.\]
Since $S$ is closed and embedded, this is equivalent to the existence of a multivector field $\vartheta\in\X^\bullet_S(M)$ such that $J^1\vartheta|_S=\tau$ (then $\vartheta$ is automatically tangent to $S$, with $\vartheta|_{S}=\gamma$). Note that the section $\vartheta$ is determined only up elements in $I^2(S)\cdot \mathfrak{X}^{\bullet}(M)$. Therefore we have a canonical identification between the set of holonomic sections of $J^1(\wedge^\bullet TM: \wedge^\bullet TS)$ and the quotient $J^1_S\X^{\bullet}_S(M)= \X^{\bullet}_S(M)/(I^{2}_S\cdot \X^{\bullet}(M))$, which we will use with no further notice. 

In general, a Poisson submanifold $S\subset (M,\pi)$ gives rise to a Lie subalgebroid $T^*_SM\subset (T^*M,[~,~]_{\pi},\pi^{\sharp})$ of the tangent bundle. This Lie subalgebroid encodes precisely the first jet data of $J^1_S\pi\in J^1_{S}\Pi(M,S)$. To make this precise, we consider the following. First consider multiderivations of $T^*M$ with symbol tangent to $S$, 
\[\Der^{d}_S(T^*M):=\{D\in \Der^d(T^*M) : \sigma_D(\al_1,\ldots,\al_{d})\in TS,\ \forall\, \al_i\in T^*_SM\}.\]
We have the following relative version of Proposition \ref{prop:jets:derivations}.

\begin{proposition}\label{prop:jets:derivations:S}\mbox{}
\begin{enumerate}[(a)]
\item
$\Der^{\bullet}_S(T^*M)$ is closed under the commutator bracket, and we have a commutative diagram of graded Lie algebra homomorphisms
\[
\xymatrix{
\X^{\bullet+1}_S(M) \ar[r]^{i}\ar[d]^{J^1_S} & \Der^{\bullet}_S(T^*M)\ar[d]^{r_S}\\
J^1_S\X^{\bullet+1}_S(M)\ar[r]^{i_S} & \Der^{\bullet}(T^*_SM),
}
\]
where $i$ is the map from Proposition \ref{prop:jets:derivations} and
where $r_S$ is the operation of restricting a derivation of $T^*M$ tangent to $S$ to a derivation of $T^*_SM$.
\item The map $i_S$ induces an injective vector bundle map
\[ \ii_S: J^1(\wedge^{d +1} TM: \wedge^{d+1} TS)\hookrightarrow \DD^d(T^*_S M), \]
whose image consists of elements $D\in \DD^d(T^*_S M)$whose symbol is a multivector $\sigma_D\in\wedge^d TS\subset \wedge^{d-1}T_S M\otimes TS$.
\end{enumerate}
\end{proposition}

%
%

We have the following analogue of Corollary \ref{cor:jets:algebroids}. The proof is entirely similar, where one uses instead Proposition \ref{prop:jets:derivations:S}. In particular, this yields the main result of this section.

\begin{corollary}
\label{cor:jets:Poisson}
There is a 1-to-1 correspondence
\[ 
\left\{\txt{$\tau\in\Gamma( J^1(\wedge^{2} TM: \wedge^{2} TS))$ \\with $[\ii_S(\tau),\ii_S(\tau)]=0$ \\ \,}\right\}
\tilde{\longleftrightarrow}
\left\{\txt{Lie algebroid structures on $T^*_SM$\\ with anchor $\rho:T^*_S M\to TS$ s.t.:\\ $\langle \mu_S(\al),\rho(\be) \rangle=-\langle \mu_S(\be),\rho(\al) \rangle$\\  \,} \right\}.
\]
Under this correspondence, holonomic sections (i.e., first order jets of Poisson structures) correspond to Lie algebroids for which $\mu_S$ is a closed IM 2-form
\[ 
J^1_S\Pi(M,S)=
\left\{\txt{$\tau\in J^1_S\X^2_S(M)$ with \\$[\tau,\tau]=0$ \\ \,}\right\}
\tilde{\longleftrightarrow}
\left\{\txt{Lie algebroids $(T^*_SM,[\cdot,\cdot],\rho)$\\ with $\mu_S:T^*_SM\to T^*S$\\ a closed IM 2-form\,} \right\}.
\]
\end{corollary}

\section{Over-symplectic groupoids}
\label{sec:oversymplectic}

In this section we start by recalling the notion of over-symplectic groupoid and its properties. This is the type of groupoid one obtains by restricting a symplectic groupoid to an invariant submanifold, i.e., a (complete) Poisson submanifold of the underlying Poisson manifold. This is also the type of groupoid for which one can hope to find coisotropic embeddings. We show in this section that the infinitesimal data corresponding to an over-symplectic groupoid $(\G_S,\omega_S)$ is precisely a first order jet of a Poisson structure at $S$. Our conventions for Lie groupoids and Lie algebroids are the same as in the companion paper \cite{FM22} and in the monograph \cite{CFM21}.

\subsection{Over-symplectic groupoids}
Unless otherwise stated, all groupoids are assumed to be Hausdorff and target connected. The latter implies that the orbits of a groupoid coincide with the orbits of its Lie algebroid. We will denote the restriction of a Lie groupoid $\G\tto M$ to a submanifold $S\subset M$ by
\[ \G|_S:=\s^{-1}(S)\cap\t^{-1}(S). \]
This is always a groupoid $\G|_S\tto S$, but it may fail to be smooth, i.e., to be a Lie subgroupoid. However, if $S$ is saturated by orbits of $\G$ then 
\[ \G|_S=\s^{-1}(S)=\t^{-1}(S), \]
and $\G|_S\tto S$ is a target connected Lie subgroupoid of $\G\tto M$. 

\begin{proposition}
\label{prop:Poisson:over:symplectic}
Let $(\G,\omega)\tto (M,\pi)$ be a symplectic groupoid and $S\subset M$ a closed embedded Poisson submanifold. Then $\G|_S\tto S$ is a closed embedded coisotropic Lie subgroupoid of $\G\tto M$. The restriction of the symplectic form $\omega_S:=\omega|_{\G|_S}$ is a closed multiplicative form on $\G|_S$ satisfying
\[ \ker \omega_S\subset \ker\d\s_S\cap\ker\d \t_S,\]
where are $\s_S$ and $\t_S$ are the source and target maps of $\G|_S\tto S$. 
\end{proposition}

\begin{proof}
Since $S$ is a closed Poisson submanifold, it is a union of symplectic leaves, and since $\G$ is target connected, it follows that $S$ is a saturated submanifold for $\G$. Therefore, $\G|_S\tto S$ is a closed embedded Lie subgroupoid of $\G\tto M$. Moreover, $\ker\d\s=\ker\d\s_S\subset T\G|_S$ and $\ker\d\t=\ker\d\t_S\subset T\G|_S$. Since $\omega$ is multiplicative and closed, the restriction $\omega_S$ is also multiplicative and closed.

Now, recall that if a Poisson map is transverse to a coisotropic submanifold, then it pulls it back to a coisotropic submanifold. Since a Poisson submanifold is also coisotropic, we can apply this fact to the target map $\t:(\G,\omega)\to (M,\pi)$ to conclude that $\G|_S=\t^{-1}(S)\subset \G$ is a coisotropic submanifold.

In order to prove the statement concerning the kernel of $\omega_S$ observe that, since $\ker\d\s,\ker\d\t\subset T(\G|_S)$, for $v\in\ker\omega_S$ we have that
\[ \omega_S(v,w)=\omega(v,w)=0, \quad \forall w\in \ker\d\s +\ker\d\t. \]
Hence, we obtain that
\[ \ker \omega_S\subset  (\ker\d\t)^{\perp_\omega}\cap  (\ker\d\s)^{\perp_\omega}=\ker\d\s\cap\ker\d \t=\ker\d\s_S\cap\ker\d \t_S. \qedhere\] 
\end{proof}

The previous proposition shows that the restriction of a symplectic groupoid to a Poisson submanifold is an \emph{over-symplectic groupoid} in the sense of \cite{BCWZ04}.

\begin{definition}
\label{def:over-symplectic}
An \textbf{over-symplectic groupoid} is a groupoid $\G_S\tto S$ together with a closed multiplicative 2-form $\omega_S\in \Omega^2(\G_S)$ such that 
\[ \ker \omega_S\subset \ker\d\s_S\cap\ker\d \t_S. \] 
\end{definition}

We will need the following properties of over-symplectic groupoids.

\begin{proposition}[Proposition 4.5 from \cite{BCWZ04}]
\label{prop:over:symplectic}
Let $\G_S\tto S$ be a Lie groupoid and $\omega_S\in\Omega^2(\G_S)$ a closed multiplicative 2-form. The following are equivalent
\begin{enumerate}[(a)]
\item $(\G_S,\omega_S)$ is over-symplectic;
\item $(\ker\d \t_S)^{\perp_{\omega_S}}=\ker\d\s_S$;
\item $\mathrm{rank} (\omega_S|_x)=2\dim S$, for all $x\in S$;
\item $\ker\omega_S|_x\subset \ker\d_x\s_S\cap\ker\d_x \t_S$, for all $x\in S$;
\item There exists a Poisson structure on $S$, such that $\t_S: (\G_S,\omega_S)\to (S,\pi_S)$ is a forward Dirac map.
\end{enumerate}
Moreover, if these hold, $\omega_S$ has constant rank and the orbits of $\G_S$ coincide with the symplectic leaves of $(S,\pi_S)$.
\end{proposition}

We have the following converse to Proposition \ref{prop:Poisson:over:symplectic}.

\begin{proposition}
Let $i:(\G_S,\omega_S)\hookrightarrow (\G,\omega)$ be a coisotropic embedding of an over-symplectic groupoid in a symplectic groupoid. The corresponding embedding of units $(S,\pi_S)\hookrightarrow (M,\pi)$ is a Poisson map. If the embedding is closed then $S$ is a saturated submanifold of $M$ and $\G_S=\G|_S$.
\end{proposition}

\begin{proof}
Since $\G_S$ is an embedded subgroupoid of $\G$, the source/target fibers of $\G_S$ are embedded submanifolds of the source/target fibers of $\G$
\[ \s_S^{-1}(x)\subset \s^{-1}(x),\quad \t_S^{-1}(x)\subset \t^{-1}(x),\quad \forall\, x\in S. \]
We claim that they have the same dimension, so these are open inclusions. Indeed, 
\begin{align*}
\dim \s_S^{-1}(x)&=\dim \t_S^{-1}(x)=\dim \G_S-\dim S,\\
\dim \s^{-1}(x)&=\dim \t^{-1}(x)=\dim \G-\dim M=\frac{1}{2}\dim \G,\\
\dim \G&=\dim(\ker\omega_S)+\dim \G_S=2(\dim \G_S-\dim S),
\end{align*}
where in the second equation we used that $\G$ is a symplectic groupoid and in the last equation we used Proposition \ref{prop:over:symplectic} and the fact that the embedding is coisotropic. These equations imply our claim
\[ \dim \s_S^{-1}(x)=\dim \t_S^{-1}(x)=\dim \s^{-1}(x)=\dim \t^{-1}(x).\]

It follows that the orbits of $\G_S\tto S$ are open subsets of the orbits of $\G\tto M$. Therefore, $S$ is a Poisson submanifold with Poisson structure $\pi_S$. Thus we obtain the following commutative diagram
\[
\xymatrix{
(\G_S,\omega_S) \ar[r]^{i}\ar[d]^{\t_S} & (\G,\omega)\ar[d]^\t\\
(S,\pi_S)\ar[r]^i & (M,\pi) 
}
\]
where the horizontal arrows are backward Dirac and $\t$ forward Dirac. It follows by \cite[Lemma  3]{FrMa18}  applied to $\G|_{S}$ that also $\t_S$ is forward Dirac, i.e., the induced Poisson structure on the base of the over-symplectic groupoid $(\G_S,\omega_S)$ coincides with $\pi_S$.

Now if the embedding is closed, the source/target fibers of $\G_S$ are closed submanifolds of the source/target fibers of $\G$. Since they are connected, they coincide
\[ \s_S^{-1}(x)=\s^{-1}(x),\quad  \t_S^{-1}(x)=\t^{-1}(x), \quad \forall x\in S. \]
This implies that $S$ is saturated and that $\G_S=\G|_{S}$. 
\end{proof}

Note that for an over-symplectic groupoid $(\G_S,\omega_S)$ the distribution $\ker\omega_S$ is integrable. If $\ker\omega_S$ is a simple foliation then $\omega_S$ descends to a symplectic form $\underline{\omega_S}$ on the leaf space $\G_S/\ker\omega_S$, which inherits the structure of a symplectic groupoid
\[ (\G_S/\ker\omega_S,\underline{\omega_S})\tto S\]
integrating the Poisson manifold $(S,\pi_S)$. This explains the term ``over-symplectic''.

We can summarize this discussion by the following diagram
{\footnotesize
\begin{equation*}
\xymatrixrowsep{0.6cm} \xymatrixcolsep{0.2cm}
\xymatrix{\left\{\txt{symplectic groupoids\\ $(\G,\omega)\tto M$\\ with saturated submanifold $S$}\right\}\ar[rr]^*\txt{restriction\\ along $S$}\ar[dr]&&
 \left\{\txt{over-symplectic groupoids\\$(\G_S,\omega_S)\tto S$\\ }\right\}\ar[dl] \ar@/^/@{-->}[ll]^*\txt{coisotropic\\embedding}\\
& \left\{\txt{symplectic groupoids\\
$(\G_S/\ker\omega_S,\underline{\omega_S})\tto S$\\ }\right\}&}
\end{equation*}
}

We will see in Subsection \ref{ex:counterex:coistropic} an example of an over-symplectic groupoid which does not admit a coisotropic embedding into a symplectic groupoid. 

\subsection{Over-symplectic groupoids and first order jets}
\label{sec:grpds:jets}

We will show now that over-symplectic groupoids are the global objects corresponding to first order jets of Poisson structures along Poisson submanifolds. In other words, the previous diagram is the groupoid version of the diagram \eqref{eq:diagram:jets}.

Let $(\G_S,\omega_S)\tto S$ be an over-symplectic groupoid. If $A_S\to S$ denotes its Lie algebroid, then $\omega_S$ induces a closed IM 2-form $\mu:A_S\to T^*S$ (see Appendix \ref{appendix}). By Proposition \ref{prop:over:symplectic}, this bundle map is surjective. In fact, the following holds. 

\begin{proposition}\label{prop:1:1:over:sympl:surjective:IM}
If $\G_S\tto S$ is a target 1-connected Lie groupoid with Lie algebroid $A_S\to S$, there is a 1-to-1 correspondence
\[ 
\left\{\txt{over-symplectic structures\\ ${\omega_S}\in\Omega^2(\G_S)$ \,}\right\}
\tilde{\longleftrightarrow}
\left\{\txt{surjective IM closed 2-forms\\ $\mu:A_S\to T^*S$ \,} \right\}
\]
%
\end{proposition}
\begin{proof}
	According to our conventions, $A_S=\ker\d\t|_S$. Then, as explained in Section \ref{appendix}, the 1-to-1 correspondence between multiplicative forms and IM forms associates to a \emph{closed} multiplicative form ${\omega_S}\in\Omega^2_\mult(\G_S)$ the \emph{closed} IM form $\mu:A_S\to T^*S$ given by composing the maps
	\[ \xymatrix{\ker\d\t|_S \ar[r]^{\phantom{1}{\omega_S}^\flat} & T^*_S\G_S\ar[r]^{\pr} & T^*S}, \]
	where $\pr$ is the pullback along the unit map $S\hookrightarrow \G_S$. 
By Proposition \ref{prop:over:symplectic}, ${\omega_S}\in\Omega^2_\mult(\G_S)$ is an over-symplectic structure if and only if 
\begin{equation}\label{eq:perpendicular:condition}
 (\ker\d \t|_S)^{\perp_{{\omega_S}}}=\ker\d\s|_S.
 \end{equation} 
The inclusion $\ker\d\s\subset (\ker\d \t)^{\perp_{{\omega_S}}}$ holds for any multiplicative 2-form, and using the decomposition $T_S\G_S=TS\oplus \ker\d\s|_S$ we see that \eqref{eq:perpendicular:condition} is equivalent to 
\[ (\ker\d \t|_S)^{\perp_{{\omega_S}}}\cap TS=\{0\}.\]
This condition is equivalent to $\mu$ being surjective. 
%
%
%
	\end{proof}

Given a pair $(A_S,\mu)$, where $\mu:A_S\to T^*S$ a surjective closed IM form, we can choose a manifold and an isomorphism of vector bundles $A_S\simeq T^*_SM$ such that $\mu$ becomes the canonical projection $\mu_S:T^*_SM\to T^*S$: for example, we can choose $M:=(\ker\mu)^*$. So by Corollary \ref{cor:jets:Poisson}, we see that the infinitesimal data $(A_S,\mu)$ codifying an over-symplectic groupoid is the same thing as a first order jet of a Poisson structure $(T^*_SM,\mu_S)$. This justifies one of the entries in the dictionary from the Introduction.

Now if $(\G_S,\omega_S)\tto S$ is an over-symplectic groupoid integrating the pair $(A_S,\mu)$ and it admits a coisotropic embedding $i:(\G_S,\omega_S)\hookrightarrow (\G,\omega)$ into a symplectic groupoid $(\G,\omega)\tto(M,\pi)$, then we obtain
\begin{itemize}
\item a Poisson embedding $(S,\pi_S)\hookrightarrow (M,\pi)$;
\item an isomorphism $(A_S,\mu)\simeq (T^*_SM,\mu_S)$;
\item $(M,\pi)$ is a solution of the realization problem for $\tau=(T^*_SM,\mu_S)$:
\[ J^1_S\pi=\tau. \]
\end{itemize}
Next we will look at some examples, before we discuss the problem of existence of coisotropic groupoid embeddings.

\section{Examples of over-symplectic groupoids}
\label{sec:examples:over:sympl}

We give a few examples of over-symplectic groupoids which will be useful later.

\subsection{Products}\label{example:os:products}
Let $(\Sigma,\omega_{\Sigma})\tto S$ be a symplectic groupoid and $G$ be a Lie group. We obtain an over-symplectic groupoid by forming the product
\[ \G_S:=\Sigma\times G, \quad \omega_S:=\pr_\Sigma^*\omega_{\Sigma}. \]

\subsection{Transitive over-symplectic groupoids}
\label{ex:transitive:over-symplectic}
By Proposition \ref{prop:over:symplectic}, an over-symplectic groupoid $(\G_S,\omega_S)\tto S$ is transitive if and only if the Poisson structure induced on its base is non-degenerate: $\pi_S=\omega^{-1}$, for a symplectic form $\omega\in\Omega^2(S)$. Assuming this to be the case, and denoting the isotropy group at $x\in S$ by $G=\G_{S,x}$, we obtain a principal $G$-bundle $\t:\s^{-1}(x)\to S$ with a symplectic form $\omega$ on the base. Applying again Proposition \ref{prop:over:symplectic} (e), one sees that $\t^*\omega=\omega_S|_{\s^{-1}(x)}$.


Conversely, if we are given a principal $G$-bundle $p:P\to S$ with a symplectic form $\omega\in\Omega^2(S)$, then the gauge construction produces a transitive groupoid
\[ \G_S=P\times_G P\tto S, \]
together with the closed, multiplicative 2-form
\[\omega_S=(p\circ\pr_2)^*\omega-(p\circ\pr_1)^*\omega\in\Omega^2(\G_S).\]
This makes $(\G_S,\omega_S)$ into a transitive over-symplectic groupoid.

These constructions are inverse to each other, up to isomorphism. So transitive over-symplectic groupoids correspond to principal bundles with a symplectic form on the base.

At the infinitesimal level, if $A_S$ is the Lie algebroid of $\G_S$, the corresponding closed IM-2 form is given by $\mu=\omega^{\flat}\circ \rho$.

\subsection{Over-symplectic bundles of Lie groups}\label{example:os:bundlegroups}
At the other extreme, Proposition \ref{prop:over:symplectic} implies that an over-symplectic groupoid $(\G_S,\omega_S)\tto S$ is a bundle of Lie groups if and only if the Poisson structure on its base vanishes: $\pi_S=0$. These already provide non-trivial examples even in the case of a trivial group bundle, i.e., if $\G_S$ is the product
\[\G_S\simeq G\times S\to S,\]
for a fixed Lie group $G$.
Since any multiplicative 2-form on a Lie group is trivial, a multiplicative 2-form $\omega_S\in \Omega^2(G\times S)$ has the form
\[\omega_{S}=\eta+ c\]
where (see \cite[Section 6.4]{BCWZ04})
\begin{itemize}
    \item the mixed component $\eta_{(g,x)}\in T^*_gG\otimes T^*_xS$ satisfies that,
for any $w\in T_xS$, $i_w\eta$ is a multiplicative 1-form on $G$, which is the same as a bi-invariant 1-form on $G$, or a $G$-invariant element of $\gg^*$;
\item the horizontal component is a group homomorphism \[G\ni g\mapsto c(g)\in \Omega^2(S).\]
\end{itemize}
Then, the condition that $\omega_S$ is closed translates to:
\[\frac{\d }{\d t}\Big|_{t=0}c(\exp(tv))=
\d (i_v \eta)|_{\{e\}\times S}\in \Omega^2(S), \quad v\in \gg.\]

At the infinitesimal level we have the trivial Lie algebra bundle $A_S=\gg\times S\to S$, and the induced closed IM 2-form is
\[\mu:\gg\times S\to T^*S,\ \langle w,\mu(v)\rangle =i_wi_v\eta|_{g=e}, \quad v\in \gg,\, w\in TS.\]
The IM-condition says that $\mu:\gg \to \Omega^1(S)$ satisfies
\begin{equation}
    \label{eq:IM:condition:group:bundle}
    \mu([v,w])=0.
\end{equation}

Conversely, the reconstruction of $\omega_S$ from $\mu$ works as follows. The component $\eta$ can be always defined
\[\eta(v,w)=i_{w}\mu(\d L_{g^{-1}}v),\ \quad v\in T_gG,\, w\in T_xS.\]
The component $c:G\to \Omega^2(S)$ is a group homomorphism integrating the Lie algebra homomorphism $\d\circ \mu:\gg \to \Omega^2(S)$, which exists e.g.\ if $G$ is simply connected. Moreover, if $\mu$ itself integrates to a group homomorphism $\lambda:G\to \Omega^1(S)$, then
\[\omega_S=\d \tilde{\lambda},\quad \textrm{where}\quad \tilde{\lambda}_{(g,x)}=\pr_S^*(\lambda(g)).\]

When $G$ is a compact Lie group, the group homomorphism $c:G\to \Omega^2(S)$ must be trivial, and so its Lie algebra version $\d\circ \mu:\gg \to \Omega^2(S)$ vanishes. The IM condition \eqref{eq:IM:condition:group:bundle} then means we can view $\mu:\gg\to\Omega^1(S)$ as an 1-form $\mu\in\Omega^1(S,(\gg^*)^G)$. When this form is exact (e.g., if $H^1(S)=0$), there is $f\in C^{\infty}(S)\otimes (\gg^*)^G$ such that 
\[ \mu(v)=\langle \d f,v\rangle,\quad v\in\gg. \]
The nondegeneracy  of $\mu$ is equivalent to $f$ being an immersion, and we obtain the following expression for $\omega_S$ in terms of $f$
\[\omega_S((w_1,v_1),( w_2, v_2))= \langle \d f(v_1), \d L_g^{-1} (w_2)\rangle- 
 \langle \d f(v_2), \d L_g^{-1} (w_1)\rangle,\] for all $(w_i,v_i)\in T_gG\oplus T_xS$.

\subsection{Non-existence of a coisotropic embedding}
\label{ex:counterex:coistropic}
The jet $\tau\in J^1_S\Pi(M,S)$ in Example \ref{ex:non:holonomic} satisfies $\pi_S\equiv 0$. The corresponding pair $(T_S^*M,\mu_S)$ has zero anchor, i.e., it is a bundle of Lie algebras. This Lie algebroid integrates to a bundle of 1-connected Lie groups $\G_S\tto S$ which is Hausdorff. An explicit description is the following. We use coordinates $(x,y)$ in $S=\R^2$ and  $\s=\t:\G_S\to S$ is the trivial bundle $\pr:\R^5=\R^2\times\R^3\to \R^2$, where on the fiber we use coordinates $(u,v,w)$. The multiplication $m:\G_S\times_S \G_S\to \G_S$ of the group over $(x,y)\in S$ is given by
\[ (u_1,v_1,w_1)\cdot(u_2,v_2,w_2)= (u_1+u_2,v_1+e^{xu_1}v_2,w_1+w_2+\frac{e^{xu_1}-1}{x}v_2). \]
Note that this is not a trivial group bundle.

By Proposition \ref{prop:1:1:over:sympl:surjective:IM}, there is
a multiplicative closed 2-form $\omega_S$ on $\G_S$. The resulting over-symplectic groupoid $(\G_S,\omega_S)$ does not admit a coisotropic embedding, for otherwise $\tau$ would be in the image of $J^1_S$, contradicting Example \ref{ex:non:holonomic}.

It is not hard to show that the over-symplectic structure is
\[ \omega_S=\d x\wedge \d u+\d y\wedge\d v+w\, \d x\wedge \d y-x\, \d y\wedge \d w,\]
by proving that this form is closed, multiplicative and induces the IM form \eqref{eq:IM:form:non-split}.

\subsection{Principal type}\label{example:os:princ:type}
Let $(\Sigma,\omega_{\Sigma})\tto S$ be a symplectic groupoid, and let $\H\tto S$ be a transitive Lie groupoid. Consider their product over $S$
\[\G_S:=\{(\sigma,h)\in \Sigma\times \H\, :\, \s_{\Sigma}(\sigma)=\s_{\H}(h),\, \t_{\Sigma}(\sigma)=\t_{\H}(h)\}\tto S,\]
with groupoid structure such that the inclusion $\G_S\hookrightarrow \Sigma\times \H$ is a groupoid map. Smoothness of $\G_S$ follows because, as for any transitive Lie groupoid, the anchor map $(\t_{\H},\s_{\H}):\H\to S\times S$ is a submersion. Then $\G_S$ is an over-symplectic groupoid with multiplicative 2-form
\[\omega_S:=\pr_{\Sigma}^*(\omega_{\Sigma})\in \Omega^2(\G_S),\]
which we will call the \textbf{over-symplectic groupoid of principal type} associated to $(\Sigma,\omega)$ and $\H$. The terminology comes from the fact that, if one fixes a base point $x_0\in S$, the transitive groupoid $\H$ can be identified with the gauge groupoid $\H\simeq P\times_{G}P$ associated to the principal bundle $P:=\s_{\H}^{-1}(x_0)$, with projection $\t_{\H}:P \to S$ and structure group $G:=\s_{\H}^{-1}(x_0)\cap \t_{\H}^{-1}(x_0)$.

The Poisson structure $\pi_S$ induced on $S$ by the over-symplectic groupoid $(\G_S,\omega_S)$ is the same as the one induced by the symplectic groupoid $(\Sigma,\omega_{\Sigma})$. Some of the previous examples fit into this setting
\begin{itemize}
\item When $P$ is the trivial principal bundle, $P\simeq G\times S$, or equivalently $\H=(S\times S)\times G$, then $\G_S$ is the product $\G_S=\Sigma\times G$ of Example \ref{example:os:products}.
\item For a symplectic pair groupoid $(\Sigma,\omega_{\Sigma})=(S,\omega_S)\times (S,-\omega_S)$, we have $\G_S\simeq \H$ and we obtain the transitive over-symplectic groupoid of Example \ref{ex:transitive:over-symplectic}.
\item When $\Sigma$ is the symplectic groupoid $(T^*S, \omega_{\can})$, then we obtain the bundle of Lie groups $\G_S=T^*S\times_S \K$, with $\omega_S=\pr_{T^*S}^*\omega_{\can}$, where $\K$ is the bundle of isotropy groups of $\H$, or equivalently, the associated bundle $P\times_GG$, where $G$ acts on $G$ by conjugation. 
\end{itemize}

\subsection{Over-symplectic groupoid inducing a non-integrable Poisson structure}
\label{example:Lie-Poisson:sphere}
Let $H$ be a compact semi-simple Lie group, and consider the symplectic groupoid $(T^*H,\omega_{\mathrm{can}})\tto \hh^*$ integrating the linear Poisson structure on $\hh^*$. Let $S\subset \hh^*$ be the unit sphere around the origin with respect to an $H$-invariant inner product. Then $S$ is a Poisson submanifold, and therefore we obtain the over-symplectic groupoid
\[\G_S:=T^*H|_S\simeq H\ltimes S\tto S,\quad \textrm{with}\quad \omega_S:=\omega_{\mathrm{can}}|_{\G_S}\]
For $\hh\not\simeq \mathfrak{so}(3,\mathbb{R})$, the Poisson manifold $(S,\pi_S)$ is not integrable, and so $\G_S$ is not of principal type. 

\section{Groupoid coisotropic embeddings}
\label{sec:grpd}

In this section, given a over-symplectic groupoid $(\G_S,\omega_S)$ admitting a \emph{multiplicative Ehresmann connection} we construct a local model: it is a symplectic groupoid $(\G_0,\omega_0)$ where $(\G_S,\omega_S)$ sits as a coisotropic subgroupoid. We will show that the local model is essentially unique and we discuss a normal form theorem. We will see that the local model can be thought of as a multiplicative (or groupoid) version of the classical Gotay coisotropic embedding theorem.

\subsection{The groupoid local model}
\label{sec:local:model:grpds}
Let $(\G_S,\omega_S)$ be an over-symplectic groupoid. We start by observing that $\G_S$ acts on the vector bundle
\[ \ka:=(\ker\omega_S)|_S\subset A_S,\]
where $A_S$ is the Lie algebroid of $\G_S$. By Proposition \ref{prop:over:symplectic}, $\ka\subset \ker\rho_S$. Then, an arrow $g\in \G_S$ from $x:=\s_S(g)$ to $y:=\t_S(g)$ acts by conjugation on the isotropies:
\begin{equation}
\label{eq:representation:grpd} 
g:\ker\rho_S|_{x}\to \ker\rho_S|_{y}, \quad g\cdot \al=\frac{\d}{\d t}\Big|_{t=0} g\, \exp_{x}(t\alpha)\, g^{-1},
\end{equation}
where
$\exp_{x}:\ker\rho_S|_x\to (\G_S)_{x}$ is the Lie group exponential. Since $\omega_S$ is multiplicative, this action preserves $\ka\subset\ker\rho_S$. So $\ka$ is a $\G_S$-representation. 

The dual representation allows us to define the action groupoid
\[ \G_S\ltimes\ka^*\tto \ka^*. \]
We would like to construct a multiplicative symplectic structure on this groupoid (which is not always possible; see Subsection \ref{ex:counterex:coistropic}). It will be convenient to introduce the following notion, which originates from the theory of multiplicative connections developed in \cite{FM22}.

\begin{definition}
\label{def:partially:split:grpd}
A {\bf multiplicative connection 1-form} on the over-symplectic groupoid $(\G_S,\omega_S)$ is a $\ka$-valued, multiplicative, 1-form $\al\in\Omega^1(\G_S;\ka)$ satisfying
\[ \al(\xi^L)=\xi, \]
where $\xi^L\in\X(\G_S)$ is the left-invariant vector field determined by $\xi\in\Gamma(\ka)$. 

An over-symplectic groupoid $(\G_S,\omega_S)$ which admits a multiplicative connection 1-form is called \textbf{partially split} (in the next subsection, we will motivate the choice of this terminology).
\end{definition}

Now, given a multiplicative connection 1-form $\alpha\in\Omega^1_{\mult}(\G_S,\ka)$ we have an associated (ordinary) 1-form $\langle \al,\cdot\rangle\in\Omega^1(\G_S\ltimes\ka^*)$, defined by
\[(v,z)\mapsto \langle \al_g(v),\eta\rangle,\quad \textrm{for}\quad (v,z)\in T_g\G_S\times_{TS}T_\eta\ka^*.\] 


\begin{proposition}
\label{prop:local:model:grpd}
Given an over-symplectic groupoid $(\G_S,\omega_S)$ and a multiplicative connection 1-form $\alpha\in\Omega^1_{\mult}(\G_S,\ka)$ the closed 2-form
\begin{equation}
    \label{eq:symplectic:form:local:model}
    \omega_0=\pr^*_{\G_S}\omega_S+\d\langle \al,\cdot\rangle
    \in\Omega^2(\G_S\ltimes\ka^*)
\end{equation}
is multiplicative. Moreover, there is an open groupoid neighborhood 
\[(\G_S\tto S)\hookrightarrow (\G_0\tto M_0) \subset (\G_S\ltimes\ka^*\tto\ka^*)\]
on which $\omega_0$ is non-degenerate and the first map is a coisotropic embedding
\[i:(\G_S,\omega_S)\hookrightarrow  (\G_0,\omega_0).\]
\end{proposition}

\begin{proof}
Since $\al$ is multiplicative, a direct computation shows that $\langle \al,\cdot\rangle$ is a multiplicative 1-form on $\G_S\ltimes\ka^*$, and so is its differential. Since the projection $\pr_{\G_S}:\G_S\ltimes\ka^*\to \G_S$ is a groupoid morphism, it pulls back multiplicative forms to multiplicative forms, and we conclude that $\omega_0$ is multiplicative.

That $\omega_0$ is non-degenerate along the zero-section $\G_S\hookrightarrow \G_S\ltimes \ka^*$ and that $i$ is a coisotropic embedding can be checked directly using the condition $\al(\xi)=\xi^L$, and that $(\G_S,\omega_S)$ is over-symplectic. Alternatively, this follow from Gotay's coisotropic embedding theorem -- see the next section.

Finally, we need to find an open set $\G_0$ where $\omega_0$ is non-degenerate and which is a subgroupoid. Consider the open subset $M_0\subset \ka^*$ consisting of points $x$ such that $\omega_0|_x$ is non-degenerate. By Proposition \ref{prop:over:symplectic} (c), $\G:=(\G_S\ltimes \ka^*)|_{M_0}$ is over-symplectic, and by the last part of that proposition, $\omega_0$ is non-degenerate on $\G$. Thus $(\G,\omega_0)$ is indeed a symplectic groupoid. In order to comply with our convention that Lie groupoids are target-connected, we let $\G_0$ be the open subgroupoid of $\G$ consisting of the connected components of the identities.
\end{proof}

\begin{definition}
\label{def:grpd:local:model}
Given a partially split over-symplectic groupoid $(\G_S,\omega_S)$, with a multiplicative connection 1-form $\alpha$, the symplectic groupoid $(\G_0,\omega_0)$ constructed in the previous proposition is called a \textbf{groupoid local model} for $(\G_S,\omega_S)$.
\end{definition}

We will see in the next subsection alternative expressions for $\omega_0$ and we will show that for any two multiplicative connection 1-forms the corresponding groupoid local models are isomorphic around $\G_S$.

\begin{remark}
In general, the subgroupoid $\G_0\tto M_0$ of $\G_S\ltimes\ka^*\tto\ka^*$ cannot be taken to be a full groupoid neighborhood, i.e., of the form $\G_S\ltimes M_0 \tto M_0$, for some $\G_S$-invariant open set $M_0\subset \ka^*$ containing the zero section. However, this is always possible if $\G_S\tto S$ is a proper groupoid (see \cite[Lemma 5.3]{dHFe18}).
\end{remark}

\subsection{The partially split condition}

Over-symplectic groupoids are not always partially split. The natural setup for this notion is the theory of \emph{multiplicative Ehresmann connections} developed in \cite{FM22}. We now recall this notion briefly and give alternative descriptions. In the case of over-symplectic groupoids, this will lead to alternative descriptions of the symplectic form on the local model. 
\medskip

Given a Lie groupoid $\G_S\tto S$ with Lie algebroid $A_S$, we call a vector subbundle $\ka\subset A_S$ a \textbf{bundle of ideals} of $\G_S$ if $\ka\subset \ker\rho_S$ and $\ka$ is invariant under the $\G_S$-action by conjugation \eqref{eq:representation:grpd}. The fact that $\ka$ is invariant under conjugation implies that 
\[
 \al\in\Gamma(A_S),\  \xi\in\Gamma(\ka) \quad \Longrightarrow\quad [\al,\xi]\in\Gamma(\ka).
 \]
so this justifies the use of the term ``bundle of ideals". 
\bigskip

We have the following equivalent descriptions of the notion of ``connection" for a bundle of ideals $\ka \subset A_S$ for $\G_S$ (for details and terminology see \cite[Section 2.3]{FM22}).

\begin{enumerate}[(i)]
\item A {\bf multiplicative connection 1-form} for $\ka$ is a multiplicative 1-form $\alpha\in\Omega^1_\mult(\G_S,\ka)$ satisfying
    \[ \al(\xi^L)=\xi \quad (\xi\in\Gamma(\ka)). \]
    
By associating to a multiplicative connection 1-form $\alpha$ the distribution 
\[E:=\ker\alpha\subset T\G_S,\]
we obtain the following equivalent notion.
\item A {\bf multiplicative Ehresmann connection} for $\ka$ is a wide subgroupoid $E$ of $T\G_S\tto TS$ such that
\[ T\G_S=E\oplus K,\]
where $K\subset T\G_S$ is obtained by spreading $\ka$ using (left or right) translations
\[K_g:=\dd L_g(\ka_{\s(g)})=\dd R_g(\ka_{\t(g)})\subset \ker\dd \s\cap \ker \dd \t.\]

Note that $K\tto 0_S$ is a subgroupoid of the tangent groupoid $T\G_S\tto TS$, which is canonically isomorphic to the semi-direct product $\G_S\times_S\ka\tto S$, 
via 
\[ \G_S\times_S\ka\diffto K, \quad (g,v)\mapsto \d L_g(v).\]
Then a multiplicative Ehresmann connection gives a splitting of the inclusion of VB groupoids
\[
\vcenter{
\xymatrix@R=10pt{
{}\save[]+<-30pt,0cm>*\txt{$\G_S\times_S\ka\simeq$}\restore K\, \ar@<0.15pc>[dd] \ar@<-0.15pc>[dd] \ar@{^{(}->}@<-0.10pc>[rr]  \ar[dr]&  &  T\G_S  \ar@<0.15pc>[dd] \ar@<-0.15pc>[dd] \ar[dl] \ar@<-0.30pc>@/_/@{-->}[ll]\\
 & 
\G_S \ar@<0.15pc>[dd] \ar@<-0.15pc>[dd] & \\
0_S\, \ar[dr] \ar@{^{(}-}[r] &\ar[r]  & TS \ar[dl]\\
 & S}
 }
\]
Under the duality operation in the category of VB groupoids \cite{BCdH16}, we obtain the VB groupoids $T^*\G_S\tto A_S^*$ and $K^*\tto \ka^*$, and the dual VB groupoid map
\[
\vcenter{
\xymatrix@R=10pt{
{}\save[]+<-35pt,0cm>*\txt{$\G_S\ltimes\ka^*\simeq$}\restore K^*\, \ar@<0.15pc>[dd] \ar@<-0.15pc>[dd] \ar[dr] \ar@/_/@{-->}[rr]&  &  T^*\G_S\ar@{->>}[ll] 
\ar@<0.15pc>[dd] \ar@<-0.15pc>[dd] \ar[dl] \\
 &  \G_S  \ar@<0.15pc>[dd] \ar@<-0.15pc>[dd]  \\
\ka^* \ar[dr] & \ar@{->>}[l] & A^*_S \ar[dl] \ar@{-}[l]\\
 & S
}}
\]
The groupoid $K^*\tto \ka^*$ is isomorphic to the action groupoid $\G_S\ltimes\ka^*\tto \ka^*$ of the dual action of $\G_S$ on $\ka^*$. This leads to the following equivalent notion.
\item A {\bf partial splitting} is a VB groupoid morphism $\Theta:K^* \to T^*\G_S$ that splits the projection $p:T^*\G_S\to  K^*$
    \[ p\circ\, \Theta=\id.\]
The corresponding Ehresmann connection is given by
    \[ E=(\im\Theta)^0. \]
Another equivalent notion is the following.
\item A \textbf{linear}, closed, multiplicative, 2-form $\omega^\lin\in\Omega^2_\mult(K^*)$ that restricts to the canonical symplectic form on $\ka\times_S\ka^*\subset T_S K^*$
    \[ \omega^\lin((v_1,\xi_1),(v_2,\xi_2)) = \xi_2(v_1)-\xi_1(v_2), \quad \text{if} \quad(v_k,\xi_k)\in\ka\times_S\ka^*,\]
    where linear means that
    \[m_t^*(\omega^{\lin})=t\omega^{\lin},\quad (t>0),\]
    where $m_t:K^*\to K^*$ is fiberwise multiplication by $t>0$. Using the partial splitting or the connection 1-form, the linear form is given by
    \begin{equation}\label{eq:relation:omega:theta}
    \omega^{\lin}=\Theta^*\omega_{\can}=\d \langle \al,\cdot\rangle.
    \end{equation}
\end{enumerate}
A bundle of ideals admitting such a structure is called \textbf{partially split}. 

\begin{remark}
    There are several reasons for our choice of the terminology ``partial splitting''. The main reason, as explained in the companion paper \cite{FM22}, is the following. Multiplicative Ehresmann connections appeared first as Ehresmann connections for a surjective, submersive, groupoid morphism $\Phi:\G\to\H$ covering the identity. Such a morphism determines a short exact sequence of groupoids:
    \[ \xymatrix{1\ar[r] & \K\ar[r] & \G\ar[r]^\Phi & \H\ar[r] &1 }. \]
    The existence of an Ehresmann connection for $\Phi$, which is equivalent to the bundle of ideals $\ka=\textrm{Lie}(\K)$ being partially split, is \emph{weaker} than the existence of a splitting for this short exact sequence. Moreover, a partial splitting can be thought of as an infinitesimal version of a splitting of this sequence. We refer to \cite{FM22} for more details.
    The designation ``partial splitting" also allows us to distinguish these from other splittings that appear in the sequel. 
\end{remark}

\medskip

Let us return to our discussion of over-symplectic groupoids and their local models. For an over-symplectic groupoid $(\G_S,\omega_S)$, the relevant bundle of ideals is $\ka=(\ker\omega_S)|_S$ and we have $K=\ker\omega_S$. In particular, note that
\[
\G_S\ltimes\ka^*\simeq (\ker\omega_S)^*. 
\]
The symplectic form $\omega_0$ on the local model $\G_0$ can be written as follows:
\begin{itemize}
    \item using a multiplicative connection 1-form $\al\in\Omega^1_{\mult}(\G_S,\ka)$ as in the definition
    \[ \omega_0=\pr_{\G_S}^*\omega_S+\d \langle\al,\cdot\rangle; \]
    \item using a partial splitting $\Theta:\G_S\ltimes\ka^*\to T^*\G_S$
     \[ \omega_0=\pr_{\G_S}^*\omega_S+\Theta^*\omega_\can; \]
    \item using a linear, closed, multiplicative, 2-forms $\omega^\lin\in\Omega^2_\mult(\G_S\ltimes\ka^*)$
    \[ \omega_0=\pr_{\G_S}^*\omega_S+\omega^\lin.\]
\end{itemize}

The last expression shows that $\omega_0$ is the sum of a constant multiplicative form
$\pr^*\omega_S$ and a linear multiplicative form $\omega^\lin$.


\medskip

We can also explain the origins of the local model. It is a multiplicative version of the well-known Gotay's coisotropic embedding \cite{Gotay82}. In the classical case, starting with a pair $(C,\omega_C)$, where $\omega_C$ is a closed 2-form of constant rank, one constructs a symplectic manifold $(X_0,\omega_0)$ and a coisotropic embedding $(C,\omega_C)\hookrightarrow (X_0,\omega_0)$ as follows. A choice of complementary subbundle
\[
    TC=\ker\omega_C\oplus E,
\]
determines an embedding
\[ 
\Theta:(\ker\omega_C)^*\hookrightarrow T^*C.
\]
The form 
\[ 
\omega_0:=\pr_C^*\omega_C+\Theta^*\omega_\can \in\Omega^2((\ker\omega_C)^*), \]
is non-degenerate along the zero section $C\simeq 0_C$ and hence defines a symplectic form in a neighborhood $X_0\subset(\ker\omega_C)^*$ of $C$. Moreover, the zero section $C$ becomes a coisotropic submanifold of $(X_0,\omega_0)$ with induced 2-form $\omega_C$. Gotay's result says that any coisotropic embedding of $(C,\omega_C)$ in a symplectic manifold is locally isomorphic to the local model $(X_0,\omega_0)$. 

The local model we constructed is the multiplicative version of Gotay's model.



\medskip

The construction in Proposition \ref{prop:local:model:grpd} can be viewed in a slightly different way. For any over-symplectic groupoid $(\G_S,\omega_S)$, we have an associated symplectic groupoid
\begin{equation*}
(T^*\G_S,\omega_{\can}+\pr^*\omega_S) \tto A_S^*.
\end{equation*}

\begin{corollary}\label{corollary:embed:symplectic:groupoid}
For any splitting $\Theta$, the local model is a symplectic subgroupoid
\[\Theta:(\G_0,\omega_0)\hookrightarrow (T^*\G_S,\omega_{\can}+\pr^*\omega_S).\]
\end{corollary}

In Subsection \ref{sec:Xu:submanifolds}, we will discuss this construction at the level of Lie algebroids, and then we will describe explicitly the corresponding  Poisson structure on $A_S^*$. 

\subsection{Existence of multiplicative Ehresmann connections}

The following result gives some geometric meaning to the existence of a partially split condition: it shows that a multiplicative Ehresmann connection exists if and only if the groupoid $(\ker\omega_S)^*=\G_S\ltimes\ka^*$ can be made into a symplectic groupoid around $\G_S$ such that the symplectic form restricts to $\omega_S$. Therefore the partially split condition is optimal for the existence of a groupoid local model on $\G_S\ltimes \ka^*$.

\begin{proposition}
\label{prop:partially:split:grpd:optimal} An over-symplectic groupoid $(\G_S,\omega_S)$ is partially split if and only if there exists a groupoid neighborhood
\[(\G_S\tto S)\hookrightarrow (\G_0\tto M_0) \subset (\G_S\ltimes \ka^*\tto \ka^*),\]
together with multiplicative symplectic structures $\omega_0\in \Omega^2(\G_0)$ whose pullback along the zero section $i:\G_S\hookrightarrow  \G_S\ltimes \ka^*$ is $\omega_S$,
\[ i^*\omega_0=\omega_S. \]
\end{proposition}


\begin{proof}
Proposition \ref{prop:local:model:grpd} gives one implication.

For the converse, let $(\G_0,\omega_0)$ be as in the statement. That $\G_S$ is coisotropic follows from Proposition \ref{prop:Poisson:over:symplectic}. The restriction of the tangent groupoid $T(\G_S\ltimes \ka^*)\tto T\ka^*$ to the invariant submanifold $T_S\ka^*\simeq TS\oplus \ka^*$ is the action groupoid: 
$T\G_S\ltimes \ka^*\tto TS\oplus \ka^*$. This has the subgroupoids \[K:=(T\G_S)^{\perp_{\omega_0}}\tto S, \quad T\G_S\tto TS\quad \textrm{and}
\quad \G_S\ltimes \ka^*\tto \ka^*,\]
where the last is an embedded normal bundle for $\G_S$. We obtain that
\[E:=(\G_S\ltimes \ka^*)^{\perp_{\omega_0}}\cap T\G_S\]
is a vector bundle complementary to $K$ in $T\G_S$. Multiplicativity of $\omega_0$ implies that $E$ is subgroupoid. Hence $(\G_S,\omega_S)$ is partially split. 
\end{proof}

\begin{example}
\label{ex:non:split}
Consider the over-symplectic groupoid $(\G_S,\omega_S)$ from Subsection \ref{ex:counterex:coistropic}.
The action of $\G_S$ on the line bundle 
\[\ka=(\ker\omega_S)|_S=\langle x\pd{v}+\pd{w}\rangle |_S\]
is given on a section $s\in \Gamma(\ka)$ by
\[ (x,y,u,v,w)\cdot s{(x,y)}=e^{xu}s{(x,y)}. \]
Therefore, the groupoid $(\ker\omega_S)^*=\G_S\ltimes\ka^*$ has one dimensional orbits, and so it cannot be a symplectic groupoid. The previous proposition implies that $(\G_S,\omega_S)$ is not partially split. This also follows from the fact that the corresponding first jet cannot be realized as a Poisson structure. 
\end{example}

Another type of obstructions to the partially split condition, different from the previous example, derive from the general theory in \cite{FM22} of multiplicative connections. In particular, the results from \cite[Section 2.3]{FM22} when applied to over-symplectic groupoids, give the following

\begin{proposition}
\label{prop:corollary:partial:splittings}
Let $(\G_S,\omega_S)$ be partially split over-symplectic groupoid and let $\ka=(\ker\omega_S)|_S$. A choice of partial splitting $\Theta:(\ker\omega_S)^*\to T^*\G_S$ yields
\begin{enumerate}[(i)] 
    \item A linear connection $\nabla$ on $\ka$ which preserves the bracket
    \[\nabla_{X}[\xi,\eta]_{\ka}=[\nabla_{X}\xi,\eta]_{\ka}+[\xi,\nabla_{X}\eta]_{\ka},\quad X\in\X(S),\xi,\eta\in\Gamma(\ka); \]
    In particular, $\ka$ is a locally trivial bundle of Lie algebras.
    \item For each $x\in S$, a splitting of the isotropy Lie algebra $\gg_x=\ker \rho_S|_x$ into a direct sum of ideals
    \[ \gg_x=\ka_x\oplus\mathfrak{l}_x. \]
\end{enumerate}
\end{proposition}

There are also some classes of over-symplectic groupoids that are always partially split. For example, we have the following fundamental result from \cite{FM22}.
\begin{theorem}
\label{thm:proper}
Every bundle of ideals for a proper Lie groupoid is partially split.
\end{theorem}


\begin{corollary}
Every proper over-symplectic groupoid is partially split.
\end{corollary}

Other classes of partially split over-symplectic groupoids are discussed in the example section.

\begin{remark}
Bundle of ideals arising from locally trivial groupoid fibrations have been studied in \cite{LSX09} in the context of the theory of non-abelian gerbes. There the authors have defined a cohomology which contains an obstruction class for the existence of multiplicative Ehresmann connections. 
\end{remark}

\subsection{Uniqueness of the local model}

We will extend the Moser method to the groupoid setting. Consider a linear action groupoid
\[\G_S\ltimes \Rep\tto \Rep,\] 
where $\Rep\to S$ is a representation of $\G_S\tto S$.

\begin{proposition}
\label{prop:uniqueness:grpd}
Let $(\G_S,\omega_S)$ be an over-symplectic groupoid. Consider two open groupoid neighborhoods
\[(\G_S\tto S)\hookrightarrow (\G_k\tto M_k) \subset (\G_S\ltimes \Rep \tto \Rep),\quad k=0,1,\]
together with multiplicative symplectic structures $\omega_k\in \Omega^2(\G_k)$ extending $\omega_S$
\[ i^*\omega_k=\omega_S, \quad k=0,1.\] 
Then there exists a symplectic groupoid isomorphism $\Phi:(\G_0',\omega_0)\diffto (\G_1',\omega_1)$, defined between open groupoid neighborhoods $\G_S\subset \G_k'\subset \G_k$, satisfying $\Phi|_{\G_S}=\id$.  
\end{proposition}

Applying this result to the representation $\Rep=\ka^*$, we obtain that the local model is well-defined up to isomorphism. 

\begin{corollary}
The germs of any two groupoid local models for $(\G_S,\omega_S)$ are isomorphic.
\end{corollary}

We begin by making the following observations about such groupoids.

\begin{lemma}
\label{lem:invariantsubgrp}
Any open groupoid neighborhood,
\[(\G_S\tto S)\hookrightarrow (\G\tto M) \subset (\G_S\ltimes \Rep \tto \Rep ),\]
has an open subgroupoid 
\[(\G_S\tto S)\hookrightarrow (\G'\tto M') \subset (\G\tto M)\]
which is invariant under multiplication by $\lambda\in [0,1]$
\[(g,e)\in \G'\ \Longrightarrow\  m_{\lambda}(g,e)=(g,\lambda e)\in\G'.\]
\end{lemma}
\begin{proof}
It is easy to see that the following is a subgroupoid of $\G$
\[\G':=\big\{(g,e)\, : \, (g,\lambda e)\in \G, \forall\, \lambda\in [0,1] \big\}\subset \G,\]
which is open and invariant under multiplication by $\lambda \in [0,1]$.  
\end{proof}

\begin{lemma}
\label{lem:primitive}
Consider an open groupoid neighborhood 
\[(\G_S\tto S)\hookrightarrow (\G\tto M) \subset (\G_S\ltimes \Rep\tto \Rep),\]
which is invariant under multiplication by $\lambda \in [0,1]$. Any closed multiplicative form $\al\in\Omega^k_\mult( \G)$ whose pullback along the zero section $i: \G_S\hookrightarrow \G$ vanishes, i.e.\
\[i^*\al=0\in \Omega^k_{\mult}(\G_S),\]
is exact with primitive a multiplicative form, i.e.\
\[ \al=\d \be, \quad\text{with}\quad \be\in\Omega^{k-1}_\mult( \G). \]
Moreover, we can take $\be$ to vanish at points of $\G_S$.
\end{lemma}

\begin{proof}
Assume first that $\G=\G_S\ltimes \Rep$. We only need to observe that the usual homotopy operator preserves multiplicative forms. Indeed, let $i:\G_S\hookrightarrow \G_S\ltimes \Rep$ be the zero section, $P:\G_S\ltimes \Rep\to \G_S$ the projection, denote by $X$ the Euler vector field of the vector bundle $p:\Rep\to S$ and set $Y=(0,X)\in \X(\G_S\ltimes \Rep)$. Then
\[ \id-(i\circ P)^*=\d H+H\d, \]
where $H:\Omega^\bullet( \G_S\ltimes \Rep)\to \Omega^{\bullet-1}( \G_S\ltimes \Rep)$, $\bullet\geq 1$, is the homotopy operator
\[ H(\gamma):=\int_0^1 \frac{1}{\lambda}m_{\lambda}^*(i_{Y}\gamma)\, \d \lambda. \]
That $\frac{1}{{\lambda}}m_{\lambda}^*(i_{Y}\gamma)$ is smooth at $\lambda=0$ follows because $m_0^*(i_{Y}\gamma)=P^*\circ i^*(i_{Y}\gamma)=0$. Since $m_{\lambda}:\G_S\ltimes \Rep\to \G_S\ltimes \Rep$ is a groupoid morphism and $Y$ is a multiplicative vector field, $H$ maps multiplicative forms to multiplicative forms. So if $\al$ is a closed multiplicative form whose pullback to $\G_S$ vanishes, we have $\al=\d H(\al)$. The form $\be=H(\al)$ is multiplicative and vanishes at points of $\G_S$. 

Finally, note that if $\G\subset \G_S\ltimes \Rep$ is invariant under multiplication by ${\lambda}\in [0,1]$, the same formula gives the homotopy operators for $\G$. 
\end{proof}

The following shows that the representation $\Rep$ from Proposition \ref{prop:uniqueness:grpd} is actually isomorphic to $\ka^*$ (see Proposition \ref{prop:local:models:agree} for a more general version of this result). 

\begin{lemma}\label{lemma:omega:along:zero}
Consider an open groupoid neighborhood,
\[(\G_S\tto S)\stackrel{i}{\hookrightarrow} (\G\tto M) \subset (\G_S\ltimes \Rep \tto \Rep ),\]
together with a multiplicative symplectic form $\omega\in \Omega^2(\G)$. Then $\G_S$ is a coisotropic submanifold. Let $\ka:= (T_S\G_S)^{\perp_{\omega}}\subset T_S\G_S$. Then the vector bundle map:
\[\varphi:\Rep \to \ka^*, \quad  \langle \varphi(u),v\rangle =\omega((0,v),(u,0)),\quad u\in \Rep_x,\ v\in \ka_x,\]
is a $\G_S$-equivariant isomorphism, where we regard
 \[(v,0)\in T_{(1_x,0)}(\G_S\ltimes \Rep),\quad  (0,u)=\frac{\d}{\d t}\Big|_{t=0}(1_x,tu)\in T_{(1_x,0)}(\G_S\ltimes \Rep).\]
\end{lemma}

\begin{proof}
That $\G_S$ is coisotropic follows from Proposition \ref{prop:Poisson:over:symplectic}. Note that $\varphi$ is injective \[\ker\varphi=\Rep \cap \ka^{\perp_{\omega}}=\Rep \cap T_S\G_S=0.\] That $\G_S$ is a coisotropic submanifold implies also that $\ka$ and $V$ have the same rank, so $\varphi$ is a linear isomorphism. 

It remains to show that $\varphi$ is equivariant. Let $g\in \G_{S}$, $x:=\s(g)$, $y:=\t(g)$, and let $v\in \ka_y$, $u\in \Rep_{x}$. Denote $\gamma_t:=\exp_y(tv)\in (\G_{S})_{y}$ and $u_t:=tu$. We have that
\begin{align*}
\omega((g^{-1}\cdot v,0),(0,u))&=\omega\Big(\frac{\d}{\d t}\Big|_{t=0}(g^{-1}\gamma_t g,0), \frac{\d}{\d t}\Big|_{t=0}
(1_{x},u_t)\Big)
\\
&=\omega\Big(\frac{\d}{\d t}\Big|_{t=0}(g^{-1}\gamma_t,0)(g,0), \frac{\d}{\d t}\Big|_{t=0}
(g^{-1},g\cdot u_t)(g,u_t)\Big)
\\
&=\omega
\Big(
\frac{\d}{\d t}\Big|_{t=0}(g^{-1}\gamma_t,0),
\frac{\d}{\d t}\Big|_{t=0}
(g^{-1},g\cdot u_t)
\Big)
+0
\\
&=\omega
\Big(
\frac{\d}{\d t}\Big|_{t=0}(g^{-1},0)(\gamma_t,0),
\frac{\d}{\d t}\Big|_{t=0}(g^{-1},g\cdot u_t)(1_y,g\cdot u_t)
\Big)\\
&=0+
\omega
\Big(
\frac{\d}{\d t}\Big|_{t=0}(\gamma_t,0),
\frac{\d}{\d t}\Big|_{t=0}(1_y,g\cdot u_t)
\Big)\\
&=\omega((v,0),(0,g\cdot u)),
\end{align*}
where we have used twice that $\omega$ is multiplicative. This equality is equivalent to  $\varphi$ being $\G_{S}$ equivariant: $\varphi(g\cdot u)=g\cdot \varphi(u)$.
\end{proof}

We are now ready to prove the multiplicative Weinstein-Moser theorem. 
\begin{proof}[Proof of Proposition \ref{prop:uniqueness:grpd}]
Let $\ka:=\ker\omega_S|_S$. Consider the $\G_S$-equivariant isomorphisms from Lemma \ref{lemma:omega:along:zero}, $\varphi_k:\Rep\diffto \ka^*$, corresponding to $\omega_k$. By pushing forward $\omega_k$ along the groupoid isomorphisms $\id\times \varphi_k:\G_S\times \Rep\diffto \G_S\ltimes \ka^*$, we may assume that $\Rep=\ka^*$ and that $\omega_k$ restricted to $\ka\times_S\ka^*$ is the canonical pairing. Therefore, for each $t$, the 2-form $\omega_t:=t\omega_1+(1-t)\omega_0$ is non-degenerate on $T_S(\G_S\ltimes \ka^*)\simeq T_S\G_S\oplus \ka^*$. 

We claim that there exists an open subgroupoid $\G\subset \G_S\ltimes \ka^*$ containing $\G_S$ and invariant under multiplication $m_\lambda$, $\lambda\in [0,1]$, such that, for all $t\in [0,1]$, $\omega_t$ is a multiplicative symplectic form on $\G$. The 2-form $\omega_t$ is defined and multiplicative on the open subgroupoid $\G_2:=\G_0\cap \G_1\tto M_2:=M_0\cap M_1$. Let $M_3\subset M_2$ be the set consisting of points $x$ such that $\omega_t|_{1_x}$ is non-degenerate for all $t\in [0,1]$. As we have seen, $M_3$ contains $S$, and by the Tube Lemma, $M_3$ is open. By Proposition \ref{prop:over:symplectic} (d), for all $t\in [0,1]$, $\omega_t$ is over-symplectic on the open subgroupoid $\G_3:=\G_2|_{M_3}\tto M_3$, and by the conclusion of that proposition, $\omega_t$ is non-degenerate on $\G_3$. Finally, we let $(\G\tto M)\subset (\G_3\tto M_3)$ be an open subgroupoid containing $\G_S$ which is invariant under multiplication by $t\in [0,1]$, as in Lemma \ref{lem:invariantsubgrp}.

We apply a groupoid version of Moser's trick to the path of symplectic, multiplicative forms $\omega_t\in \Omega^2_{\mult}(\G)$, $t\in [0,1]$. If $X_t$ is a time-dependent multiplicative vector field generating an isotopy of groupoid automorphisms $\Phi^t:\G\to \G$, starting at the identity at $t=0$, such that
\[ (\Phi^t)^*\omega_t=\omega_0,\quad \forall t\in[0,1], \]
then, by the usual argument, $X_t$ must satisfy the equation
\[ \d i_{X_t}\omega_t=\omega_0-\omega_1. \]
Since $\omega_0-\omega_1$ is a closed, multiplicative 2-form that vanishes on points of $\G_S$ and $\G$ is invariant under multiplication by $t\in [0,1]$, by Lemma \ref{lem:primitive}, we can find a primitive $\be$ which is multiplicative and vanishes on $\G_S$. Consider the vector field $X_t\in \X(\G)$ defined by the equation
\[  i_{X_t}\omega_t=\be,\quad t\in [0,1].\]
Then $X_t$ is a multiplicative vector field vanishing on $\G_S$. The isotopy $\Phi^t$, generated by $X_t$, is defined on an open subgroupoid of $\G$ containing $\G_S$, and satisfies $(\Phi^t)^*(\omega_t)=\omega_0$. Then $\Phi^1$ is the desired symplectic groupoid automorphism. 
\end{proof}

\subsection{Linearization of symplectic groupoids}\label{sec:linearization:grpds}

Having established conditions for the existence of a groupoid local model, we now ask
\begin{itemize}
\item To what extent can one expect the groupoid local model of a partially split $(\G_S,\omega_S)$ to be a \emph{normal form}?
\end{itemize}

Recall that for a Lie groupoid $\G\tto M$ with a \emph{saturated} submanifold $S\subset M$, one defines the \textbf{linear approximation} to $\G$ around $S$ as the normal bundle to the restriction $\G_S=\G|_S$
\[ \nu(\G_S)\tto \nu(S),\]
where the groupoid structure is induced from the restriction of $T\G\tto TM$ to $T_SM$. More concretely, it can also be identified with the action groupoid
\[ \nu(\G_S)= \G_S\ltimes \nu(S), \]
associated with the normal representation of $\G_S\tto S$ on $\nu(S)$.


Recall (see \cite{dHFe18}) that a Lie groupoid $\G\tto M$ is called  \textbf{linearizable} around a saturated submanifold $S\subset M$ if there are groupoid neighborhoods $\G_S\subset\U\subset \G$ and $\G_S\subset \V\subset \nu(\G_S)$ and a groupoid isomorphism
\[ \Phi:\U\diffto \V,\]
which is the identity on $\G_S$. 
It is called \textbf{invariantly linearizable} around $S$, if there are saturated neighborhoods $S\subset U\subset M$ and $S\subset V\subset \nu(S)$ such that
\[ \U=\G|_U \quad \text{and} \quad \V=\nu(\G_S)|_V. \]
We call the isomorphism $\Phi$ a \textbf{linearization} of $\G$ around $S$. The linearization problem for Lie groupoids has been intensively studied  \cite{CrSt13,dHFe18,Weinstein00,Weinstein02,Zung06}. In particular, it is well-known that proper (respectively, $\s$-proper) Hausdorff Lie groupoids can be linearized (respectively, invariantly linearized) around saturated submanifolds.

\medskip

Returning to the coisotropic embedding problem of an over-symplectic groupoid, it turns out that the groupoid from the local model is naturally isomorphic to the linear approximation. Notice that this property is independent of the existence of a multiplicative Ehresmann connection. 

\begin{proposition}
\label{prop:local:models:agree}
Let $(\G_S,\omega_S)$ be an over-symplectic groupoid. Given a groupoid coisotropic embedding $i:(\G_S,\omega_S)\hookrightarrow (\G,\omega)$, then the isomorphism given by the symplectic form descends to a groupoid isomorphism
\[
\xymatrix{
T_{\G_S}\G \ar[r]^{\omega^\flat}\ar[d] & T^*_{\G_S}\G  \ar[d]\\
\nu(\G_S) \ar@{-->}[r]_{\simeq} & (\ker\omega_S)^* 
}
\]
whose base map is an isomorphism of $\G_S$-representations
\[\nu(S)\simeq \ka^*,\quad \textrm{where}\ \ \ka:=\ker(\omega_S)|_{S}.\]
\end{proposition}

\begin{proof}
Since $i$ is a coisotropic embedding, it follows that it induces an isomorphism of vector bundles making the diagram commute. Now observe that, since $\omega$ is multiplicative, all solid arrows in the diagram are groupoid morphisms, so it follows that the dashed arrow is also a groupoid morphism. Denote this isomorphism by $\Phi:\nu(\G_S) \diffto (\ker\omega_S)^*$ and its base map by $\varphi:\nu(S)\diffto \ka^*$. 
As we have remarked before, these groupoids are isomorphic to action groupoids
\[\nu(\G_S)\simeq \G_S\ltimes \nu(S),\quad (\ker\omega_S)^*\simeq \G_S\ltimes \ka^*,\]
where the isomorphisms are the pairs $(\pr_{\G_S},\s)$, bundle projection and source map. Since $\Phi$ is a vector bundle map covering the identity, we obtain a commutative diagram of groupoid isomorphisms
\[
\xymatrix{
\nu(\G_S) \ar[r]^{\Phi}\ar[d] & (\ker\omega_S)^* \ar[d]\\
\G_S\ltimes\nu(S)  \ar[r]^{(\id,\varphi)} & \G_S\ltimes\ka^*. 
}
\]
In particular, the map $\varphi$ is an isomorphism of $\G_S$-representations. 
\end{proof}

The previous proposition allows to restate Proposition \ref{prop:partially:split:grpd:optimal} as follows.


\begin{corollary}
Given a symplectic groupoid $(\G,\omega)\tto M$, its restriction $(\G_S,\omega_S)$ to a saturated submanifold $S$ is partially split if and only if the linear approximation $\nu(\G_S)\tto\nu(S)$ admits a closed multiplicative 2-form $\omega_0$ extending $\omega_S$ and which is non-degenerate along $\G_S$.
\end{corollary}

Proposition \ref{prop:local:models:agree} also shows that if a symplectic groupoid $(\G,\omega)$ is isomorphic to the local model $(\G_0,\omega_0)$ around $\G_S$, then the groupoid $\G$ is linearizable around $S$. Next, we show that Proposition \ref{prop:uniqueness:grpd} implies that this is the only obstruction.

%
%
%

\begin{theorem}
\label{thm:normal:form:groupoid}
If a symplectic groupoid $(\G,\omega)$ is linearizable (as a Lie groupoid) around a saturated submanifold $S$, then $(\G,\omega)$ is locally isomorphic (as a symplectic groupoid) to the groupoid local model of $(\G_S,\omega_S)$.
\end{theorem}

\begin{proof}
A linearization of $\G$ around $S$, induces a multiplicative symplectic structure $\omega_1$ on a groupoid neighborhood $\G_1\subset \G_S\ltimes \ka^*$ of $\G_S$ such that $i^*(\omega_1)=\omega_S$. By Proposition \ref{prop:partially:split:grpd:optimal}, $(\G_S,\omega_S)$ is partially split, and so the groupoid local model $(\G_0,\omega_0)$ of Proposition \ref{prop:local:model:grpd} exists. Since $i^*(\omega_0)=\omega_S$, Proposition \ref{prop:uniqueness:grpd} implies that $(\G_0,\omega_0)$ and $(\G_1,\omega_1)$ are isomorphic around $\G_S$, and so $(\G,\omega)$ and the groupoid local model $(\G_0,\omega_0)$ are isomorphic around $\G_S$.
\end{proof}

The linearization theorems mentioned before (see, e.g., \cite{dHFe18}) imply the following.

\begin{corollary}\label{corollary:linearization:proper:Hausdorff}
Let $(\G,\omega)\tto M$ be a proper symplectic groupoid. For any closed embedded invariant submanifold $S\subset M$, $(\G,\omega)$ is isomorphic around $S$ to the groupoid local model of $(\G_S,\omega_S)$. If, additionally, $(\G,\omega)$ is target-proper, then $\G$ is invariantly linearizable. 
%
%
\end{corollary}

Let us now see that Theorem \ref{thm:main:grpd:normal:form} in the Introduction follows from Corollary \ref{corollary:linearization:proper:Hausdorff}. First, the assumption in Theorem \ref{thm:main:grpd:normal:form} looks weaker than in this corollary, since there one only assumes the restriction $\G_S\tto S$ to be target-proper. However, this is only apparent, as the following general fact shows. 

\begin{lemma}
Let $f:X\to Y$ be a submersion with connected fibers, and let $Y_c\subset Y$ be the collection of points $y\in Y$ whose preimage $f^{-1}(y)\subset X$ is compact. Then $Y_c$ is open in $Y$, and $f:f^{-1}(Y_c)\to Y_c$ is a locally trivial fibration. 
\end{lemma}
\begin{proof}
The foliation on $X$ given by the fibers of $f$ has trivial holonomy. By applying the Local Reeb Stability Theorem (see, e.g., \cite{MM04}) to a fiber $f^{-1}(y)$, $y\in Y_c$, we obtain a saturated open neighborhood $f^{-1}(y)\subset U\subset X$ and a diffeomorphism $U\simeq f^{-1}(y)\times V$, with $y\subset V\subset Y$, under which $f$ becomes $\mathrm{pr}_V$. Since the fibers of $f$ are connected, it follows that $U=f^{-1}(V)$. Hence $V\subset Y_c$, and $f$ is can be trivialized above $V$. 
\end{proof}

\begin{proof}[Proof of Theorem \ref{thm:main:grpd:normal:form}] The lemma applied to $\t:\G\to M$ 
yields the saturated open set $M_c$ for which $\G|_{M_c}$ is target-proper. 
The result follows from Corollary \ref{corollary:linearization:proper:Hausdorff}.
\end{proof}

\section{Examples of groupoid local models}\label{sec:examples:groupoid:local:models}

In \cite[Section 3]{FM22} we give many examples of classes of groupoids with bundles of ideals that are partially split. These lead to examples of over-symplectic groupoids which are partially split. 
We give here some classes where one can find the symplectic groupoid local models explicitly.

\subsection{Transitive over-symplectic groupoids}\label{ex:transitive:over-symplectic:cont}
Given a principal $G$-bundle $P\to S$, the gauge groupoid 
\[ \G_S:=P\times_G P\tto S, \] 
is a transitive groupoid for which the adjoint bundle $\ka=P\times_G\gg$ is a bundle of ideals. It was observed in \cite[Section 3]{FM22} that this is always partially split and that there is a 
1-to-1 correspondence
\[ 
\left\{\txt{multiplicative connection\\ 1-forms $\alpha\in\Omega^1(\G_S,\ka)$ \,}\right\}
\tilde{\longleftrightarrow}
\left\{\txt{principal bundle\\ connections $\eta\in\Omega^1(P;\gg)$\,} \right\}
\]
This correspondence is as follows: given a principal bundle connection $\eta\in\Omega^1(P,\gg)$, the 1-form $\pr_2^*\eta-\pr_1^*\eta\in \Omega^1(P\times P,\gg)$ descends to a unique multiplicative connection 1-form $\alpha\in\Omega^1(\G_S,\ka)$ such that
\begin{equation}
    \label{eq:connection:principal}
    q^*\al=\pr_2^*\eta-\pr_1^*\eta,
\end{equation}
where $q:P\times P \to P\times_G P$ is the projection. 

If $(\G_S,\omega_S)$ is a transitive over-symplectic groupoid, then we saw in the example of Subsection \ref{ex:transitive:over-symplectic} that it is isomorphic to the gauge groupoid of a principal $G$-bundle $p:P\to S$ over a symplectic manifold $(S,\omega)$,
\[ \G_S\simeq P\times_G P\tto S.\] 
Under this isomorphism, the over-symplectic structure is given by
\[\omega_S=(p\circ\pr_2)^*\omega-(p\circ\pr_1)^*\omega\in\Omega^2(\G_S).\]
It follows that the bundle of ideals is the adjoint bundle
\[ \ka=(\ker\omega_S)|_S=P\times_G \gg.  \]
Hence, every transitive over-symplectic groupoid $(\G_S,\omega_S)$ is partially split.  An explicit form for the local model $(\G_0,\omega_0)$ associated with a principal connection $\eta$ is given by a groupoid neighborhood of the zero section in 
\[ \G_S\ltimes \ka^*=(P\times P)\times_G \gg^*, \]
and the closed two 2-form $\omega_0$ is obtained from expression \eqref{eq:symplectic:form:local:model} 
\[ \omega_0=\pr_{\G_S}^*\omega_S+\d\langle \al,\cdot\rangle, \]
and $\al$ is determined by \eqref{eq:connection:principal}.
This local model is well-known \cite{CrMa12,MarcutPhD} and is the groupoid version of Vorobjev's local model around symplectic leaves \cite{Vorobjev01}.

\subsection{Principal type}\label{example:os:princ:type:part:split}
Consider an over-symplectic groupoid of principal type, $\G_S:=\Sigma\times_{S\times S}\H$, associated to the symplectic groupoid $(\Sigma,\omega_{\Sigma})\tto S$ and the transitive groupoid $\H\simeq P\times_GP\tto S$, as in the example from Subsection \ref{example:os:princ:type}. Note that
\[\ker\omega_S=\Sigma\times_{S\times S}\big(\ker\d\t_{\H}\cap \ker\d\s_{\H}\big)\simeq \G_S\times_S\ka,\]
where, as in the previous example, the bundle $\ka$ coincides with the adjoint bundle
\[ \ka=(\ker\omega_S)|_S=P\times_G \gg.\]
Also as in the previous example, a principal connection $\eta\in \Omega^1(P;\gg)$ gives rise to a multiplicative connection $\alpha\in \Omega^1(\H,\ka)$. The projection $\pr_{\H}:\G_S\to \H$ is a groupoid map, so it follows that we have a multiplicative connection 1-form on $\G_S$
\[\pr_{\H}^*\alpha\in \Omega_{\mult}^{1}(\G_S,\ka).\]
We conclude that over-symplectic groupoids of principal type are partially split. The local model has supporting groupoid an open subgroupoid
\[\G_0\subset (P\times_S\Sigma\times_SP)\times_G\gg^*\tto P\times_G\gg^*, \]
and multiplicative form
\[ \omega_0=\pr^*_\Sigma\omega_\Sigma+\d\langle\pr^*_\H\al,\cdot\rangle. \]

\subsection{Action groupoids}
\label{ex:action} Let $(\G_S,\omega_S)$ be an over-symplectic groupoid, where $\G_S$ is the action groupoid of an action of a Lie group $G$ on a manifold $S$
\[ \G_S=G\ltimes S\tto S. \]
The kernel of $\omega_S$ defines a bundle of ideals $\ka\hookrightarrow \gg\times S$. This means that each fiber is an ideal $\ka_x\subset \ker\rho_x$ such that
\[ \Ad_g(\ka_x)=\ka_{gx}, \]
where $\rho:\gg\to\X(S)$ denotes the infinitesimal action. 

If there exists a $G$-equivariant splitting $l:\gg\times S\to\ka$ then the over-symplectic groupoid $(\G_S,\omega_S)$ is partially split: this splitting yields a multiplicative connection 1-form $\al\in\Omega^1_M(G\ltimes S;\ka)$ by setting
\[ \al_{(g,x)}(\d L_g(v),w):= l(v,x). \]

Consider for example a compact Lie group $G$ with Lie algebra $\gg$ and denote by $(\cdot,\cdot)$ a $G$-invariant inner product on $\gg$. Given any $G$-invariant submanifold $S\subset \gg$ the action groupoid
\[ \G_S:=G\ltimes S\tto S. \]
is an over-symplectic groupoid with closed 2-form
\begin{equation}
    \label{eq:action:groupoid:constant:form}
    (\omega_S)_{(g,x)}((\d L_g(v_1),z_1),(\d L_g(v_2),z_2))=(v_2,z_1)-(v_1,z_2)-([v_1,v_2],x), 
\end{equation} 
where $(g,x)\in G\times S$, $v_1,v_2\in \gg$ and $z_1,z_2\in T_x S\subset\gg$. The corresponding Lie algebra bundle $\ka\subset \gg\times S$ is:
\begin{equation}\label{eq:ka:normal:bundle}
    \ka_x=\{(w,x): w\in (T_x S)^\perp\}.
\end{equation}
The Lie group $G$ acts on the total space of the bundle $\ka$
\[ g\cdot (w,x)=(\Ad_g(w),gx), \]
and the groupoid $\G_S\ltimes\ka\tto \ka$ is then identified with the action groupoid $G\ltimes \ka\tto \ka$. Using the inner product, we can identify $\gg\simeq\gg^*$ and $\ka\simeq\ka^*$. 

The orthogonal projection gives a $G$-invariant splitting
\[ l:\gg\times S\to\ka, \quad l_x(v):=\pi_{\ka_x}(v). \]
Then we obtain the multiplicative connections 1-form $\al\in\Omega^1_M(G\ltimes S;\ka)$ defined by
\[ \al_{(g,x)}(\d L_g(v),z):=\pi_{\ka_x}(v). \]
This gives a closed, multiplicative, 1-form $\langle \al,\cdot\rangle\in\Omega^1_M(G\ltimes\ka)$ (recall that $\ka\simeq\ka^*$)
\begin{equation}
    \label{eq:action:groupoid:linear:form}
    \langle \al,\cdot\rangle_{(g,w,x)}(\d L_g(v),u,z)=(v,w), 
\end{equation} 
where $g\in G$, $(w,x)\in\ka$, $v\in\gg$ and $(u,z)\in T_{(w,x)}\ka\subset \gg\times\gg$. The symplectic form on the local model $\G_0\subset G\ltimes\ka$ is given, as usual, by restricting the closed, multiplicative, 2-form
\[ \omega_0=\pr_{G\ltimes S}^*\omega_S+\d\langle \al,\cdot\rangle \in \Omega^2_M(G\ltimes\ka).\]
Using \eqref{eq:action:groupoid:constant:form} and \eqref{eq:action:groupoid:linear:form} we find
\begin{align}
(\omega_0)_{(g,w,x)}&((\d L_g(v_1),u_1,z_1),(\d L_g(v_2),u_2,z_2))=\notag\\
&=(v_2,z_1)-(v_1,z_2)-([v_1,v_2],x)+(v_2,u_1)-(v_1,u_2)-([v_1,v_2],w)\notag\\
    &=(v_2,z_1+u_1)-(v_1,z_2+u_2)-([v_1,v_2],x+w)\label{eq:formula:sympl:str}.
\end{align} 

Note that $\ka$ is the normal bundle of $S$ in $\gg$ (see \eqref{eq:ka:normal:bundle}) and that the Riemannian exponential map $\ka\to \gg$ is given by
$(w,x)\mapsto w+x$.
In particular, it is $G$-equivariant, and so it induces a map of groupoids 
\[ G\ltimes\ka\to G\ltimes\gg,\]
which is an open embedding around $G\ltimes S$. Formula \eqref{eq:formula:sympl:str} shows that it is in fact an embedding of symplectic groupoids, where we use the identifications $T^*G\simeq G\times\gg^*\simeq G\ltimes\gg$ given by left-invariant translations and the invariant inner product, and endow $T^*G$ with the canonical symplectic structure.

We conclude that for a compact Lie group $G$ and any $G$-invariant submanifold $S\subset \gg^*$, a choice of $G$-invariant inner product gives an identification of the local model $\G_0$ around $S$ with $T^*G$. In other words, the symplectic groupoid $T^*G\tto \gg^*$ is linearizable around any invariant submanifold. This also agrees, of course, with Corollary \ref{corollary:linearization:proper:Hausdorff}, since $T^*G\tto\gg^*$ is an s-proper groupoid whenever $G$ is compact. The next example shows that the assumption that $G$ is compact is crucial.

\begin{example}\label{example:aff1}
Consider the non-abelian 2-dimensional Lie group $G$ and the associated symplectic groupoid $(T^*G,\omega_\can)\tto \gg^*$, which integrates the linear Poisson structure on $\gg^*\simeq\R^2$ given by
\[ \pi=x\,\pd{x}\wedge\pd{y}. \]
Then $T^*G$ is not linearizable around the closed, embedded Poisson submanifold $S=\{x=0\}$. For this, note that the over-symplectic groupoid $\G_S:=(T^*G)|_S$, is, as groupoid, the trivial bundle of Lie groups $G\times S\to S$, and that the orbits of $\G_S$ on the normal bundle $\nu(S)=\ka^*$ are one dimensional. So, as in Example \ref{ex:non:split}, a partial splitting does not exist.
\end{example}

\section{Local models of Poisson structures}
\label{sec:local:model}

%
%
%
%
%

We now come back to the infinitesimal level with the aim of constructing local models for elements $\tau\in J^1_S\Pi(M,S)$. Since any such element can be represented by a pair $\tau=(T^*_SM,\mu_S)$, i.e., as the infinitesimal counterpart of an over-symplectic groupoid, all we have to do is apply the dictionary to transfer the construction in the previous section to obtain a local model for Poisson structures.


\subsection{The local model}
\label{sec:local:model:algbrd}

Let us fix a first order jet of a Poisson structure
\[ (A_S,[\cdot,\cdot]_{A_S},\rho_{A_S}),\quad \mu_S:A_S\to T^*S,\]
i.e., a Lie algebroid together with a surjective IM 2-form. As usual, $\pi_S$ denotes the induced Poisson structure on $S$.
Let 
\[\ka:=\ker(\mu_S).\]
The IM conditions imply that
\begin{equation}\label{eq:short:exact:sequence}
0\rmap \ka\rmap A_S\stackrel{\mu_S}{\rmap} T^*S\rmap 0,
\end{equation}
is a short exact sequence of Lie algebroids, where $T^*S\Ato S$ is regarded as the cotangent Lie algebroid of $\pi_S$. In particular $\ka$ is a bundle of ideals in $A_S$. Therefore, $A_S$ has a canonical representation on $\ka$
\begin{equation}
\label{eq:representation:algbrd} 
\nabla^{\ka}_\al\gamma:=[\al,\gamma]_{A_S},\quad \al\in\Gamma(A_S), \ga\in\Gamma(\ka).
\end{equation}
This is the infinitesimal version of the groupoid representation \eqref{eq:representation:grpd}. It gives rise to the semi-direct product Lie algebroid $A_S\times_S \ka\Ato S$, with Lie bracket
\[ [(\al_1,\gamma_1),(\al_2,\gamma_2)]:=([\al_1,\al_2]_{A_S},\nabla^\ka_{\al_1}\gamma_2-\nabla^\ka_{\al_2}\gamma_1), \]
and anchor $\rho_{A_S}\circ\pr_{A_S}$. This is the infinitesimal version of the groupoid $\G_S\times_S\ka\tto S$.

The infinitesimal version of the groupoid $\G_S\ltimes\ka^*\tto \ka^*$ is, of course, the action algebroid $A_S\ltimes \ka^*\Rightarrow  \ka^*$ associated with the dual representation $\nabla^{\ka^*}$ of $A_S$ on $\ka^*$
\[ \Lie_{\rho_{A_S}(\al)}\langle \be,X \rangle=\langle \nabla^{\ka}_\al\be, X \rangle+\langle \be, \nabla^{\ka^*}_\al X \rangle. \]

To build the local model, we need know whether the Lie algebroid $A_S\ltimes \ka^*\Ato \ka^*$ admits a closed IM 2-form $\mu_0$ which is non-degenerate around $S$ and whose restriction to $A_S$ is $\mu_S$. From the global picture, we know that such an IM form might not exist, and so we need the infinitesimal analogue of multiplicative connections and the partially split condition. 

Using the notion of infinitesimal multiplicative (IM) form with coefficients -- see Appendix \ref{appendix} -- one obtains the infinitesimal version of Definition \ref{def:partially:split:grpd}.

\begin{definition}
\label{def:partially:split:jet}
An \textbf{IM connection 1-form} on a first order jet $(A_S,\mu_S)$ with kernel $\ka=\ker\mu_S$ is a IM 1-form $(L,l)\in\Omega^1_\imult(A_S;\ka)$ whose symbol satisfies
\[ l(\xi)=\xi,\quad \forall \xi\in\Gamma(\ka). \]
A first order jet is called \textbf{partially split} if it admits an IM connection 1-form. 
\end{definition}

The general theory of IM connections 1-forms was developed in \cite{FM22}, and will be recalled briefly in the next subsection.

Given an IM connection 1-form $(L,l)\in\Omega^1_\imult(A_S;\ka)$ we have an associated IM 1-form $\langle (L,l),\cdot\rangle\in\Omega^1_\imult(A_S\ltimes\ka^*)$ which has components $\mu:A_S\ltimes\ka^*\to\R$ and $\zeta:A_S\ltimes\ka^*\to T\ka^*$ defined on constant sections $\al\in \Gamma(A_S)$ by 
\begin{align}
    \mu(\al)&=\widetilde{l(\al)},\\
    \zeta(\al)&=\d\widetilde{l(\alpha)}-\widetilde{L(\alpha)},
\end{align}
where for a $\ka$-valued form $\eta\in \Omega^{k}(S,\ka)$ we denote by $\widetilde{\eta}\in \Omega^{k}(\ka^*)$ the $k$-form given by $\widetilde{\eta}|_{\xi}:=\langle \pr^*(\eta),\xi\rangle$, with $\pr:\ka^*\to S$ denoting the bundle projection. That $\langle (L,l),\cdot\rangle=(\mu,\zeta)$ is indeed a IM form is proven in \cite[Proposition 5.4]{FM22}, and follows directly from the IM conditions. 

Using the IM-differential from definition \eqref{eq:IM:differential}, one obtains the closed IM 2-form
\[ \d_\imult\langle (L,l),\cdot\rangle=(\zeta,0)\in\Omega^2_\imult(A_S\ltimes\ka^*). \]
Then we can state the infinitesimal analogue of Proposition \ref{prop:local:model:grpd}.

\begin{proposition}
\label{prop:local:model:algbrd}
Given a first order jet $(A_S,\mu_S)$ and an IM connection 1-form $(L,l)\in\Omega^1_\imult(A_S;\ka)$ the closed IM 2-form \begin{equation}
\label{eq:normal:IM:form} 
\mu_0=\pr^*\mu_S+\d_\imult\langle (L,l),\cdot\rangle\in\Omega^2_\imult(A_S\ltimes\ka^*),
\end{equation}
is non-degenerate along $S$. Moreover, the zero section $i_0:A_S\hookrightarrow A_S\ltimes\ka^*$ is an algebroid embedding satisfying $i_0^*\mu_0=\mu_S$.
\end{proposition}

\begin{proof}
The pullback of a (closed) IM form by a Lie algebroid morphism is a (closed) IM form. Since the projection $\pr:A_S\ltimes\ka^*\to A_S$ is a Lie algebroid morphism, it follows that \eqref{eq:normal:IM:form} defines a closed IM 2-form.

Using the canonical isomorphism
\[ T^*_S \ka^*\simeq \ka\oplus T^*S, \]
the restriction
\[ (\mu_0)|_{0_S}:A_S\to T^*_S\ka^*\] 
has components
\begin{itemize}
\item $(\pr^*\mu_S)|_{0_S}:A_S\to \ka\oplus T^*S$, $\al\mapsto (0,\mu_S(\al))$;
\item $(\d_\imult\langle (L,l),\cdot\rangle)|_{0_S}:A_S\to \ka\oplus T^*S$, $\al\mapsto (l(\al),0)$.
\end{itemize}
Hence, $(\mu_0)|_{0_S}$ is a fiberwise surjective map. Since the two vector bundles have the same rank, $\mu_0$ is a non-degenerate IM 2-form along $0_S$.
\end{proof}

In the setting of the previous proposition, let $M_0\subset \ka^*$ be the open neighborhood of $S$ where $\mu_0$ is non-degenerate. Then
\[ A_0:=(A_S\ltimes\ka^*)|_{M_0}\simeq A_S\ltimes M_0, \]
carries the non-degenerate, closed, IM 2-form
\[ \mu_0|_{A_0}:A_0\diffto T^*M_0.\]

%

\begin{definition}
\label{def:algbrd:local:model}
Given a partially split first order jet $(A_S,\mu_S)$, the pair $(A_0,\mu_0)$ is called the \textbf{Lie algebroid local model}  corresponding to the IM connection 1-form $(L,l)\in\Omega^1_\imult(A_S;\ka)$. The corresponding \textbf{Poisson local model} is the Poisson manifold $(M_0,\pi_0)$, where 
\[ \pi_0^\sharp:=\rho_0\circ\mu_0^{-1}. \]
\end{definition}
Note that $\mu_0:A_0\diffto T^*M_0$ is an isomorphism of Lie algebroids, giving an identification 
\[ J^1_S\pi_0=(A_S,\mu_S).\]

In conclusion, we have constructed a local model, in the sense of Definition \ref{def:local:model}, for the class $\mathscr{C}\subset J^1_S\Pi(M,S)$ of partially split first order jets of Poisson structures. Namely, by the axiom of choice, for each $\tau\in \mathscr{C}$, we can chose a IM connection 1-form $(L_{\tau},l_{\tau})$, and we let $\sigma(\tau)$ be the associated local model, i.e., $\sigma(\tau):=(M_0,\pi_0)$. We obtain Theorem \ref{thm:main:one} from the Introduction. In the next subsection we will see that different choices of IM connection 1-forms produce isomorphic local models.

\subsection{The partially split condition for jets}
\label{sec:local:model:description}

In order to understand the partially split condition and the existence of IM connection 1-forms, as in Definition \ref{def:partially:split:jet}, we recall briefly the notion of IM Ehresmann connection introduced in \cite{FM22}. 

Consider a Lie algebroid $A_S$ and a bundle of ideals $\ka\subset A_S$, i.e., $\ka\subset \ker\rho_{S}$ and $[\al,\ga]_{A_S}\in \Gamma(\ka)$, for all $\al\in\Gamma(A_S)$, $\ga\in\Gamma(\ka)$. So $\ka$ is a representation of $A_S$
\[
\nabla^{\ka}_\al\gamma:=[\al,\gamma]_{A_S}.
\]

In \cite{FM22} we gave the following equivalent descriptions of IM connections for $\ka$
\begin{enumerate}[(i)]
    \item An {\bf IM Ehresmann connection} for $\ka$ is a VB subalgebroid $E\Ato TS$ of $TA\Ato TS$ such that
    \[ TA=\ka\oplus E.\]
    \item There is a natural inclusion $A_S\times_S\ka\subset TA_S$
    \begin{equation*}
    (\al,\gamma)\mapsto \frac{\d}{\d t}\Big|_{t=0}(\al+t\gamma)
    \end{equation*}
    which is a map of VB algebroids
    \[
    \vcenter{
    \xymatrix@R=10pt{
    A_S\times_S \ka\  \ar@{=>}[dd] \ar[dr]  \ar[dr]  \ar@{^{(}->}@<-0.10pc>[rr] &  &  TA_S
    \ar@{=>}[dd] \ar[dl]\\
     &  A_S  \ar@{=>}[dd]  \\
    0_S\  \ar[dr]  \ar[dr]  \ar@{^{(}-}[r] &\ar[r] & TS \ar[dl] \\
     & S
    }}
    \]
    An IM Ehresmann connection yields a {\bf partial splitting} $\theta:A_S\ltimes \ka^*\to T^*A_S$, i.e., a VB-algebroid map which is a splitting of the dual projection
    \[\vcenter{
    \xymatrix@R=10pt{
    A_S\ltimes\ka^*\, \ar@{=>}[dd] \ar[dr] \ar@/_/@{-->}[rr]_{\theta}&  &  T^*A_S\ar@{->>}[ll] 
     \ar@{=>}[dd] \ar[dl]\\
     &  A_S  \ar@{=>}[dd]  \\
    \ka^* \ar[dr] \ar@/_/@{-->}[rr]& \ar@{->>}[l] & A^*_S \ar[dl] \ar@{-}[l]\\
     & S
    }}\]
    \item Given a partial splitting $\theta:A_S\ltimes \ka^*\to T^*A_S$ for $\ka$,  one obtains an \textbf{IM connection 1-form}, i.e., an IM 1-form $(L,l)\in\Omega^1_\imult(A_S;\ka)$ whose symbol satisfies
    \[ l(\xi)=\xi,\quad \forall \xi\in\Gamma(\ka). \]
    It is defined by setting for any  $\al\in\Gamma(A)$ and $v\in TM$
    \begin{align*}
    L(\al)(v)&:=\pr_E(\theta^\vee(\d\al(v))),\\
    l(\al)&:=\pr_E\theta^\vee(\hat{\al})|_M,
    \end{align*}
    where $\hat{\al}\in\X(A)$ is the vertical lift of the section $\al$.
    \item Finally, given an IM connection 1-form $(L,l)\in\Omega^1_\imult(A_S;\ka)$, one obtains a {\bf linear}, closed multiplicative 2-form $\mu^\lin\in\Omega^2(A_S\ltimes\ka^*)$ by
    \[ \mu^\lin:=\d_{\imult}\langle (L,l),\cdot\rangle. \]
\end{enumerate}

We now return to our case of interest, namely first order jets $(A_S,\mu_S)$ where the relevant bundle of ideals is $\ka=\ker\mu_S$. In the sequel, we will use the equivalence between (i)-(iv) without further notice. So, for example, $(A_S,\mu_S)$ is partially split if it admits a partial splitting $\theta:A_S\ltimes \ka^*\to T^*A_S$ as in (ii). This is equivalent to our original Definition \ref{def:partially:split:jet}, which was stated in terms of the existence of an IM connection 1-form $(L,l)\in\Omega^1_\imult(A_S;\ka)$, as in (iii). The closed IM 2-form in the local model \eqref{eq:normal:IM:form} can be written 
using the partial splitting $\theta$ from (ii) as
\[ \mu_0=\pr^*_{A_S}\mu_S+\theta^*\mu_{\can},\]
where $\mu_{\can}\in \Omega^{2}_{\imult}(T^*A_S)$ is the canonical IM 2-form (see Example \ref{ex:canonical:IM}), or using the linear IM form $\mu^{\lin}$ from (iv) as
\[ \mu_0=\pr^*_{A_S}\mu_S+\mu^\lin.\]
In this expression, the first term is a constant IM 2-form and the second term is a linear IM 2-form.

\begin{remark}\label{remark:split:is:not:part:split}
Let us stress that the partial split condition is not implied by the existence of a Lie algebroid splitting $r:T^*S\to A_S$ of $\mu_S$ -- see Example \ref{example:log:symplectic:not:p:s}.
\end{remark}

As in the global case -- cf.~Proposition \ref{prop:corollary:partial:splittings} -- the partially split condition implies certain properties of the isotropy Lie algebras of a first order jet

\begin{proposition}[\cite{FM22}]
\label{prop:corollary:IM:partial:splittings}
If $(A_S,\mu_S)$ is a partially split first order jet of Poisson structure, a choice of IM connection 1-form $(L,l)$, gives
\begin{enumerate}[(i)]
    \item a linear connection on $\ka=\ker \mu_S$
    \begin{equation} 
    \label{eq:connection:ka}
    \nabla^L_X\xi:=i_XL(\xi),
    \end{equation}
    which preserves the Lie bracket
    \[\nabla^{L}_{X}[\xi,\eta]_{\ka}=[\nabla^{L}_{X}\xi,\eta]_{\ka}+[\xi,\nabla^{L}_{X}\eta]_{\ka},\quad X\in\X(S),\xi,\eta\in\Gamma(\ka); \]
    \item for each $x\in S$, a decomposition of the isotropy Lie algebra $\gg_x:=\ker \rho_{A_S}|_x$ into a direct sum of ideals
    \[ \gg_x\simeq \ka_x\oplus (\gg_x\cap \ker l). \]
\end{enumerate}
\end{proposition}

We quote from \cite{FM22} another interesting characterization of partially split jets which will be useful later. Given a (usual) connection $\nabla$ on a Lie algebroid $A_S\Ato S$ one has the following associated $A_S$-connections on $A_S$ and $TS$
\begin{align*}
\overline{\nabla}_\al\be&:=\nabla_{\rho(\be)}\al+[\al,\be], \\
\overline{\nabla}_\al X&:=\rho(\nabla_X\al)+[\rho(\al),X]
\end{align*}
They satisfy
\[ \rho(\overline{\nabla}_\al\be)=\overline{\nabla}_\al\rho(\be). \]
The {\bf basic curvature} of $\nabla$ is defined as the tensor
\begin{equation}
\label{eq:basic:curvature:alg} 
R^{\bas}_\nabla(\al,\be)(X):=\nabla_X([\al,\be])-[\nabla_X\al,\be]-[\al,\nabla_X\be]-\nabla_{\overline{\nabla}_\be X}\al+\nabla_{\overline{\nabla}_\al X}\be,
\end{equation}
where $X\in\X(S)$ and $\al,\be\in\Gamma(A_S)$. This curvature arises naturally in the study of \emph{Cartan connections} (see \cite{AC13,Blaom06,Blaom12}). Proposition 5.9 from \cite{FM22} implies the following.

\begin{proposition}
\label{prop:Cartan:connection:splitting}
A first order jet $(A_S,\mu_S)$ with $\ka=\ker\mu_S$ is partially split if and only if there is a vector bundle splitting $l:A_S\to \ka$ of the short exact sequence 
\[ 
\xymatrix{0\ar[r] & \ka \ar[r] & A_S\ar[r]^{\mu_S} & T^*S\ar[r] & 0}
\]
and a (usual) connection $\nabla$ on $A_S$ such that
\begin{equation}
    \label{eq:partial:splitting:basic:curvature}
    \overline{\nabla}l=0,\qquad l(R^{\bas}_\nabla)=0.
\end{equation} 
In this case, the pair $(l,\nabla)$ determines a IM connection 1-form $(L,l)\in\Omega^1_\imult(A_S;\ka)$ by setting
\[ 
L:\Gamma(A_S)\to \Omega^1(S;\ka), \quad i_X L(\al):=l(\nabla_X\al). 
\]
Moreover, every IM connection 1-form $(L,l)$ takes this form for some pair $(l,\nabla)$ satisfying \eqref{eq:partial:splitting:basic:curvature}.
\end{proposition}

\begin{remark}
The connection $\nabla$ on $A_S$ in the previous proposition is not uniquely determined by the IM connection 1-form $(L,l)\in\Omega^1_\imult(A_S;\ka)$. In particular, this connection should not be confused with the connection $\nabla^L$ on $\ka$ from Proposition \ref{prop:Cartan:connection:splitting}. They are related by
\[ \nabla_X^L\xi=l(\nabla_X\xi), \]
for $X\in\X(S)$ and $\xi\in\Gamma(\ka)$.
\end{remark}

\subsection{Existence and uniqueness of the local model}
The following proposition gives some geometric meaning to the partially split condition: it shows that the partially split condition is equivalent to the algebroid local model $A_S\ltimes\ka^*$ being the cotangent algebroid of a Poisson structure extending the given first jet. This is the infinitesimal version of Proposition \ref{prop:partially:split:grpd:optimal}.

\begin{proposition}
\label{prop:partially:split:algbrd:optimal}
A first order jet $(A_S,\mu_S)$ with kernel $\ka:=\ker\mu_S$ is partially split if and only if there exists an open neighborhood $S\subset M_0\subset \ka^*$ and a closed, non-degenerate IM 2-form $\mu\in\Omega^2_\imult(A_S\ltimes M_0)$ whose pullback along the Lie algebroid map $i:A_S\hookrightarrow A_S\ltimes \ka^*$ is $\mu_S$
\[ i^* \mu=\mu_S.\]
\end{proposition}


\begin{proof}
If $(A_S,\mu_S)$ is partially split,  then the local model $\mu_0$ from Proposition \ref{prop:local:model:algbrd} is an IM 2-form as in the statement.

Conversely, let $\mu\in\Omega^2_\imult(A_S\ltimes M_0)$ be as in the statement. We use that canonical decomposition given by the vector bundle structure
\[T^*_{S}\ka^*\simeq \ka\oplus T^*S.\]
By assumption, $\mu$ takes the following form along $S$
\[\mu|_S:A_S\to \ka\oplus T^*S,\quad \mu|_S=(l, \mu_S).\]
Let $m_t:A_S\ltimes \ka^*\to A_S\ltimes \ka^*$ denote the fiberwise multiplication by $t>0$, which is a Lie algebroid morphism. We have that
\[ \lim_{t\to 0}m^*_t \mu=\pr^*\mu_S.\]
Moreover, $\mu_t:=\frac{1}{t}m_t^*(\mu-\pr^*\mu_S)$ is a path of closed, IM 2-forms, defined on $m_{t^{-1}}(M_0)$. Then one can easily see that the following limit exists
\[ \mu^{\lin}:=\lim_{t\to 0}\mu_t,\]
(e.g., in coordinates) and it is a linear, closed, IM 2-form $\mu^{\lin}\in\Omega^2_\imult(A_S\ltimes \ka^*)$. 

If we assume that $l|_{\ka}=\id$, then, for each $t>0$, $(\mu_t|_{S})|_{\ka}=(\id,0)$, and so this holds also for the limit $\mu^{\lin}$. Thus the characterization from item (iv) of the previous subsection of the partially split condition holds.

The general case can be reduced to this using Lemma \ref{lemma:how:to:fix:along:zero:section}, which shows that the map $l|_{\ka}:\ka\to \ka$ is an $A_S$-equivariant linear isomorphism. Let $\varphi:=(l|_{\ka})^*:\ka^*\to \ka^*$ denote the transpose map. Then $(\id,\varphi):A_S\ltimes \ka^*\diffto A_S\ltimes \ka^*$ is a Lie algebroid isomorphism covering $\varphi$. So $\tilde{M}_0:=\varphi(M_0)$ is an open neighborhood of $S$ over which $A_S\ltimes \tilde{M}_0$ carries the non-degenerate IM 2-form $\tilde{\mu}:=(\id,\varphi^{-1})^*(\mu)\in \Omega^2_{\imult}(A_S\ltimes \tilde{M}_0)$. Note that $\tilde{\mu}$ also extends $\mu_S$ and by construction, for every $\beta\in \Gamma(\ka)$, we have
\[l(\beta)=\mu|_{S}(\beta)=\varphi^*\circ \tilde{\mu}|_S(\beta)=l(\tilde{\mu}(\beta)),\]
which shows that $\tilde{\mu}$ restricted to $\ka\subset (A_S\ltimes \ka^*)|_S$ is the identity map. Therefore, the above argument applied to $\tilde{\mu}$ shows that $(A_S,\mu_S)$ is partially split.
\end{proof}

The infinitesimal multiplicative version of Moser's argument, proved in Proposition \ref{prop:Moser:algbrd} (i)), gives the following result.

\begin{proposition}
\label{prop:uniqueness:algbrd}
Let $(A_S,\mu_S)$ be a first order jet of a Poisson structure with kernel $\ka:=\ker\mu_S$. Consider two closed, non-degenerate IM 2-forms $\mu_0,\mu_1\in\Omega^2_\imult(A_S\ltimes M)$, defined on some neighborhood $S\subset M\subset \ka^*$, which extend $\mu_S$
\[ i^*\mu_k=\mu_S, \quad k=0,1.\]
Then there exists a isomorphism of Lie algebroids 
\[ \Phi:A_S\ltimes M_0\diffto A_S\ltimes M_1,\quad \Phi^*(\mu_1)=\mu_0,\] 
defined between open neighborhoods $S\subset M_0,M_1\subset M$ and whose base map fixes $S$.
In particular, this base map gives a Poisson isomorphism between the induced Poisson structure 
\[ \phi:(M_0,\pi_0)\diffto (M_1,\pi_1). \]
\end{proposition}

We conclude that the local model is well-defined up to isomorphism.

\begin{corollary}[Uniqueness of local models]\label{corollary:uniqueness:model}
Given a first order jet of a Poisson structure $(A_S,\mu_S)$,
any two local models associated with distinct IM connection 1-forms are isomorphic around $S$.
\end{corollary}


\subsection{The local model as a Lie-Dirac submanifold}
\label{sec:Xu:submanifolds}

We discuss now the Lie algebroid version of the symplectic groupoid 
 \begin{equation}\label{eq:symplectic:groupoid:assoc:jet}
(T^*\G_S,\omega_{\can}+\pr^*_{\G_S}\omega_S) \tto A_S^*,
\end{equation}
and the infinitesimal version of Corollary \ref{corollary:embed:symplectic:groupoid}. This gives another perspective on the partially split condition and the geometry of the local model.
\medskip

The main result in this section can be stated as follows.

\begin{theorem}\label{theorem:Lie_Dirac}
Let $(A_S,\mu_S)$ be a first order jet of a Poisson structure. Then
\begin{enumerate}[(i)]
\item The Lie algebroid $T^*A_S\Rightarrow A_S^*$ carries the closed non-degenerate IM 2-form:
\[\mmu:=\mu_{\can}+\pr_{A_S}^*\mu_S:T^*A_S\to T^*A_S^*,\]
where $\mu_{\can}$ is the canonical IM 2-form on $T^*A_S$ (see Example \ref{ex:canonical:IM}).
\item The Poisson structure $\pi_{A_S,\mu_S}$
on $A_S^*$ corresponding to $\mmu$ is given by
\[\pi_{A_S,\mu_S}^{\sharp}=\pi_{A_S}^{\sharp}\circ (\id -\pr_{A_S}^*\mu_{S}\circ\mu_{\can}^{-1}),\]
where $\pi_{A_S}$ is the fiberwise linear Poisson structure on $A_S^*$. 
\item If $(A_S,\mu_S)$ is partially split, then for any IM connection 1-form $(L,l)$ the dual $l^*:\ka^*\hookrightarrow A^*_S$ realizes the local model $(M_0,\pi_0)$ as a Lie-Dirac submanifold of $(A_S^*,\pi_{A_S,\mu_S})$.
\item Conversely, let $(N,\pi_{N})\subset (A_S^*,\pi_{A_S,\mu_S})$ be a Lie-Dirac submanifold containing $S$ such that the corresponding Lie subalgebroid $A_{N}\subset T^*A_S^*$ satisfies $\mu^{-1}(A_N|_S)=A_S$. Then $(A_S,\mu_S)$ is partially split, and $(N,\pi_N)$ is isomorphic around $S$ to the local model $(M_0,\pi_0)$.
    \end{enumerate}
\end{theorem}

Before we prove this result, let us clarify the terminology used in the statement. First, a submanifold $i:N\hookrightarrow M$ of a Poisson manifold $(M,\pi)$ is called a \emph{Poisson-Dirac submanifold} if the Dirac structure $L_\pi=\graph(\pi^{\sharp})$ pulls back to a Dirac structure on $N$ which is given by a bivector field: $i^*L_\pi=L_{\pi_N}$. Then $\pi_N$ is a Poisson structure. We are interested in the following special type of Poisson-Dirac submanifolds.

\begin{definition}
A \textbf{Lie-Dirac submanifold} of a Poisson manifold $(M,\pi)$ is an immersed submanifold, $N\hookrightarrow M$, together with a vector bundle decomposition $T^*_NM=A_N\oplus (TN)^{\circ}$, for which $A_N$ is a Lie subalgebroid of $T^*M$.
\end{definition}

These submanifolds were introduced and studied by Xu in \cite{Xu03}, who called them ``Dirac submanifolds'', and further investigated in \cite{CrFe04} where the term ``Lie-Dirac submanifold'' was coined. In these references it is proved that if $N\subset (M,\pi)$ is a Lie-Dirac submanifold as in the definition then the following hold (see \cite[Theorem 2.3 and Remark 2.7 (iii)]{Xu03}, and also \cite[Section 8.3]{CrFe04}).
\begin{enumerate}[(a)]
\item $N$ is a Poisson-Dirac submanifold, hence has a Poisson structure $\pi_N$ which does not depend on the splitting of $T^*_NM$.
\item The restriction map $r:A_N\diffto T^*N$ is an isomorphism of Lie algebroids, where $T^*N$ has the algebroid structure induced by $\pi_N$.
\end{enumerate}

At the groupoid level, Lie-Dirac submanifolds correspond to symplectic subgroupoids. We have seen in Corollary \eqref{corollary:embed:symplectic:groupoid} that the groupoid local models arise as symplectic subgroupoids of the symplectic groupoid \eqref{eq:symplectic:groupoid:assoc:jet}, so Theorem 
\ref{theorem:Lie_Dirac} gives the infinitesimal analogue of this statement.

The theorem also has the following surprising consequence.

\begin{corollary}\label{corollary:embedding}
The inclusion of the local model:
\[i:=l^*:(M_0,\pi_0)\hookrightarrow (A_S^*,\pi_{A_S,\mu_S})\]
is a backward Dirac map. Therefore, the local model depends only on the base map of the IM connection 1-form.
\end{corollary}

We now turn to the proof of the theorem. 

\begin{proof}[Proof of (i)]
The Lie algebroid $T^*A_S\Rightarrow A_S^*$
with the IM 2-form $\mmu=\mu_{\can}+\pr_{A_S}^*\mu_S$ represent the infinitesimal counterpart of the symplectic groupoid \eqref{eq:symplectic:groupoid:assoc:jet}. Clearly, $\mmu$ is a closed IM 2-form, being the sum of the canonical closed IM 2-from $\mu_{\can}$ (see Example \ref{ex:canonical:IM}), and the pullback of the closed IM 2-form $\mu_S$ via the Lie algebroid map $\pr_{A_S}:T^*A_S\to A_S$. To show that $\mmu$ is invertible, we use a connection on $A_S$ and the dual connection on $A_S^*$, to identify \[T^*A_S\simeq T^*S\times_S A_S^*\times_S A_S,\quad T^*A_S^*\simeq T^*S\times_S A_S\times_S A_S^*.\]
Under these isomorphisms, $\mmu$ becomes
\[\mmu(\xi,\al,a)=(\xi+\mu_S(a),a,\al),\]
which is clearly invertible. 
\end{proof}

\begin{proof}[Proof of (ii)]
The Poisson structure corresponding to the closed, non-degenerate IM 2-form $\mu_{\can}$ is the linear Poisson structure $\pi_{A_S}$ on $A_S^*$, with $\pi_{A_S}^{\sharp}=\rho_{T^*A_S}\circ\mu_{\can}^{-1}$. Using the explicit description from the proof of (i), we see that the Poisson structure $\pi_{A_S,\mu_S}$ corresponding to $\mmu$ is given by the formula from item (ii)
\[\pi_{A_S,\mu_S}^{\sharp}=\rho_{T^*A_S}\circ\mmu^{-1}=\pi_{A_S}^{\sharp}\circ (\id -\pr_{A_S}^*\mu_{S}\circ\mu_{\can}^{-1}).\qedhere\]
\end{proof}

\begin{proof}[Proof of (iii)] This follows from the infinitesimal analog of Corollary \ref{corollary:embed:symplectic:groupoid}. 

\begin{lemma}\label{prop:local:model:Xu}
Let  $\theta:A_S\ltimes\ka^*\to T^*A_S$ be a partial splitting with base map $i:=l^*:\ka^*\to A_S^*$, and let $(M_0,\pi_0)$ be the corresponding local model. Then $i:M_0\hookrightarrow (A_S^*,\pi_{A_S,\mu_S})$ is a Lie-Dirac submanifold with Lie algebroid
\[A_{M_0}=\mathrm{Im}(\mmu \circ \theta)\subset T^*A_S^*|_{i(M_0)},\]
and the induced Poisson structure on $M_0$ is $\pi_0$. 
\end{lemma}

To prove this lemma, as seen in Proposition \ref{prop:local:model:algbrd}, we have the closed IM 2-form
\begin{equation}\label{eq:mu_01}
\mu_0\in\Omega^2_\imult(A_S\ltimes\ka^*), \quad \mu_0=\theta^*(\mmu)
\end{equation}
which is non-degenerate on $M_0\subset \ka^*$, and $\pi_0^{\sharp}=\rho_{\ltimes}\circ \mu_{0}^{-1}$. Over $M_0$, we have the following diagram of Lie algebroid maps
\[\begin{tikzcd}
A_S\ltimes M_0 \arrow[rr, "\mu_0"] \arrow[d, "\theta", hook]     &  & T^*M_0 \arrow[d, "\phi", hook] \\
T^*A_S \arrow[rr, "\mmu"] &  & T^*A_S^*                      
\end{tikzcd}
\]
where the horizontal arrows are isomorphisms, the algebroids on the right correspond to the Poisson structures $\pi_0$ and $\pi_{A_S,\mu_S}$, respectively, and the map $\phi$ is defined such that the diagram commutes. Clearly, $A_{M_0}=\mathrm{Im}(\phi)$. So $A_{M_0}$ is a Lie subalgebroid and $\phi:T^*M_{0}\to A_{M_0}$ is a Lie algebroid isomorphism. Next, note that the explicit form of \eqref{eq:mu_01} is
\[\mu_0(v)=i^*\big(\mmu \circ\theta(v)\big),\quad \forall\, v\in A_S\ltimes M_0,\]
or equivalently, using that $\mu_0$ is invertible and the definition of $\phi$
\begin{equation}\label{eq:restricting}
\al=i^*(\phi(\al)),\quad \forall\, \al\in T^*M_0.
\end{equation}
This shows that $A_{M_0}\cap (Ti(M_0))^{\circ}=0$, and so $A_{M_0}$ makes $i:M_0\hookrightarrow (A_S^*,\pi_{A_S,\mu_S})$ into a Lie-Dirac submanifold. We claim that the induced Poisson structure $\widetilde{\pi}_0$ on $M_0$ coincides with $\pi_0$. For this, it suffices that they induce the same Lie algebroid structure on $T^*M_{0}$. We know that both are isomorphic to $A_{M_0}$: the one corresponding to $\pi_0$ via the map $\phi:T^*M_0\diffto A_{M_0}$, and by item (b) the one corresponding to it $\widetilde{\pi}_0$ via the pullback map $i^*:A_{M_0}\diffto T^*M_0$. However, by \eqref{eq:restricting}, these maps are inverse to each other. So, indeed $\pi_0=\widetilde{\pi}_0$.   
\end{proof}

\begin{proof}[Proof of (iv)]
Let $S\subset N\subset A_S^*$ be a Lie-Dirac submanifold such that the corresponding Lie algebroid $A_N\subset T^*A_S^*$ is given above $S$ by
\[ A_N|_S =\mmu(A_S)\subset T^*A_S^*.\]
By composing the Lie algebroid map $\mmu^{-1}|_{A_N}:A_N\hookrightarrow T^*A_S$ with the restriction $T^*A_S\to A_S\ltimes \ka^*$, we obtain a Lie algebroid map $\Phi:A_N\to A_S\ltimes \ka^*$ which covers $\phi=p|_N$, where $p:A^*_S\to\ka^*$ is the projection. We claim that by shrinking $N$ we can ensure that $\Phi$ is a Lie algebroid isomorphism onto $A_S\ltimes M_0$, where $M_0=\phi(N)$ is an open neighborhood of $S$. 

The condition that $N$ is a Lie-Dirac submanifold with Lie algebroid $A_N$ gives
\[A_N\oplus (TN)^{\circ}=T^*_{N}A_S^*.\]
Under the decompositions for $T^*A_S$ and $T^*A_S^*$ from the proof of item (i), this condition along $S$ becomes
\[\{(\mu_S(a),a)\, :\, a\in A_S\}\oplus (T_SN)^{\circ}=T^*S\oplus A_S, \]
where we identify $T^*S\oplus A_S \simeq  T^*_{S}A_S^*$. Since $\ka=\ker\mu_S$ and $TS\subset T_SN$, this condition is equivalent to
\[(T_SN)^{\circ}\oplus \ka= A_S.\]
Equivalently, the differential of $\phi$ along $S$ gives a linear isomorphism
\[\d \phi:T_SN\diffto T_S\ka^*.\]
Therefore, after shrinking $N$, we may assume that $\phi$ is an open embedding. By assumption, above $S$ we have a Lie algebroid isomorphism $\Phi|_S:A_N|_S\diffto A_S$. Therefore, by shrinking $N$ further, we may assume that $\Phi$ is a Lie algebroid isomorphism $\Phi:A_N\diffto A_S\ltimes M_0$, where $M_0=\phi(N)$ is an open neighborhood of $S$ in $\ka^*$, proving our claim.

Now observe that the restriction map $\mu_N:A_N\to T^*N$ is a closed non-degenerate IM 2-form on $A_N$, therefore defining a Poisson structure $\pi_N$ on $N$. This follows since $N$ is Lie-Dirac with Lie algebroid $A_N\subset T^*A_S^*$. Using the isomorphism $\Phi$, we obtain the closed non-degenerate IM 2-form $\mu_0:=(\Phi^{-1})^*(\mu_N)$ on $A_S\ltimes M_0$. By the proof of (i), the pullback of $\mmu$ along the inclusion $A_S\hookrightarrow T^*A_S$ is $\mu_S$. This implies that the pullback of $\mu_0$ along $A_S\hookrightarrow A_S\ltimes M_0$ is $\mu_S$. Proposition \ref{prop:partially:split:algbrd:optimal} implies that $(A_S,\mu_S)$ is partially split, and Proposition \ref{prop:uniqueness:algbrd} implies that the Poisson structure $\pi_0$ on $M_0$ corresponding to $\mu_0$ is isomorphic to the local model around $S$. Since $\phi:(N,\pi_N)\diffto (M_0,\pi_0)$ is a Poisson diffeomorphism the conclusion follows.
\end{proof}

\subsection{IM connection 1-forms as couplings}
\label{sec:obstructions:partial:split}

In the case of symplectic leaves of Poisson structures the classical approach to the local model, due to Vorobjev, is via so-called couplings, which originated in the theory of symplectic fibrations. There exists a similar approach to our local model, as we now discuss.

Consider a partially split first jet $(A_S,\mu_S)$ with a fixed IM connection 1-form $(L,l)\in \Omega^1_\imult(A_S;\ka)$. We use the base map to split the short exact sequence
\[ \xymatrix{0\ar[r]& \ka\ar[r]& A_S\ar@/^/@{-->}[l]^{l}\ar[r]^{\mu_S}& T^*S\ar[r]& 0}\]
and to identify $A_S\simeq T^*S\oplus \ka$. Then we obtain the following data:
\begin{enumerate}[(i)]
    \item A linear connection $\nabla^L$ on the vector bundle $\ka\to S$, given by
    \[ \nabla^L_X\xi:=i_XL(\xi). \]
    \item A tensor $U\in \Gamma(TS\otimes T^*S\otimes \ka)$, given by
    \[ U(\alpha,X):=-i_XL(\alpha). \]
\end{enumerate}
That this is a connection and a tensor follows from the symbol equation \eqref{eq:symbol:IM:form} and the fact that $l|_{\ka}=\id$. Notice that the connection $\nabla^L$ has already appeared in Proposition \ref{prop:corollary:IM:partial:splittings}. The following result is proved in \cite[Section 5.4]{FM22}:

\begin{proposition}[\cite{FM22}]
\label{prop:structure:eqs}
The data $(\nabla^L,U)$ associated with a IM connection 1-form $(L,l)\in \Omega^1_\imult(A_S;\ka)$ satisfies the structure equations
\begin{enumerate}
\item[(S1)] the connection $\nabla^{L}$ preserves the Lie bracket $[\cdot,\cdot]_{\ka}$, i.e.,
\[\nabla^{L}_{X}[\xi,\eta]_{\ka}=[\nabla^{L}_{X}\xi,\eta]_{\ka}+[\xi,\nabla^{L}_{X}\eta]_{\ka};\]
\item[(S2)] the curvature of $\nabla^{L}$ is related to $[U,\cdot]$ as follows
\[\nabla^{L}_{\pi^{\sharp}_S(\alpha)}\nabla^{L}_{X}-\nabla^{L}_{X}\nabla^{L}_{\pi^{\sharp}_S(\alpha)}-\nabla^{L}_{[\pi^{\sharp}_S(\alpha),X]}=[U(\alpha,X),\cdot ]_{\ka};\]
\item[(S3)] $U$ satisfies the skew-symmetry condition
\begin{equation}\label{eq:U:skew} 
U(\alpha,\pi^{\sharp}_S(\beta))=- U(\beta,\pi^{\sharp}_S(\alpha)).
\end{equation}
and the ``mixed'' cocyle-type equation
\begin{align*}
\nabla^{L}_{\pi^{\sharp}_S(\alpha)}U(\beta,X)&-\nabla^L_{\pi^{\sharp}_S(\beta)}U(\alpha,X)+\nabla^L_{X}U(\alpha,\pi_S^{\sharp}(\beta))\\
&+U(\alpha,[\pi^{\sharp}_S(\beta),X])-U(\beta,[\pi_S^{\sharp}(\alpha),X])=U([\alpha,\beta]_{\pi_S},X),
\end{align*}
\end{enumerate}
for all $X\in \X^1(S)$, $\alpha,\beta\in \Omega^1(S)$, $\xi,\eta\in \Gamma(\ka)$. 
\end{proposition}

We introduce the following notation.

\begin{definition}\label{def:coupling:data}
Let $(S,\pi_S)$ be a Poisson manifold and $(\ka,[\cdot,\cdot]_{\ka})$ a Lie algebra bundle over $S$. A \textbf{coupling data} is a pair $(\nabla^L,U)$, where $\nabla^L$ is a connection on $\ka$ and $U\in \Gamma(TS\otimes T^*S\otimes \ka)$ is a tensor field satisfying the structure equations (S1), (S2) and (S3).
\end{definition}

In the rest of this section we show that the coupling data determines completely the partially split jet, and hence the local form. We start by stating the following special case of a result proved in \cite[Section 5.4]{FM22}.

\begin{proposition}
\label{prop:operators:splitting}
Let $(S,\pi_S)$ be a Poisson manifold, $(\ka,[\cdot,\cdot]_{\ka})$ a Lie algebra bundle over $S$ and $(\nabla^L,U)$ coupling data.  
The pair $(A_S,\mu_S=\mathrm{pr}_{T^*S})$ is a partially split first order jet of a Poisson structure, with IM connection 1-form $(L,l)$ given by
\begin{equation}\label{eq:L:nabla:U}
l=\mathrm{pr}_{\ka}, \quad i_XL(\alpha,\xi)=\nabla^L_{X}\xi-U(\alpha,X).
\end{equation}
Moreover, any first order jet of a Poisson structure which is partially split is isomorphic to one of this type.
\end{proposition}


With the description of a partially split first jet in terms of coupling data given by the previous propositions, we can now turn to the corresponding description of the local model. This local model is defined in the open set $M_0\subset \ka^*$ over which the IM form $\mu_0$ is invertible. This is clarified in the following result where we continue to use the notation from Proposition \ref{prop:operators:splitting}.

\begin{proposition}\label{prop:local:model:explicit}
Let $(A_S,\mu_S)$ be partially split with a fixed IM connection 1-form $(L,l)\in \Omega^1_\imult(A_S;\ka)$ and associated coupling data $(\nabla^L,U)$.
The local model is defined on the open set $M_0\subset \ka^*$ consisting of points $z\in \ka^*$ where the linear map 
\[\id + \langle  z,U\rangle :T^*_xS\to T^*_xS,\quad x=\pr_S(z),\]
is invertible. At $z\in M_0$ the Poisson structure $\pi_0$ of the local model decomposes as
\begin{equation}
    \label{eq:local:form:Poisson:coupling}
    \pi_0|_{z}=\pi_{\ka_x}+\mathrm{hor}_z\gamma_{z}\in \wedge^2\mathrm{Vert}\oplus \wedge^2\mathrm{Hor}_{\nabla^L},
\end{equation}
where the vertical component $\pi_{\ka_x}|_z$ is the linear Poisson structure on the dual of the Lie algebra $\ka_x$ and the horizontal component $\mathrm{hor}_z\gamma_{z}$ is the $\nabla^L$-horizontal lift of the bivector $\gamma_z\in \wedge^2T_{x}S$ given by
\[\gamma_z^{\sharp}=\pi_S^{\sharp}\circ \big(\id + \langle  z,U\rangle \big)^{-1}.\]
%
\end{proposition}

\begin{proof}
Proposition \ref{prop:operators:splitting} implies that the anchor of the Lie algebroid $A_S\ltimes \ka^*$ is
\[\rho_{\ltimes}|_z:T^*_xS\oplus \ka_x\to T_z\ka^*,\quad (\alpha,\xi)\mapsto \pi_{\ka_x}^{\sharp}\xi|_z+\mathrm{hor}_z(\pi_S^{\sharp}\alpha),\]
where $z\in \ka^*_x$, $\mathrm{hor}_z:T_xS\to T_z\ka^*$ is the $\nabla^L$-horizontal lift and $\xi\in \ka_x$ is thought of as an element of $T^*_{z}\ka_x^*$. Note that this decomposes the anchor into a vertical and a $\nabla^L$-horizontal component.

Now using \eqref{eq:normal:IM:form} and \eqref{eq:L:nabla:U}, we find that the IM 2-form $\mu_0$ can also be decomposed into a vertical and a $\nabla^L$-horizontal component
\[\mu_0|_z:T^*_xS\oplus \ka_x\to T^*_z\ka^*,\quad (\alpha,\xi)\mapsto
\big(\alpha+ \langle z, U(\alpha)\rangle \big)\circ \d_z\pr_S + \mathrm{vert}^*_z(\xi),\]
where $\pr_S:\ka^*\to S$ is the bundle projection, $\langle z, U(\alpha)\rangle \in T^*_xS$ and $\mathrm{vert}_z:T_z\ka^*\to \ker(\d_z\pr_S)\simeq \ka^*_x$ is the vertical projection corresponding to the connection $\nabla^L$. This shows that $\mu_0$ is invertible on the open $M_0$ consisting of elements $z\in\ka^*$ where the map $\id + \langle  z,U\rangle :T^*_xS\to T^*_xS$ is invertible. On this open set, the composition $\rho_{\ltimes}\circ\mu_0^{-1}$ yields formula \eqref{eq:local:form:Poisson:coupling} in the statement. Note that skew-symmetry of $\gamma_z$ follows from \eqref{eq:U:skew}.
\end{proof}

As a consequence of the proposition we see that the push-forward of $\pi_0$ along the projection $\pr_S:\ka^*\to S$ is $\gamma$, and hence it is tangent to the leaves of $\pi_S$.

\begin{corollary}
For any Poisson submanifold $P\subset (S,\pi_S)$, the pre-image $\pr_S^{-1}(P)\cap M_0$ is a Poisson submanifold of the local model $(M_0,\pi_0)$. 
\end{corollary}

When $P$ is a symplectic leaf of $(S,\pi_S)$, we recover a Vorobjev-type local model.

\begin{proposition}\label{prop:glueing:vorobjev}
Let $(M_0,\pi_0)$ be the local model of the first order jet $(A_S,\mu_S)$ corresponding to the IM connection 1-form $(L,l)$. For any symplectic leaf $(S_0,\omega_0)$ of $(S,\pi_S)$ the Poisson manifold $\pr_{S}^{-1}(S_0)\cap M_0$ is the Vorobjev local model corresponding to the transitive first order jet $(A_{S_0},\mu_{S_0})$ over $(S_0,\omega_0)$ given by
\[A_{S_0}:=\left(A_S/(\ker l\cap \ker \rho_S)\right)|_{S_0},\quad \mu_{S_0}([\alpha]):=\mu_S(\alpha)|_{TS_0}.\]  
\end{proposition}

\begin{proof}
Since $\pi_S|_{S_0}$ is non-degenerate, \eqref{eq:U:skew} implies that, for any $z\in \pr_S^{-1}(S_0)$, the map $\langle z, U\rangle:T_x^*S\to T_x^*S$ preserves the isotropy bundle  $\ker\pi_S^{\sharp}=(TS_0)^{\circ}$. Therefore, we have an induced map: $\langle z, U_0\rangle :T^*_xS_{0}\to T^*_xS_{0}$. Using property \eqref{eq:U:skew} again, we see that we can write $U_0(\alpha) =  i_{\pi_S^{\sharp}(\alpha)}\mathbb{F}$, for a unique $\mathbb{F}\in \Omega^2(S;\ka)$.
Hence, the Poisson structure on $\pr_{S}^{-1}(S_0)\cap M_0$ is the Poisson structure corresponding to the Vorobjev triple $(\mathbb{F}+\omega_S,\nabla^{L_0},\pi_{\ka_0})$, where $\nabla^{L_0}$ it the pullback of $\nabla^{L}$ to the subbundle $\ka^*|_{S_0}$, and 
$\pi_{\ka_0}$ is the linear Poisson structure dual to $\ka|_{S_0}$. In our language, this is the local model corresponding to the first order jet given in the statement. 
\end{proof}

\begin{remark}
From well-known properties of the Vorobjev local model (see, e.g., \cite{MarcutPhD}), we conclude that the symplectic structures on the leaves of the local model $M_0$ lying over a leaf $S_0$ of $\pi_S$ vary in an affine fashion, in the following sense.
\begin{enumerate}[(i)]
\item The Poisson structure on $\pr_S^{-1}(S_0)\cap M_0$ extends to a Dirac structure $\mathbb{L}_{S_0}$ on $\pr_S^{-1}(S_0)$;
\item If $(S_{v},\omega_v)$ is the presymplectic leaf of $\mathbb{L}_{S_0}$ through $v\in \pr_S^{-1}(S_0)$, then the connected component of $v$ in $S_v\cap M_0$ is  the symplectic leaf of $M_0$ through $v$. 
\item The fiberwise multiplication by $\lambda\neq 0$ on $\pr_S^{-1}(S_0)\to S_0$ induces an isomorphism $m_{\lambda}:S_{v}\diffto S_{\lambda v}$ and
\[m_{\lambda}^*(\omega_{\lambda v})=\lambda\omega_{v}+(1-\lambda)\omega_0.\]
\end{enumerate}
In Example \ref{example:deformation:IM:form}, we show that $\pi_0$ might not extend to a Dirac structure on $\ka^*$. 
\end{remark}

In \cite[Section 5.5]{FM22} we called an IM connection 1-form $(L,l)\in \Omega^1_{\imult}(A_S;\ka)$ \textbf{leafwise flat} if the induced vector bundle splitting $T^*S\to A_S$ is a Lie algebroid map. In terms of the coupling data, this is equivalent to (see \cite[Definition 5.14]{FM22})
\[U(\al,\pi_S^{\sharp}(\be))=0,\quad \al,\be\in T^*S.\]
On the other hand, it is easy to see that this is equivalent also to $\gamma=\pi_S$, where we used the notation of Proposition \ref{prop:local:model:explicit}. Thus, we obtain the following.
\begin{corollary}
Let $(M_0,\pi_0)$ be the local model corresponding to the first order jet $(A_S,\mu_S)$ and the IM connection 1-form $(L,l)\in \Omega^1_{\imult}(A_S;\ka)$. The restriction of the bundle projection $p:\ka^*\to S$ is a Poisson map
\[p|_{M_0}:(M_0,\pi_0)\to (S,\pi_S)\] 
if and only if $(L,l)$ is leafwise flat.
\end{corollary}

\begin{remark}[Local model in coordinates] Let $(A_S,\mu_S)$ be a partially split first jet, and fix a IM connection 1-form $(L,l)\in \Omega^1_\imult(A_S;\ka)$, which gives rise to coupling data $(\nabla^L,U)$. Choose a local chart $(V,x^i)$ for $S$ over which the bundle $\ka$ trivializes, and let $\{e_a\}$ be a frame for $\ka|_V$. Then we can write all our data as
\begin{align*}
    &\pi_S=\frac{1}{2}\pi_S^{ij}\pd{x^i}\wedge\pd{x^j} &
    &[e_a,e_b]_{\ka}=C_{ab}^c e_c,\\
    &\nabla^L_{\pd{x^i}}e_a=\Gamma^b_{ia}e_b,  &
    &U(\d x^i,\pd{x^j})=U^{ia}_j e_a.
\end{align*}
Then one checks that the structure equations become
{\small
\begin{align*}
\label{eq:partial:splitting:local}
&C_{ad}^c\Gamma^d_{ib}+C_{db}^c\Gamma^d_{ia}-C_{ab}^d\Gamma^c_{id}
=\pdd{C_{ab}^c}{x^i}\\
&\pi^{ik}_S\left(\pdd{\Gamma^b_{ja}}{x^k}-\pdd{\Gamma^b_{ka}}{x^j}+
\Gamma^d_{ja}\Gamma^b_{kd}-\Gamma^d_{ka}\Gamma^b_{jd}\right)
=U^{id}_j C^b_{da}\\
&\pi^{il}_S\left(\pdd{U^{ja}_k}{x^l}-\frac{1}{2}\pdd{U^{ja}_l}{x^k}+\Gamma^a_{ld}U^{jd}_k-U^{jd}_l\Gamma^a_{kd}\right)
+\frac{1}{2}\pdd{\pi^{il}_S}{x^k}U^{ja}_l\\
&\qquad -\pi^{jl}_S\left(\pdd{U^{ia}_k}{x^l}-\frac{1}{2}\pdd{U^{ia}_l}{x^k}+\Gamma^a_{ld}U^{id}_k-U^{id}_l\Gamma^a_{kd}\right)
-\frac{1}{2}\pdd{\pi^{jl}_S}{x^k}U^{ia}_l
=\pdd{\pi^{ij}_S}{x^l}U^{la}_k\\
&U^{ia}_k\pi^{jk}_S+U^{ja}_k\pi^{ik}_S=0
\end{align*}}

\noindent Note that for the structure equation (S3) we have used its skew-symmetric version, arising from the relation \eqref{eq:U:skew} -- the extra equation is precisely the coordinate version of \eqref{eq:U:skew}. 

The local coordinate expression associated with the Poisson structure $\pi_0$ of the local model, is the following. In coordinates $(x^i,z_a)$, where $z_a$ are fiber coordinates on $\ka^*$ defined by the basis dual to $\{e_a\}$, it has a matrix of structure functions given by the product of the block-matrices
\begin{align*}
\pi_0=&\left(
\begin{array}{cc}
\id  & 0  \\
\Gamma  &   \id
\end{array}
\right)
\left(
\begin{array}{cc}
\pi_S\cdot (\id+\langle z,U\rangle)^{-1}  & 0  \\
0  &   \pi_{\ka}
\end{array}
\right)
\left(
\begin{array}{cc}
\id  & \Gamma^t  \\
0  &   \id
\end{array}
\right)\\
=&
\left(
\begin{array}{cc}
\delta^i_j & 0  \\
\Gamma_{ib}^cz_c  &   \delta^a_b
\end{array}
\right)
\left(
\begin{array}{cc}
\pi^{ij}_S & 0  \\
0 &   C^c_{ab}z_c
\end{array}
\right)
\left(
\begin{array}{cc}
\delta^i_j+U_j^{ic}z_c  & 0  \\
0  &   \delta^a_b
\end{array}
\right)^{-1}
\left(
\begin{array}{cc}
\delta^i_j  & \Gamma_{jb}^cz_c  \\
0  &   \delta^a_b
\end{array}
\right).
\end{align*}
\end{remark}

\section{Linearization and normal form theorems}
\label{sec:linearization:algbrds}

In the previous section, we have seen that a partially split first order jet has a well-defined local model. When this is also a normal form, we use the following terminology.

\begin{definition}\label{def:linearizable:Poisson}
A Poisson manifold $(M,\pi)$ is said to be \textbf{linearizable} around a Poisson submanifold $S\subset M$, if the first jet $J^1_S\pi=(A_S,\mu_S)$ is partially split and $(M,\pi)$ is isomorphic around $S$ to the local model $(M_0,\pi_0)$.
\end{definition}

In other words, $(M,\pi)$ is linearizable around a Poisson submanifold $S$, when there are open neighborhoods $S\subset U\subset M$ and $S\subset V\subset \ka^*$, and a Poisson diffeomorphism $\varphi:(U,\pi|_U)\diffto (V,\pi_0|_V)$, such that $\varphi|_{S}=\id_S$. If $U$ and $V$ can be chosen to be saturated, then we say that $\pi$ is \textbf{invariantly linearizable} around $S$. 

In this section we will discuss conditions which imply that the local model represents a local normal form, i.e., we discuss normal form theorems.

\subsection{A linearization result}

Our aim in this subsection is to prove the following:

\begin{theorem}\label{thm:normal:form:algebroid}
Let $(M,\pi)$ be a Poisson manifold and $S\subset M$ a Poisson submanifold. Then the following are equivalent
\begin{enumerate}[(i)]
\item $(M,\pi)$ is linearizable around $S$; 
\item The Lie algebroid $T^*M$ is linearizable around $S$.
 \end{enumerate}
\end{theorem}

Let us start by recalling a few facts about linearization of Lie algebroids. First, given a Lie algebroid $A\to M$ with an invariant, embedded submanifold $S\subset M$, the restricted Lie algebroid $A_S:=A|_S$ has a canonical representation on the normal bundle $\nu(S)$ to $S$, which generalizes the Bott connection from foliation theory. It can be defined by
\[\nabla^{\nu}:\Gamma(A_S)\times \Gamma(\nu(S))\rmap \Gamma(\nu(S)),\quad \nabla^{\nu}_\al(v):=[\rho_A(\tilde{\al}),\tilde{v}]|_{S}\quad \mathrm{mod}\ TS,\]
where $\widetilde{\al}\in \Gamma(A)$ and $\widetilde{v}\in \X(M)$ are any smooth extensions of $\al$ and $v$, respectively. The \textbf{linear approximation} of $A$ around $S$ is the action Lie algebroid associated with this representation
\[(A_S\ltimes \nu(S), [\cdot,\cdot]_{\ltimes },\rho_{\ltimes }).\]

%
%
%
\begin{definition}\label{defi:linearizable:LieAlg}
We say that $A$ is \textbf{linearizable} around $S$ if there are open neighborhoods $S\subset U\subset M$ and $S\subset V\subset \nu(S)$ and an isomorphism of Lie algebroids
\[ \phi: A_S\ltimes V\diffto A|_{U},\]
which is the identity on $A_S$. 
\end{definition}

In order to obtain more geometric insight into the linear approximation, fix a tubular neighborhood whose differential along $S$ induces the identity of $\nu(S)$
\[ \varphi: \nu(S) \diffto U\subset M.\]
Fix also a vector bundle
isomorphism covering $\varphi$
\[\phi: A_S\times_S \nu(S)\diffto A|_U,\]
such that $\phi|_{A_S\times 0_S}=\id_{A_S}$. We use $\phi$ to transport the Lie algebroid structure of $A|_U$ to obtain a Lie algebroid structure on $A_S\times_S \nu(S)\to \nu(S)$, which we denote
\[(A_1,[\cdot,\cdot]_{1},\rho_1). \] 
Let $m_t$ be the multiplication by $t>0$ on the second component
\[
\vcenter{\hbox{\xymatrix{A_S\times_S \nu(S) \ar[r]^{m_t}\ar[d] & A_S\times_S \nu(S)\ar[d]  \\ \nu(S)\ar[r]^{m_t} & \nu(S)}}}
\qquad  m_t(a,v)=(a,m_t(v))=(a,tv).
\]
By pulling back the Lie algebroid structure of $A_1$ along $m_t$, we obtain Lie algebroids
\[(A_t,[\cdot,\cdot]_t,\rho_{t}), \quad t>0.\]
These Lie algebroids are on the same vector bundle $A_t:=A_S\times_S \nu(S)\to \nu(S)$, and the structure maps are uniquely determined by the condition that for $\al,\be\in \Gamma(A_S)$
\[[\overline{\al},\overline{\be}]_{t}:=m_t^*([\overline{\al},\overline{\be}]_{1}), \quad \rho_t(\overline{\al}):=(m^0_t)^*(\rho_1(\overline{\al})),\]
where the bar indicates the corresponding constant section. In other words, we have constructed a family of Lie algebroid isomorphisms
\[m_t:A_t\diffto A_1\cong A, \quad t>0.\]
Moreover, it is easy to see that this path of Lie algebroids extends smoothly at $t=0$, and in the limit one obtains precisely the action Lie algebroid
\[A_0=A_S\ltimes \nu(S), \quad \textrm{with} \quad \lim_{t\to 0}[\cdot,\cdot]_t=[\cdot,\cdot]_{\ltimes},\quad \ \lim_{t\to 0}\rho_t=\rho_{\ltimes}.\]
This follows because bundle maps $m_t:A_0\to A_0$ are Lie algebroid automorphisms of the linear approximation and this actually characterizes it.

\smallskip

Let us go back to our case, where $A=T^*M$ is the cotangent bundle of a Poisson manifold $(M,\pi)$ and $S\subset M$ is a Poisson submanifold. As we saw before, we have the adjoint representation $\nabla^{\ka}$ on the conormal bundle
\[ \nu^*(S)=(TS)^0=\ker\mu_S=\ka, \]
while the dual bundle $\nu(S)=\ka^*$ carries the dual representation $\nabla^{\ka^*}$. Next, we show that this representation is nothing else but the Bott connection.

\begin{lemma}
The two representations of $A_S$ on $\nu(S)=\ka^*$ coincide: $\nabla^{\nu}=\nabla^{\ka^*}$.
\end{lemma}

\begin{proof}
Given sections $\al\in\Gamma(A_S)$ and $\be\in \Gamma(\ka)$, denote by $\widetilde{\alpha},\widetilde{\beta}\in \Omega^1(M)$ extensions of $\al$ and $\be$ to $M$. Also, given a section $v\in \nu(S)$, let $\widetilde{v}\in \X(M)$ be a vector field whose restriction $\widetilde{v}|_S$ represents $v$.  Using the expression of the bracket on the Lie algebroid $(T^*M,[\cdot,\cdot]_{\pi},\pi^{\sharp})$, we have that
\[\Lie_{\pi^{\sharp}(\widetilde{\alpha})}(\widetilde{\beta}(\widetilde{v}))=[\widetilde{\alpha},\widetilde{\beta}]_{\pi}(\widetilde{v})+\widetilde{\beta}([\pi^{\sharp}(\widetilde{\alpha}),\widetilde{v}])+
\d\widetilde{\alpha}(\pi^{\sharp}(\widetilde{\beta}),\widetilde{v}).\]
Restricting to $S$, we have that $\pi^{\sharp}(\widetilde{\beta})|_S=\rho_{A_S}(\be)=0$. Therefore, the above formula restricts to the duality relation
\[\Lie_{\rho_{A_S}(\al)}\langle\be,v\rangle=\langle\nabla^{\ka}_\al \be,v\rangle+
\langle \beta, \nabla_\al^{\nu}v\rangle.\qedhere\]
 \end{proof}

We conclude that the linear approximation to the cotangent Lie algebroid of $(M,\pi)$ along the Poisson submanifold $S$ coincides with the action algebroid $A_S\ltimes \ka^*$ we have used to construct the local model. Moreover, it is completely determined by the first jet $J^1_S\pi$ or, equivalently, by the pair $(A_S,\mu_S)$. We can now give a proof of the main result in this subsection.

\begin{proof}[Proof of Theorem \ref{thm:normal:form:algebroid}]
Assume first that we have a Lie algebroid automorphism $\phi:A_S\ltimes V\diffto T^*U$ giving a linearization of $T^*M$ around $S$. By assumption, $\phi|_{A_S}=\id_{A_S}$. The identity map of $T^*U$ is a non-degenerate, closed IM 2-form, and therefore, so is its pullback $\mu_0:=\phi^*(\id):A_S\ltimes V\to T^*V$. We have that $i^*(\mu_0)=\mu_S$, where $i:A_S\to A_S\ltimes V$ is the inclusion. Proposition \ref{prop:partially:split:algbrd:optimal} implies that $(A_S,\mu_S)$ is partially split, and Proposition \ref{prop:uniqueness:algbrd} says that the Poisson local model is isomorphic around $S$ to $(V,\pi_V:=\rho_{\ltimes}\circ \mu_0^{-1})$. We conclude that the base map of $\phi$ is a Poisson isomorphism
\[\varphi:(V,\pi_V)\diffto (U,\pi|_U),\]
and so $\pi$ is linearizable around $S$. 
\smallskip

For the converse, assume that $(M,\pi)$ is linearizable around $S$ and consider a Poisson isomorphism $\varphi:(U,\pi|_U)\diffto (V,\pi_0|_V)$ between open neighborhoods $S\subset U\subset M$ and $S\subset V\subset \ka^*$, with $\varphi|_S=\id_S$. The pushforward along $\varphi$ is an isomorphism between the cotangent Lie algebroids
\[\varphi_*:T^*U\diffto T^*V,\quad  \varphi_*(\xi):=(\d \varphi^{-1})^*(\xi).\]
On the other hand, $\pi_0$ is constructed from the Lie algebroid local model $(A_S\ltimes V,\mu_0)$, where the IM 2-form is nondegenerate. Hence, we have a Lie algebroid isomorphism
\[\mu_0:A_S\ltimes V\diffto T^*V,\]
which extends $\mu_S$, i.e., $i^*(\mu_0)=\mu_S$, where $i:A_S\hookrightarrow A_S\ltimes V$ is the inclusion. Combining the two maps, we obtain a Lie algebroid isomorphism 
\[\phi:=(\varphi_{*})^{-1}\circ\mu_0:A_S\ltimes V\diffto T^*U,\]
which covers $\varphi^{-1}$. We still need to fix $\phi$ so that its restriction $\phi_{A_S}:=\phi|_{A_S}:A_S\diffto A_S$ is the identity. This is done in the following lemma.

\begin{lemma}
The map $\phi_{A_S}$ restricts to a bundle isomorphism $\phi_{\ka}:=\phi|_{\ka}:\ka\diffto \ka$, and we have a Lie algebroid isomorphism
 \[(\phi_{A_S},\phi_{\ka^*}):A_S\ltimes \ka^*\diffto A_S\ltimes \ka^*,\]
where $\phi_{\ka^*}:=(\phi_{\ka}^*)^{-1}:\ka^*\diffto \ka^*$.
\end{lemma}

By precomposing $\phi$ with the inverse of the isomorphism in this lemma, we obtain an isomorphism of Lie algebroids,
\[\phi':A_S\ltimes V'\diffto T^*U,\quad V'=\phi_{\ka^*}(V)\subset \ka^*\]
which satisfies $\phi'|_{A_S}=\id_{A_S}$. Hence, $T^*M$ is linearizable around $S$.

To prove the lemma, we note that $\phi_{A_S}$ covers the identity of $S$, so we have that $\rho_{A_S}\circ \phi_{A_S}=\rho_{A_S}$ and $\mu_S\circ \phi_{A_S}=\mu_S$ (since this holds for $\mu_0$ and $\phi_*$). Therefore, the restriction $\phi_{\ka}:=\phi|_{\ka}:\ka\diffto \ka$ is a vector bundle automorphism of $\ka=\ker \mu_S$, $\phi_{\ka}:=\phi|_{\ka}:\ka\diffto \ka$. Since $\phi_{A_S}$ is a Lie algebroid map, we have that
\[\phi_{\ka}(\nabla^{\ka}_{\alpha}\beta)=\phi_{A_S}([\alpha,\beta]_{A_S})=[\phi_{A_S}(\alpha),\phi_{A_S}(\beta)]_{A_S}=\nabla^{\ka}_{\phi_{A_S}(\alpha)}\phi_{\ka}(\beta).\]
This implies the analog equation for the dual representation on $\ka^*$ and the automorphism $\phi_{\ka^*}:=(\phi_{\ka}^*)^{-1}:\ka^*\diffto \ka^*$. This completes the proof of the lemma.
 \end{proof}

\begin{remark}
\label{rem:path:to:linear}
As discussed before, the linearization procedure for Lie algebroids yields a smooth family of Lie algebroid structures $A_t$ connecting $A_1\simeq  T^*M$ to the linear model $A_0=A_S\ltimes \ka^*$. Since $m_t:A_t\diffto A_1$ is a Lie algebroid isomorphism for $t>0$, we have the family of Poisson structures 
\begin{equation}\label{wrong:path:to:linear} 
\pi_t=\rho_t\circ (m_t^*(\mu_1))^{-1},
\end{equation}
where $\mu_1:A_1\diffto  T^*M$ is the identification we started with. However, in general, $\lim_{t\to 0}\pi_t$ may not exist (even if local models exists!), and if this limit exists, it may not be isomorphic to the local model (see Section \ref{ex:family with no limit}).

It would be interesting to understand better the following.
\begin{question}
If the first order jet of $(M,\pi)$ at a Poisson submanifold $S\subset M$ is partially split, can $\pi$ be connected to a local model by a smooth path of Poisson structure $\pi_t$, $t\in [0,1]$, defined around $S$, with constant first order jet along $S$? 
\end{question}

A possible candidate can be constructed as follows. Consider the path $A_t$ of Lie algebroid structures which connects $T^*M$ to the linearization $A_0$. Then the path of bivectors $\tilde{\pi}_t=\rho_t\circ \mu_t^{-1}$, where 
\begin{equation}\label{path:to:linear}
\mu_t:=m_t^*\big(\mathrm{pr}^*\mu_S+t^{-1}\, (\mu-\mathrm{pr}^*\mu_S)\big),
\end{equation}
has constant first order jet along $S$ and the limit at $t=0$ exists. However, these bivectors may not be Poisson since $\mathrm{pr}^*\mu_S$ may fail to be a closed IM 2-form. The problem is that the projection $\mathrm{pr}:A_1\to A_S$ is usually not a Lie algebroid morphism.

The question has positive answers in the following situations: 
\begin{itemize}
\item[-] $M=S\times N$ is a product of a Poisson manifold $(S,\pi_S)$ and a Poisson manifold $(N,\pi_N)$ the a zero $\pi_N(x_0)=0$; 
\item[-] $S$ is a symplectic leaf (see \cite{Vorobjev01,Vorobjev05,MarcutPhD}) -- instead of $\mathrm{pr}^*\mu_S$, one uses in formula \eqref{path:to:linear} the IM 2-form corresponding to the closed de Rham 2-form $r^*(\pi_S^{-1})$, where $r:U\to S$ is the retraction of a tubular neighborhood;
\item[-] there exists a Lie algebroid map $p: T^*U \to A_S$, defined on some open set $S\subset U$, such that $p|_{A_S}=\id_{A_S}$ -- this case includes the previous two;
\item[-] $\pi$ is linearizable around $S$.
\end{itemize}
\end{remark}

\subsection{The normal form theorem} 
\label{sec:main:thm}
The following result implies Theorem \ref{thm:main:two} in the Introduction.

\begin{theorem}[Normal form]
\label{Normal:form:theorem}
Let $(M,\pi)$ be a Poisson manifold and $S\subset M$ a Poisson submanifold. If $T^*_SM$ is integrable by a compact, Hausdorff, Lie groupoid whose target fibers have trivial 2nd de Rham cohomology, then $\pi$ is invariantly linearizable around $S$.
\end{theorem}

This result is a generalization to Poisson submanifolds of the normal form theorem around leaves from \cite{CrMa12,Marcut14}, which in turn generalizes Conn's linearization theorem around points \cite{Conn85}. Conn's original proof uses the Nash-Moser fast convergence method to build a linearization map. Later, a geometric proof was obtained by Crainic and Fernandes \cite{CrFe11}. Both methods have been extended to symplectic leaves in \cite{Marcut14} and \cite{CrMa12}, respectively, and the results in these papers can be adapted to the general setting of Poisson submanifolds, to give two proofs of Theorem \ref{Normal:form:theorem}. We will explain these two approaches in the following subsection.

\subsection{The geometric method}

With the geometric method, developed in \cite{CrFe11} and extended in \cite{CrMa12}, we obtain a version of Theorem \ref{Normal:form:theorem} under a stronger assumption.
\begin{proof}[Proof of Theorem \ref{Normal:form:theorem}, with an extra assumption]
We assume that the groupoid integrating $T^*_SM$ from the statement has, in addition, 1-connected target fibers. Under these assumptions, we claim that: 
\begin{itemize} 
\item there is a saturated neighborhood $U\subset M$ of $S$ such that $T^*U\subset T^*M$ is integrable by a Hausdorff, target-proper, Lie groupoid $\G\tto U$.
\end{itemize}
Since Hausdorff, target-proper, groupoids are invariantly linearizable around invariant submanifolds, it then follows that the Lie algebroid $T^*U$ (and hence $T^*M$) is invariantly linearizable around $S$, so one can apply Theorem \ref{thm:normal:form:algebroid}.

It remains to prove the claim. When $S$ is symplectic, i.e., when $T^*_SM$ is transitive, this claim is proved in Section 5 of \cite{CrMa12}. However, one observes that the proof of the claim given in \cite{CrMa12} does not use that $S$ is a symplectic leaf, and remains valid in the non-transitive case. 
\end{proof}

\subsection{The analytic method}

For the proof of Theorem \ref{Normal:form:theorem}, we will use the following result, obtained using the Nash-Moser method. 

\begin{theorem}[Rigidity \cite{Marcut14}]
\label{thm:rigidity}
Let $(M,\pi)$ be a Poisson manifold and $S\subset M$ a compact Poisson submanifold. If $T^*M$ is integrable by a Hausdorff Lie groupoid whose target fibers are compact and have trivial 2nd de Rham cohomology, then any Poisson structure $\pi'$ with $J^1_S\pi=J^1_S\pi'$ is locally isomorphic to $\pi$ around $S$.
\end{theorem}

%
%
%
%

\begin{proof}[Proof of Theorem \ref{Normal:form:theorem}]
Let $\G_S\tto S$ be a Lie groupoid integrating $A_S:=T^*_SM$ which is compact, Hausdorff and whose target fibers have trivial second de Rham cohomology. Since $\G_S$ is proper and Hausdorff, by Theorem \ref{thm:proper}, the bundle of ideals $\ka:=(TS)^{\circ}\subset A_S$ is partially split for $\G_S$, and so also for $A_S$. After choosing a partial splitting and a tubular neighborhood $\phi:\nu(S) \diffto M_0\subset M$ of $S$, we obtain the Poisson structure of the local model $\pi_0\in \X^2(M_0)$, which satisfies $J^1_S\pi_0=J^1_S\pi$. 

The algebroid $T^*M_0$ is isomorphic to the action algebroid $A_S\ltimes \ka^*$.  Since $A_S$ integrates to $\G_S\tto S$, and the action of $A_S$ on the bundle of ideals $\ka$ integrates to the action by conjugation of the target-connected Lie groupoid $\G_S$ (see, e.g., \cite[Lemma B.1]{Marcut14}), it follows that the action algebroid $A_S\ltimes \ka^*$ also integrates to the action groupoid $\G_S\ltimes \ka^*\tto \ka^*$. 

We conclude that $T^*M_0$ integrates to  a Hausdorff Lie groupoid whose target fibers are compact and have trivial 2nd de Rham cohomology. Hence, by Theorem \ref{thm:rigidity}, there exists a Poisson diffeomorphism $\phi:(U,\pi_0)\diffto (V,\pi)$, between neighborhoods $U$ and $V$ of $S$.

Finally, since the action groupoid $\G_S\ltimes \ka^*\tto\ka^*$ is target-proper, it follows that $S$ has a basis of open neighborhoods in $\ka^*$ that are invariant (see, e.g.,\cite[Lemma A.1]{Marcut14}). So we may assume that $U$ is invariant. Since all the leaves of $\pi_0|_U$ are compact, so are also those of $\pi|_V$. This implies that $V$ is invariant as well. Hence, $\pi$ is invariantly linearizable around $S$.
\end{proof}

From the proof, we extract the following result. 

\begin{corollary}
\label{cor:integrability}
Let $(M,\pi)$ be a Poisson manifold and $S\subset M$ a Poisson submanifold. If the hypotheses of Theorem \ref{Normal:form:theorem} are satisfied, then a neighborhood of $S$ in $M$ is integrable by a target-proper, Hausdorff Lie groupoid, whose target-fibers have trivial second de Rham cohomology. 
\end{corollary}

So we may apply \cite[Theorem 2]{Marcut14} and obtain the following, which implies Corollary \ref{thm:main:three} in the Introduction (for a complete version of (ii), see \cite{Marcut14}).

\begin{corollary}[Rigidity]
\label{thm:main:three:end}
Let $(M,\pi)$ be a Poisson manifold and $S\subset M$ a Poisson submanifold. If the hypotheses of Theorem \ref{Normal:form:theorem} are satisfied, then
\begin{enumerate}[(i)]
\item For any Poisson structure $\pi'\in \Pi(M,S)$, with $J^1_S\pi=J^1_S\pi'$, there is a Poisson diffeomorphism $\phi:(U,\pi)\diffto (U',\pi')$ between open neighborhoods $U$ and $U'$ of $S$, with $\phi|_S=\id_S$;
\item There is a neighborhood of $\pi$ in the space of all Poisson structures on $M$, $\mathcal{V}\subset \Pi(M)$, such that for every $\pi'\in \mathcal{V}$ there is a Poisson diffeomorphism $\phi:(U,\pi)\diffto (U',\pi')$ between open neighborhoods $U$ and $U'$ of $S$. 
\end{enumerate}
\end{corollary}

%
%


\section{Examples of first order jets and applications}
\label{sec:examples}

In this section we illustrate the constructions and results of the previous sections with several examples. We follow the notation introduced there.

\subsection{Products}
\label{ex:family with no limit}
\label{ex:product}

Consider a product Poisson structure
\begin{equation}\label{eq:product}
(M,\pi):=(S\times \mathbb{R}^n,\pi_S+\gamma),
\end{equation}
where $\pi_S$ is a Poisson structure on $S$ and $\gamma$ is a Poisson structure on $\mathbb{R}^n$ vanishing at $0$. We consider the first order jet $(A_S,\mu_S)$ of $\pi$ along Poisson submanifold $S=S\times\{0\}$, which can be identified as follows. Denote by $\gg=T^*_0\R^n$ the isotropy Lie algebra of $\gamma$ at $0$. We have the trivial Lie algebra bundle $\ka=S\times\gg\to S$ and 
\[ A_S=T^*_SM=T^*S\times \gg\] 
is the product Lie algebroid, with closed IM form the projection $\mu_S=\pr_{T^*S}$. We call the pair $(A_S,\mu_S)$ a \textbf{trivial product jet}.

The trivial product jet $(A_S,\mu_S)$ is partially split with a canonical choice of an IM connection 1-form. This follows by observing that we have canonical isomorphisms
\[ T^*A_S=T^*(T^*S)\times (\gg\ltimes \gg^*), \quad A_S\ltimes\ka^*=T^*S\times (\gg\ltimes\gg^*),\]
in terms of which the projection
\[ p:T^*(T^*S)\times (\gg\ltimes \gg^*)\to T^*S\times (\gg\ltimes \gg^*). \]
is the bundle projection in the first factor. Hence, we have the partial splitting
\[ \theta:T^*S\times (\gg\ltimes \gg^*)\to T^*(T^*S)\times (\gg\ltimes \gg^*)\oplus\ka\oplus\ka^*, \quad (\al,v,\xi)\mapsto (0_\al,v,\xi). \]
The corresponding linear IM 2-form $\mu^{\lin}=\theta^*\mu_\can\in\Omega^2_\imult(T^*S\times (\gg\ltimes \gg^*))$ is
\[ \mu^\lin:T^*S\times (\gg\ltimes \gg^*)\to T^*(S\times \gg^*)=T^*S\times (\gg\ltimes \gg^*),\quad  (\al,v,\xi)\mapsto (0,v,\xi),\]
so the closed IM 2-form \eqref{eq:normal:IM:form} of the local model is just the identity map
\[ \mu_0=\pr^*\mu_S+\theta^*\mu_\can=\id. \]
The resulting local model Poisson manifold is the product
\[(M_0,\pi_0):=(S\times \mathbb{R}^n,\pi_S+\gamma_{\mathrm{lin}}).\]
where $\gamma_{\mathrm{lin}}$ is the linear Poisson structure on $\R^n=\gg^*$, i.e.\ the linear approximation to $\gamma$ at $0$.

Following Remark \ref{rem:path:to:linear}, let us look at the path of bivectors \eqref{wrong:path:to:linear}. Using the obvious identification $A_S\times \mathbb{R}^n\simeq T^*M$, we find
\[\pi_t=\pi_S+m_t^*(\gamma)=\pi_S+\frac{1}{t}\gamma_{\mathrm{lin}}+\gamma_2+t\,\gamma_3+\ldots,\]
where each $\gamma_k$ a bivector whose coefficients are homogeneous polynomials of degree $k$. Thus $\lim_{t\to 0}\pi_t$ exists if and only if $\gamma_{\mathrm{lin}}=0$, i.e., if and only if the isotropy Lie algebra of $\gamma$ at 0 is abelian, in which case we obtain
\[\lim_{t\to 0}\pi_t=\pi_S+\gamma_2\not=\pi_0.\] 

On the other hand, in this example the projection $\mathrm{pr}:A_1\to A_S$ is a Lie algebroid morphism so $\mathrm{pr}^*\mu_S$ is a closed IM 2-form. Hence, applying \eqref{path:to:linear} we obtain a smooth path of Poisson structures connecting $\pi$ to the local model $\pi_0$
\[\tilde{\pi}_t:=\pi_S+t\, m_t^*(\gamma).\]

\subsection{Jets over zero Poisson structures}

As we saw in Example \ref{example:aff1}, first order jets $(A_S,\mu_S)$ with $\pi_S=0$ may fail be partially split. The following result, which is a direct consequence of \cite[Proposition 6.1]{FM22} and of Proposition \ref{prop:local:model:explicit}, shows that the obstructions are precisely those found in Proposition \ref{prop:corollary:IM:partial:splittings}.

\begin{proposition}\label{prop:pi_S=0}
A first order jet of a Poisson structure $(A_S,\mu_S)$ with $\pi_S=0$ is partially split if and only if the following two conditions hold:
\begin{enumerate}[(i)]
\item $\ka=\ker\mu_S$ is a Lie algebra bundle;
\item there exists a splitting $A_S\simeq T^*S\oplus \ka$, for which the fiberwise Lie bracket is a direct product with abelian first component. 
\end{enumerate}
Moreover, if these conditions hold, then the associated linear model is isomorphic to the bundle of linear Poisson structure $(\ka^*,\pi_{\ka})$.
\end{proposition}

\begin{remark}\label{rem:bundle of Lie algebras:Poisson:Dirac}
Let $(\ka,[\cdot,\cdot]_{\ka})\to S$ be a bundle of Lie algebras. The fiberwise linear Poisson structure $\pi_{\ka}$ on $\ka^*$ has the zero section as a Poisson submanifold with $\pi_S=0$. The corresponding first order jet is $(A_S=T^*S\oplus \ka, \mu_S=\pr_{T^*S})$, where sections of $T^*S$ commute with all other sections, and on $\ka$ we have the bracket $[\cdot,\cdot]_{\ka}$. By Proposition \ref{prop:pi_S=0}, this first order jet is partially split if and only if $\ka$ is a Lie algebra bundle. In this case, the local model is $\pi_0=\pi_{\ka}$.

However, even if no partial splitting exits, we have a Poisson embedding
\[ i:(\ka^*,\pi_{\ka})\hookrightarrow (A_S^*=TS\oplus \ka^*,\pi_{A_S}),\quad \xi\mapsto (0,\xi),\]
where  $\pi_{A_S}=\pi_{A_S,\mu_S}$. When $\ka$ is locally trivial this inclusion realizes $\ka^*$ as a Lie-Dirac submanifold (see Theorem \ref{theorem:Lie_Dirac}).
\end{remark}

\begin{example}\label{example:log:not:PD}
Consider the first order jet of the linear Poisson structure 
\[ (\R^2,\pi=x\,\pd{x}\wedge\pd{y}) \] 
along the Poisson submanifold $S=\{x=0\}$. In Example \ref{example:aff1} we gave a groupoid argument to see that it is not partially split. Proposition \ref{prop:pi_S=0} gives now an infinitesimal argument and one can also see this by applying Corollary \ref{corollary:embedding}: it is easy to see that the pullback of the Dirac structure corresponding to $\pi_{A_S,\mu_S}$ via any linear splitting $i:\ka^*\hookrightarrow A_S^*$ of the projection is not a smooth Dirac structure. 

Similarly, we can use Proposition \ref{prop:pi_S=0} to show that  the  first order jet of the linear Poisson structure 
\[ (\R^3,\pi=z\,\pd{x}\wedge\pd{y}) \] 
along the Poisson submanifold $S=\{z=0\}$ is not partially split. However, in this example, 
the pullback of the Dirac structure corresponding to $\pi_{A_S}=\pi_{A_S,\mu_S}$ via any linear splitting $i:\ka^*\hookrightarrow A_S^*$ of the projection is smooth, and
 corresponds to the zero Poisson structure $\pi_0=0$ on $\ka^*$. But note that $\pi_0$ is not a first order local model around $S$ because $J^1_S\pi_0=0\neq J^1_S\pi$. 
\end{example}

\subsection{Transitive jets}
\label{ex:transitive}
Let $(A_S,\mu_S)$ be a first order jet with $A_S\Ato S$ transitive. Such jets arise at symplectic leaves. That $A_S$ is transitive is equivalent to the induced Poisson structure $\pi_S$ being non-degenerate, and so its inverse is a symplectic structure $\omega$ on $S$. The IM 2-form is determined by $\omega$ via the relation $\mu_S(\al)=-i_{\rho(\al)}\omega$. Also, we have that $\ka=\ker\rho$. 

If we choose a splitting of the anchor
\[ \xymatrix{
0\ar[r] & \ka \ar[r] & A_S \ar[r]^{\rho} \ar@/^/@{-->}[l]^l & TS\ar[r]  \ar@/^/@{-->}[l]^{\tau} & 0}
\]
we can define a linear operator $L:\Gamma(A_S)\to\Omega^1(S,\ka)$ by setting
\[ i_XL(\al):=l([\tau(X),\al]). \]
One checks easily that the pair $(L,l)$ satisfies \eqref{eq:compatibility:IM:E:form}, so it is a $\ka$-valued, IM 1-form, with $l|_\ka=\id$. Hence, $L$ is a IM connection 1-form so that $(A_S,\mu_S)$ is partially split. Conversely, given any IM connection 1-form $(L,l)\in\Omega^1_\imult(A_S,\ka)$, the bundle map $l:A_S\to \ka$ determines a splitting of the anchor. These two constructions are inverse to each other (see \cite[Subsection 6.3]{FM22}). 

We conclude that transitive jets are partially split and a choice of IM connection 1-form is equivalent to a choice of splitting of its anchor. The resulting local model coincides with Vorobjev's local model for Poisson structures around symplectic leaves \cite{Vorobjev05}. 

\subsection{Cartan connections}
\label{sec:Cartan}
A {\bf Cartan connection} on a Lie algebroid $A_S$ is a connection $\nabla$ whose basic curvature vanishes identically
\[ R^{\bas}_\nabla\equiv 0. \] 
Hence, an immediate consequence of Proposition \ref{prop:Cartan:connection:splitting} is the following.

\begin{corollary}
Let $(A_S,\mu_S)$ be a first order jet equipped with a Cartan connection $\nabla$ and a splitting $l:A_S\to \ka$ which is $\overline{\nabla}$-invariant. Then $(A_S,\mu_S)$ is partially split.
\end{corollary}

\begin{example}[Action Lie algebroids]
Let $A_S=\gg\ltimes S\Ato S$ be the action Lie algebroid associated with a Lie algebra action $\rho:\gg\to\X(S)$. The canonical flat connection $\nabla$ on $A_S$ has vanishing basic curvature, hence it is a Cartan connection.
Given a bundle of ideals $\ka\subset A_S$ a splitting $l:A_S\to \ka$ is $\overline{\nabla}$-invariant if and only if it is $\gg$-equivariant
\[ l([v,w]_\gg)=[v,l(w)]_{\gg\ltimes S}, \]
for all $v,w\in\gg$ (here we identify elements of $\gg$ with constant sections of $A_S$). Hence, we recover the infinitesimal version of the example is Subsection \ref{ex:action}: a first order jet on an action algebroid admitting a $\gg$-equivariant splitting is partially split.

One class of examples is obtained when $\gg$ is a Lie algebra of compact type. Then $\gg$ admits an $\ad$-invariant scalar product $\langle\cdot,\cdot\rangle$, which yields a $\overline{\nabla}$-invariant fiberwise metric $\eta$ on $A_S$. Hence any bundle of ideals on $A_S=\gg\ltimes S$ admits a $\gg$-equivariant splitting, so it is partially split.
\end{example}

\begin{example}[Poisson submanifolds of $\gg^*$]
Let us consider the case where $S\subset\gg^*$ is a closed $\ad^*$-invariant submanifold. Then $S$ is a Poisson submanifold of the linear Poisson manifold $(\gg^*,\pi_\lin)$ and its first order jet $(A_S,\mu_S)$ is the action algebroid $A_S=\gg\ltimes S\Ato S$ obtained by restriction of the coadjoint action. Hence, we can use the results of the previous example.

Let us assume then that $\gg$ is a Lie algebra of compact type and choose an $\ad$-invariant inner product. Then $\ka^*=(\ker\mu_S)^*$ is the normal bundle of $S$ in $\gg^*$
\[ 
\ka^*=\{(w,x):x\in S, w\in (T_xS)^{\perp}\}\subset \gg^*\times\gg^*. 
\]
The Riemannian exponential is the $\gg$-equivariant map
\[\exp:\ka^*\to \gg^*, \quad (w,x)\mapsto w+x.\]
and so it induces a map of Lie algebroids 
\[ A_0=\gg\ltimes\ka^*\to T^*\gg^*=\gg\ltimes\gg^*. \]
This map is an isomorphism around $S$. On the other hand, the inner product determines a $\gg$-equivariant splitting giving a closed IM form $\mu_0$ on $A_0$
\[ \mu_0=\pr^*\mu_S+\mu^\lin.\]
Under the identification $A_0\simeq T^*\gg^*$ the two terms of $\mu_0$ correspond to the orthogonal projections on $T^*S$ and $(T^*S)^\perp$, respectively,
\[ 
\pr^*\mu_S|_\xi=\pr_{T^*_\xi S}: T_\xi\gg^*\to T^*_\xi S,\qquad
\mu^\lin|_\xi=\pr_{(T^*_\xi S)^\perp}: T_\xi\gg^*\to (T^*_\xi S)^\perp.
\]
Therefore, $\mu_0=\id$. 

We conclude that for a compact Lie algebra $\gg$, the Poisson manifold $\gg^*$ is linearizable around any closed Poisson submanifold $S\subset\gg^*$. This holds already at the groupoid level, as was shown in Subsection \ref{ex:action}. The result agrees with Theorem \ref{thm:normal:form:algebroid} since the Lie algebroid $T^*\gg^*$ is linearizable around $S$ if $\gg$ is compact. On the other hand, Theorem \ref{Normal:form:theorem} can only be applied when $\gg$ is compact, semi-simple, since only in this case $T^*_S\gg^*$ is integrable by a compact groupoid with target fibers having vanishing 2nd de Rham cohomology.
\end{example}

\subsection{Jets of principal type}
\label{ex:principal:type:bundle:local:model}
\label{ex:principal:type:Lie:algebroids}

We discuss the infinitesimal counterpart of the over-symplectic groupoids of principal type from Subsections \ref{example:os:princ:type} and \ref{example:os:princ:type:part:split}. Let $B\Ato S$ be a transitive Lie algebroid over a Poisson manifold $(S,\pi_S)$. Consider the Lie algebroid fibre product of $B$ with the cotangent Lie algebroid of $\pi_S$, i.e.,
\[A_S:=T^*S\times_{TS}B:=\{(\alpha,v)\, :\, \pi_S^{\sharp}(\alpha)=\rho_B(v)\},\]
where the structure is such that the inclusion in the product $A_S\hookrightarrow T^*S\times B$ is a Lie algebroid morphism. The projection $\mu_S:=\pr_{T^*S}:A_S\to T^*S$ is a closed IM 2-form, which is surjective, because $B$ is transitive. The resulting pair $(A_S,\mu_S)$ will be called the  \textbf{first order jet of principal type} associated to $(S,\pi_S)$ and $B$. 
 
Note that $\ka:=\ker\mu_{S}$ can be identified with $\ker\rho_B$ via $\pr_B$. As discussed in Section \ref{ex:transitive}, a splitting $l_B:B\to\ka$ of $\rho_B$ determines a IM connection 1-form form $(L_B,l_B)\in\Omega^1_\imult(B,\ka)$. Pulling this back to $A_S$, we obtain an IM connection 1-form 
\[(L:=L_B\circ\pr_B, l:=l_B\circ \mathrm{pr}_B)\in \Omega^1_{\imult}(A_S,\ka),\]
which also satisfies $l|_{\ka}=\id_{\ka}$. So first order jets of principal type are partially split. 

The associated coupling data is described as follows. The splitting $l_B$ gives an identification $B\simeq TS\oplus \ka$, where the anchor becomes $\pr_{TS}$ and the bracket
\begin{equation}\label{eq:bracket:transitive}
[(X,\xi),(Y,\eta)]_{B}=([X,Y],\Omega(X,Y)+\nabla^{B}_{X}\eta-\nabla^{B}_{Y}\xi+[\xi,\eta]_{\ka}),
\end{equation}
for all $X,Y\in \X^1(S)$, $\xi,\eta\in \Gamma(\ka)$. Here,
\begin{enumerate}
\item[-] $\Omega$ is $C^{\infty}(S)$-bilinear, so that $\Omega\in \Omega^2(S;\ka)$;
\item[-] $\nabla^{B}$ is a connection on $\ka$ preserving $[\cdot,\cdot]_\ka$ with curvature $R^{\nabla^B}=\ad(\Omega)$.
\end{enumerate}
Then, one finds that
\[i_XL(\alpha,\xi)=\nabla^B_X(\xi)+\Omega(X,\pi^{\sharp}_S(\alpha)),\]
and the coupling data (see Definition \ref{def:coupling:data}) is given by
\[
\nabla^{L}=\nabla^{B}, \qquad  U(\alpha,X)=\Omega(\pi_S^{\sharp}(\alpha),X).
\]

The local model of a first order jet of principal type $(A_S,\mu_S)$ for which the transitive Lie algebroid $B$ is integrable has a nice description: it is a Poisson geometric version of the ``symplectic induction'' construction (see e.g.\ \cite{GuillStern90}). For this, identify $B$ with the Atiyah algebroid of a principal $G$-bundle $P$
\[B\simeq TP/G.\]
Fix a principal connection $\theta\in \Omega^1(P)\otimes \mathfrak{g}$. Using the associated 1-form on $P\times \mathfrak{g}^*$
\[\widetilde{\theta}\in \Omega^1(P\times \mathfrak{g}^*), \ \ \widetilde{\theta}_{(x,\xi)}:=\pr_P^*\langle \theta_x| \xi\rangle,\]
we build a Dirac structure on $P\times \gg^{*}$ by setting
\begin{equation}\label{eq:Dirac:structure}
\L_{\theta}:=e^{-\d\tilde{\theta}}\left(\pr_S^{!}\L_{\pi_S}\right)=\{v+\pr_S^*(\xi)-i_v\d\widetilde{\theta} \, :\, \d\pr_S(v)=\pi_S^{\sharp}(\xi)\}.
\end{equation}
This Dirac structure is invariant under the diagonal action of $G$ and corresponds to a Poisson structure at points in $P\times \{0\}$. Hence, there is a $G$-invariant open set 
\[ P\times \{0\}\subset U\subset P\times \gg^*\] 
on which $\L_{\theta}$ is the graph of a Poisson structure $\pi_{\theta}\in \X^{2}(U)$. The $G$-action on $(U,\pi_U)$ is proper and free and the local model is the quotient Poisson manifold
\[(M_0,\pi_0):=(U,\pi_{\theta})/G, \quad M_0\subset \ka^*:=P\times_{G}\gg^*.\]
Note that the action of $G$ on $(U,\pi_{\theta})$ is Hamiltonian with $G$-equivariant moment map $\pr_{\gg^*}:U\to \gg^*$. 

\begin{example}[Local model around the Marsden-Weinstein reduction] Given a proper and free $G$-Hamiltonian action on a Poisson manifold $(X,\pi_X)$ with equivariant moment map $\mu:X\to \gg^*$, we have the Marsden-Weinstein reduction 
\[ S:=X/\!\!/G=\mu^{-1}(0)/G. \]
This is a Poisson submanifold $(S,\pi_S)$ of the Poisson quotient
\[ (M,\pi_M):=(X,\pi_{X})/G. \]
If we equip the principal $G$-bundle $P:=\mu^{-1}(0)\to S$ with a connection $\theta$, we have the local model above 
\[ (U,\pi_\theta)\subset (P\times\gg^*,L_{\theta}).\]
The following is proven in \cite{FrMaMoment}.

\begin{theorem}
Around $P=\mu^{-1}(0)$, the $G$-Hamiltonian spaces 
\[(X,\pi_X,\mu) \quad \textrm{and}\quad (U,\pi_{\theta},\pr_{\gg^*})\] 
are isomorphic. Hence, the quotient Poisson manifold $(M,\pi_M):=(X,\pi_{X})/G$ is linearizable around the Marsden-Weinstein quotient $(S=\mu^{-1}(0)/G,\pi_S)$.
\end{theorem}


\begin{remark}
In the symplectic case $(X,\pi_X=(\omega_X)^{-1})$, the Marsden-Weinstein reduction is the usual symplectic quotient $(S,\pi_S=(\omega_S)^{-1})$. In this case, the Dirac structure $\L_\theta$ on $P\times\gg^*$ is the graph of the closed 2-form $\omega_\theta:=\pr_S^*\omega_S-\d\tilde{\theta}$, which is the classical ``coupling construction" due to  Guillemin, Sternberg and Weinstein (see, e.g., \cite{GuillStern90}). It is a standard result in symplectic geometry that $(U,\omega_\theta)$ provides a local $G$-equivariant model for $(X,\omega_X)$ around $\mu^{-1}(0)$ (see, e.g., \cite[Proposition 5.2]{JeffKir95} or \cite[Theorem 6.1]{Meinrenken_notes}).
\end{remark}
\end{example}

Finally, we recall the following result from \cite[Subsection 6.6]{FM22}, which gives a large class of first jets that are of principal type, hence partially split.

\begin{proposition}[\cite{FM22}]\label{prop:isotropy:rigid:implies:linearization}
Let $(A_S,\mu_S)$ be a first jet of a Poisson structure such that $\ka=\ker\mu_S$ is a Lie algebra bundle, whose typical fiber $(\gg,[\cdot,\cdot])$ satisfies
\begin{equation}\label{eq:cohomology:Lie:alg}
H^0(\gg,\gg)=H^1(\gg,\gg)=0.
\end{equation}
Then $(A_S,\mu_S)$ is of principal type, with transitive Lie algebroid the Atiyah Lie algebroid of the principal $\Aut(\gg,[\cdot,\cdot])$-bundle of $\gg$-frames
\[P=\{\varphi: \gg \diffto \ka_x\, :\, x\in S,\, \varphi \, \textrm{is a Lie algebra isomorphism}\}.\]
In particular, the assumptions hold if $\ka$ is a bundle of semi-simple Lie algebras.
\end{proposition}

\subsection{Codimension one}
We now consider in detail the case of first order jets $(A_S,\mu_S)$ in codimension one, i.e., such that $\ker\mu_S$ is a line bundle. We will use the results of \cite[Subsection 6.8]{FM22} concerning IM connections for bundles of ideals of rank one.

If we choose a splitting $l_0:A_S\to\ka$ of the short exact sequence of $\mu_S$
\[ \xymatrix{
0\ar[r] & \ka \ar[r] & A_S \ar[r]^{\mu_S} \ar@/^/@{-->}[l]^{l_0} & T^*S\ar[r]  \ar@/^/@{-->}[l]^{\tau_0} & 0}
\]
we obtain a flat $T^*S$-connection $\nabla^{\ka}$ on the line bundle $\ka$
\[ \nabla^{\ka}_\al\xi:=[\tau_0(\al),\xi],\]
and an isomorphism of vector bundles
\[ A_S\simeq T^*S\oplus \ka. \]
Under this identification the IM-form becomes the projection $\mu_S=\pr_{T^*S}$, while the anchor and the Lie bracket are given by
\begin{align*}
    &\rho_{A_S}(\al,\xi)=\pi_S^\sharp(\al),\notag\\
    &[(\al,\xi),(\be,\eta)]_{A_S}=([\al,\be]_{\pi_S},\lambda_0(\al,\be)+\nabla^\ka_\al \eta-\nabla^\ka_\be \xi).
\end{align*} 
Notice that:
\begin{itemize}
   \item The flat $T^*S$-connection $\nabla^{\ka}$ is independent of the choice of splitting of $A_S$;
    \item The line bundle $\ka$ is canonically a representation of $T^*S$, so it has a characteristic class (see, e.g., Section 11.1 in \cite{CFM21})
    \[ c_1(\ka)\in H^1_{\pi_S}(S).\]
    In the sequel we will assume, for simplicity, that $\ka$ is orientable. If we fix a trivialization $\ka\simeq S\times \R$, then
    \[\nabla^{\ka}_{\alpha}(f)=\Lie_{\pi_S^{\sharp}(\alpha)}f+i_{V}(\alpha)f,\]
    for a unique Poisson vector field $V\in \X^1(S)$ representing the class $c_1(\ka)$. This vector field is independent of the choice of splitting of $A_S$. If we change the trivialization $\ka\simeq S\times \R$ by multiplying with a non-zero function $h$, $V$ changes to $V+\d_{\pi_S}\log(h)$. 
 \item The $\ka$-valued 2-vector field $\lambda_0$ is a Poisson 2-cocycle which depends on the choice of splitting $l_0$, but its cohomology class does not
    \[ c_2(A_S):=[\lambda_0]\in H^2_{\pi_S}(S,\ka).\]
    If $l:A_S\to \ka\simeq S\times \R$ is a second splitting, then $l=l_0+i_Z\circ \mu_S$, for a unique vector field $Z$ on $S$. Under this change, $\lambda_0$ becomes $\lambda=\lambda_0+\d_{\pi_S}Z$.
\end{itemize}

We can describe all possible IM connection 1-forms of a codimension one jet.

\begin{proposition}
\label{prop:partial:splittings:codim:1}
Let $(T^*S\oplus\R,\mu_S=\pr_{T^*S})$ be a codimension one first jet with a choice of trivialization $\ka\simeq S\times\R$. Then IM connection 1-forms are in 1-to-1 correspondence with triples $\theta\in\Omega^1(S)$, $Z\in\X(S)$ and $U:T^*S\to T^*S$, that satisfy 
\[  \pi_S^{\sharp}(\theta)=V,\qquad i_{\pi_S^\sharp(\be)}U(\al)=\lambda_0(\al,\be)+\d_{\pi_S} Z(\al,\be),\]
and the structure equations
\begin{align}
    \label{eq:S2:abelian} \tag{S2''}
    & i_{\pi^{\sharp}_S(\alpha)}\d\theta=0,\\
    \label{eq:S3:abelian} \tag{S3''}
    & U([\al,\be]_{\pi_S})=\Lie_{\pi_S^\sharp(\al)}U(\be)-i_{\pi^\sharp_S(\be)}\d U(\al)+\\
    &\qquad \qquad \qquad \qquad + \pi_S(U(\al),\beta)\theta +\pi_S(\theta,\al)U(\beta)-\pi_S(\theta,\be)U(\al),\notag 
\end{align}
for all $\al,\be\in\Omega^1(S)$.
\end{proposition}

\begin{proof}
As we saw in Section \ref{sec:obstructions:partial:split}, IM connection 1-forms $(L,l)$ correspond to coupling data $(\nabla^L,U)$ satisfying the structure equations (S1)-(S3) in Proposition \ref{prop:structure:eqs}. As we observed above, the splitting $l:A_S\to\ka$ is related to our fixed splitting $l_0$ by $l=l_0+i_Z\circ \mu_S$, for a unique vector field $Z$ on $S$. On the other hand, the connection $\nabla^L$ is given by
\[ \nabla^L_X=\Lie_X+\theta(X), \]
for some $\theta\in\Omega^1(S)$. By Proposition \ref{prop:operators:splitting} and the discussion preceding the proposition, we see that we can codify the coupling data in terms of a triple $(\theta,U,Z)$, related to $V$ and $\lambda_0$ as in the statement and we only have to take care of the structure equations. 

In codimension one the first structure equation in Proposition \ref{prop:structure:eqs} is always satisfied. One the other hand, one easily checks that the second and third structure equations now take the form \eqref{eq:S2:abelian} and \eqref{eq:S3:abelian}.
\end{proof}

\begin{remark}\label{remark:Basic:class}
Equation \eqref{eq:S2:abelian}  has the following interpretation. For the (possibly singular) symplectic foliation $\F$ of $(S,\pi_S)$, we denote by $H^{\bullet}_{\F-\mathrm{bas}}(S)$ the $\F$-basic cohomology, i.e., the cohomology of differential forms $\eta$ satisfying 
\[i_{\pi_S^{\sharp}(\alpha)}\eta=0 \quad \textrm{and}\quad  i_{\pi_S^{\sharp}(\alpha)}\d\eta=0,\qquad \forall \, \alpha\in T^*S,\] 
endowed with the de Rham differential. Then \eqref{eq:S2:abelian} means that $\d\theta$ is an $\F$-basic form. If we change the trivialization, $\theta$ changes to $\theta+\d f$, so $\d\theta$ stays the same. If we change the IM connection 1-form, then $\theta$ changes to $\theta+\eta$, where $\eta$ is $\F$-basic. This shows that, the class
\[[\d\theta]\in H^2_{\F\textrm{-bas}}(S)\]
is independent of the choice of IM connection 1-form and trivialization. It is easy to see that the vanishing of this class is equivalent to $c_1(\ka)$ being in the image of the canonical map 
    \[\pi_S^{\sharp}:H^1(S)\to H^1_{\pi_S}(S),\]
and also equivalent to the existence of a (usual) flat connection $\nabla$ on $\ka$ inducing $\nabla^\ka$, i.e., such that $\nabla^{\ka}_{\al}=\nabla_{\pi_S^{\sharp}(\al)}$. 
\end{remark}

From the proposition, we obtain a simple class of codimension one jets for which it is easy to decide whether they are partially split:

\begin{corollary}
\label{cor:simple:codimension1:jets}
A Poisson manifold $(S,\pi_S)$ together with a Poisson vector field $V\in \X_{\pi_S}(S)$ yield a first jet of a Poisson structure on $A_S=T^*S\oplus\R$, with 
\begin{equation}\label{eq:1st:jet:Florian}
\nabla^{\ka}_{\alpha}f=\Lie_{\pi_S^{\sharp}(\alpha)}f+(i_V\alpha)\, f \quad\textrm{and}\quad \lambda_0\equiv 0.
\end{equation}
This is partially split if and only if $V=\pi_S^{\sharp}(\theta)$ for some 1-form $\theta$ with $\d\theta$ is $\F$-basic.
\end{corollary}

\begin{proof}
If $V=\pi_S^{\sharp}(\theta)$, with $\d\theta$ $\F$-basic, then \eqref{eq:S2:abelian} holds. Letting $U=0$, \eqref{eq:S3:abelian} also holds, so we obtain an IM connection 1-form. 

Conversely, if the first order jet \eqref{eq:1st:jet:Florian} is partially split, by \eqref{eq:S2:abelian} we have $c_1(\ka)=[V]=[\pi_S^{\sharp}(\theta_0)]$, for some 1-from with $\d\theta_0$ $\F$-basic. So there is a function $f$ such that $V=\pi_S^{\sharp}(\theta_0+\d f)$, and we can set $\theta:=\theta_0+\d f$.  
\end{proof}

The next examples fit into the setting of this corollary.

\begin{example}\label{example:log:symplectic:not:p:s} Consider an orientable log-symplectic manifold $(M^{2n},\pi)$, with singular locus $S:=(\wedge^n\pi)^{-1}(0)$. Then, as shown in \cite{GuiMirPires}, $S$ has a tubular neighborhood $M_0\subset S\times \R$ in which
\[\pi=\pi_S+V\wedge t\partial_t,\]
where $V\in \X^1(S)$ is the restriction to $S$ of the modular vector field of $\pi$. Under the induced isomorphism $\ka\simeq S\times \R$, the first order jet along $S$ has classes $c_1(\ka)=[V]$ and $c_2(A_S)=[\lambda]=0$. Since $V$ is everywhere transverse to the symplectic leaves of $\pi_S$, this first order jet is not partially split. Note also that, since $\lambda=0$, the inclusion $T^*S\hookrightarrow A_S$ corresponding to the splitting is a Lie algebroid homomorphism. This supports the claim made in Remark \ref{remark:split:is:not:part:split}. 
\end{example}

\begin{example}\label{example:Florian}
On $S:=\R^3$, consider  
\[\pi_S:=(x\partial_y-y\partial_x)\wedge \partial_z,\quad V:=(x^2+y^2)\partial_z=\pi_S^{\sharp}(\theta_0),\quad \textrm{where}\quad  \theta_0:=x\d y-y\d x.\]
The corresponding first order jet is not partially split. Indeed, assume that $V=\pi_S^{\sharp}(\theta)$, for some 1-form $\theta$ with $\d\theta$ $\F$-basic. Since the symplectic foliation $\F$ has codimenion 1, it follows that $\d\theta=0$, and therefore $\theta=\d f$, for some $f\in C^{\infty}(\R^3)$. But the restriction of $\theta_0$ to a leaf $L_r:=\{x^2+y^2=r^2\}$ is not exact.
\end{example}

\begin{example}\label{example:Stephane}
On the 3-torus $S:=\{(\varphi_1,\varphi_2,\varphi_3)\, :\, \varphi_i\in S^1\}$, consider 
\[\pi_S:=\partial_{\varphi_1}\wedge\partial_{\varphi_2},\quad V:=\pi_S^{\sharp}(\theta_0),\quad \textrm{where}\quad \theta_0=\cos(\varphi_3)\d \varphi_1.\]
The corresponding first order jet is not partially split. Again, if we assume that $V=\pi_S^{\sharp}(\theta)$, with $\d\theta$ $\F$-basic, it follows that $\d\theta=0$. Hence, the integral $\int_{\gamma}\theta$ would depend only on the homology class of $\gamma:S^1\to S$. However, this is not true, since
\[\int_{S^1\times \{\varphi_2\}\times \{\varphi_3\}}\theta=\int_{S^1\times \{\varphi_2\}\times \{\varphi_3\}}\theta_0=2\pi \cos(\varphi_3),\]
where we used that $\theta-\theta_0$ vanishes on the leaves of $\pi_S$.  \end{example}

\begin{example}
Let $V\in \X^1(S^1)$ be the generator of the rotation and $W\in \X^1(S^3)$ be the generator of the $S^1$-action of the Hopf fibration $S^3\to S^2$. The first order jet on $S=S^1\times S^3$ corresponding to $\pi_S=V\wedge  W$ and $V$ is partially split. Namely, $V=-\pi_S^{\sharp}(\theta)$, where $\theta$ is a principal connection for the Hopf fibration. Note that the class 
\[[\d\theta]\in H^2_{\F-\mathrm{bas}}(S)\simeq H^2(S^2)\]
is the first Chern class of the Hopf fibration, and so it is non-trivial. Therefore, by Remark \ref{remark:Basic:class}, $\nabla^{\ka}$ is not induced by a flat connection and it follows from Proposition \ref{corollary:codim1:flat:type} (iii) below that this first order jet is not of principal type.
\end{example}

An IM connection 1-form $(L,l)\in\Omega^1_{\imult}(A_S;\ka)$ is called {\bf kernel flat} if the associated connection $\nabla^L$ is flat. We refer to \cite{FM22} for a detailed discussion of this condition. If $\nabla^L$ is flat, then the third structure equation \eqref{eq:S3:abelian} also has a nice geometric interpretation: it means that $(\d^{\nabla}U,U)$ is a closed IM 2-form with coefficients in $\ka$. We have the following consequence of \cite[Proposition 6.9]{FM22}.

\begin{proposition}
\label{corollary:codim1:flat:type}
Let $(A_S,\mu_S)$ be a codimension one first order jet inducing a Poisson structure $\pi_S$. Then
\begin{enumerate}[(i)]
    \item $(A_S,\mu_S)$ is isomorphic to the product jet $(S,\pi_S)\times\R$ if and only if
    \[ c_1(\ka)=0\quad\text{and}\quad c_2(A_S)=0; \]
    \item $(A_S,\mu_S)$ admits a kernel flat IM connection 1-form if and only $\ka$ admits a flat connection $\nabla$ inducing $\nabla^\ka$ and $c_2(A_S)$ is in the image of the canonical map
    \[ H^2_{\imult}(A_S;\ka)\to H^2_{\pi_S}(S;\ka);\]
    \item $(A_S,\mu_S)$ is of principal type if and only if $\ka$ admits a flat connection $\nabla$ inducing $\nabla^\ka$ and the class $c_2(A_S)$ is in the image of the canonical map
    \[ \pi_S^{\sharp}: H^2(S;\ka) \to H^2_{\pi_S}(S;\ka). \]
\end{enumerate}
\end{proposition}

\begin{example}
Let $\hh$ be a compact semi-simple Lie algebra, and consider the unit sphere $S\subset \hh^*$, as in the example in Subsection \ref{example:Lie-Poisson:sphere}. Under the diffeomorphism 
\[ S\times (0,\infty)\diffto \hh^*\backslash\{0\}, \quad (x,t)\mapsto \frac{1}{t}x, \] 
the linear Poisson structure on $\hh^*$ corresponds to the Poisson structure $t\cdot \pi_S$ on $S\times (0,\infty)$. This yields a splitting of
$A_S\simeq T^*S\oplus\R$, under which one has
\[\lambda(\alpha,\beta)=\pi_S(\alpha,\beta),\quad 
\nabla^{\ka}_{\alpha}=\Lie_{\pi^{\sharp}_S(\alpha)}.\]
This first order jet is partially split with coupling data
\[\nabla^L_X=\Lie_X,\quad U=-\id:T^*S\to T^*S.\]

Let $H$ be the 1-connected Lie group integrating $\hh$,
In the example in Subsection \ref{example:Lie-Poisson:sphere}, we have seen that if  $\hh\not\simeq \mathfrak{so}(3,\mathbb{R})$ the over-symplectic groupoid $\G_S=H\ltimes S$ is not of principal type. By applying Proposition \ref{corollary:codim1:flat:type} (iii), we can see that this is also true at the Lie algebroid level. Namely, since $H^1(S)=0$, any flat connection on $\ka\simeq S\times \R$ is isomorphic to the trivial one. Since $\hh\not\simeq \mathfrak{so}(3,\mathbb{R})$, we have that $H^2(S)=H^2(S,\ka)=0$. If this first jet was of principal type, then the corollary would imply that $\lambda=\pi_S$ is exact. To see that this is not the case, assume that $X$ is a primitive of $\pi_S$. Then $\Lie_X(\pi_S)=-\d_{\pi_S}X=-\pi_S$, and so $(\phi_X^{t})^*\pi_S=e^{-t}\pi_S$. In particular, for each symplectic leaf $(L,\omega_L)$, we have that $\phi_X^t(L)$ is a symplectic leaf symplectomorphic to $(L,e^t\omega_L)$. Then the symplectic volume of leaves in the sphere $S$ would be unbounded. It is well-known that this volume is bounded (see, e.g., \cite[Lemma 2.2]{Marcut14}).
\end{example}

The previous example, although not globally, is of principal type when restricted to small open sets. The following example violates even this condition
\begin{example}\label{example:Ginzburg}
Consider the first order jet $A_S\simeq T^*S\oplus\R$ on  $S=\R^2$ with data
\[ \lambda=\pi_S=(x^2+y^2)\partial_x\wedge\partial_y,\quad V\equiv 0. \]
As in the previous example, this first order jet is partially split, with coupling data
\[ \nabla^L_X=\Lie_X, \quad U=-\id. \] 
We claim that there is no open neighborhood $S_0\subset \R^2$ of 0 on which the restriction of the jet is of principal type. It is proved in \cite{Ginz96} that the class $[\pi_S]$ is non-trivial in the formal Poisson cohomology at 0, so it is non-trivial on any open neighborhood $S_0$ of $0$. Since $S_0$ can be assumed contractible, the claim follows from Proposition \ref{corollary:codim1:flat:type} (iii). 
\end{example}

\begin{remark}[Locally trivial jets]
In the previous example the jet is not locally a trivial product jet.
We call a general jet $(A_S,\mu_S)$ \textbf{locally trivial}, if $S$ can be covered by open subsets on which the restriction of the jet is isomorphic to a trivial product jet: $A_S|_U\simeq T^*U\times \gg$ (see Example \ref{ex:product}). We note the following properties of this class of jets
\begin{enumerate}[(i)]
\item First order jets of principal type are locally trivial: any transitive Lie algebroid $B\Rightarrow S$ is locally isomorphic to a product $B|_{U}\simeq TU\times \gg$ \cite{Du01}.
\item A partially split first order jet is locally trivial if and only if the local model is locally isomorphic to a product $(U,\pi_S|_U)\times (\gg^*,\pi_{\gg})$ (see Corollary \ref{corollary:uniqueness:model}). 
\item In codimension one, a first order jet is locally trivial if and only if the classes $c_1(\ka)$ and $c_2(A_S)$ vanish locally (see Proposition \ref{corollary:codim1:flat:type} (i)). 
\item A locally trivial first order jet is not necessarily partially split: Example \ref{example:Stephane} is locally trivial because $c_2(A_S)=0$ and $c_1(\ka)$ vanishes locally.
\item A partially split first order jet is not necessarily locally trivial: the first order jet from Example \ref{example:Ginzburg} is partially split, however it is not locally trivial because $c_2(A_S)$ does not vanish around 0.
\end{enumerate}
\end{remark}

In codimension one, Proposition \ref{prop:local:model:explicit} gives an explicit form of the local model.

\begin{proposition}\label{prop:local:model:codim:one}
Let $(A_S,\mu_S)$ be a partially split first order jet of a Poisson structure with $\ka\simeq S\times \R$. For an IM connection 1-form $(L,l)$ with coupling data $(\nabla^L=\dd +\theta,U)$, the corresponding local model $(M_0,\pi_0)$ is an open neighborhood $M_0\subset S\times \R$ of $S\times\{0\}$ with Poisson structure
\[\pi_0=\gamma_t+\gamma_t^{\sharp}(\theta)\wedge t\partial_t,\]
where $\gamma_t\in \X^2(S_t)$ is the bivector on the plaque $S_t:=M_0\cap( S\times\{t\})$ given by
\begin{equation}\label{eq:formula:gamma_t}
\gamma^{\sharp}_t=\pi_S^{\sharp}\circ (\id+t U)^{-1}.
\end{equation}
\end{proposition}

The following examples discuss the special cases where either $\theta$ or $U$ vanish.

\begin{example}\label{example:local:model:codim1:U=0}
If in Proposition \ref{prop:local:model:codim:one} we assume that $U=0$, then 
\[\pi_0=\pi_S+\pi_S^{\sharp}(\theta)\wedge t\partial_t\in \X^2(S\times \R).\]
Note that these Poisson structures fit into a more general class: for any Poisson vector field $V$ on $(S,\pi_S)$, one has a Poisson structure $\pi_S+V\wedge t\partial_t$, which has $S$ as a Poisson submanifold, but its first order jet will in general not be partially split (see Corollary \ref{cor:simple:codimension1:jets} or Example \ref{example:log:symplectic:not:p:s}).
\end{example}

\begin{example}
If in Proposition \ref{prop:local:model:codim:one} we assume that $\theta=0$, then the terms in the second line of \eqref{eq:S3:abelian} vanish $U:T^*S\to T^*S$ is a closed IM 2-form. In this case, the local model $(M_0,\pi_0)$ can be thought of as a deformation $\{\gamma_t\}_{t\in \R}$ of the Poisson structure $\pi_S$
\[\pi_0^{\sharp}|_{(x,t)}=\gamma_t^{\sharp}|_x=\pi_S^{\sharp}\circ(\id+tU)^{-1}|_x.\] 
If we assume that $\pi_S$ is integrable, this deformation has a nice geometric interpretation. Namely, let $(\Sigma,\omega)\tto S$ be a target 1-connected symplectic groupoid integrating $\pi_S$ and let $\eta\in \Omega^2_{\mult}(\Sigma)$ be the multiplicative closed 2-form corresponding to $U$. Then, for small $t$, $(\Sigma,\omega+t\eta)$ is a symplectic groupoid (at least locally), and the corresponding Poisson structures is precisely $\gamma_t$ (see e.g.\ \cite[Section 3.2.2]{AleMaria2}).
\end{example}

\begin{example}\label{example:deformation:IM:form}
Let us consider the particular case of the previous example where $\theta=0$ and $U=-\id$. The local model becomes simply
\[M_0=S\times (\R\backslash\{1\}),\quad \pi_0=\frac{1}{1-t}\pi_S.\]
We claim that $\pi_0$ extends to a smooth Dirac structure on $\ka^*=S\times \R$ if and only if $\pi_S$ is a regular Poisson structure. This is in contrast with what happens for Vorobjev's local model around symplectic leaves, which always comes from a global Dirac structure on $\ka^*$ (see, e.g., \cite{MarcutPhD}). 

One direction is clear: if $\pi_S$ corresponds to a regular foliation $\F$ with leafwise symplectic form $\omega$, then $\pi_0$ can be extended by the regular Dirac structure with foliation $T\F_0:=T\F\times \R$ and leafwise presymplectic form: $\omega_0=(1-t)\pr_S^*\omega|_{T\F_0}$. For the converse, we show that the graph of a non-regular $\pi_0$ does not extend to a Dirac structure on $S\times\R$. By invoking Weinstein's splitting theorem, we can reduce to the case when $\pi_S|_{x_0}=0$ and $\pi_S$ does not vanish identically in a neighborhood of $x_0$. Assume that the graph of $\pi_0^{\sharp}$ extends to a Dirac structure $\L$ on $S\times \R$. Since $\L_{(x_0,t)}=T^*_{(x_0,t)}(S\times \R)$ for $t\neq 1$, by continuity, we must have also that $\L_{(x_0,1)}=T^*_{(x_0,1)}(S\times \R)$. This implies that $\L$ corresponds to a Poisson structure $\pi_{\L}$ around $(x_0,1)$, which of course must be $\pi_0$ outside of $S\times \R$. This is a contradiction, because $\pi_0$ does not extend continuously at $(x_0,1)$.
\end{example}

\subsection{Other local models}\label{sec:other:local:local}

The local model for a Poisson manifold $(M,\pi)$ around a Poisson submanifold $S$ that we have introduced is characterized by the fact that the underlying Lie algebroid is the linear approximation to $T^*M$ around $S$. In some cases, it is possible that our local model does not exist, but there may exist some other local model. Many such local models have been considered in the literature, and they often turn out to be homogeneous of a certain degree. 
For a detailed version of the following discussion, see for example the recent PhD thesis of Aldo Witte \cite{Aldo}. However, these local models have little overlap with the ones we are considering.

Let $(M,\pi)$ be a Poisson manifold and let $S\subset M$ be a Poisson submanifold. By passing to a tubular neighborhood, we can replace $M$ by a vector bundle $E\to S$. Denote by $m_t(e)=t e$ the fiberwise multiplication by $t$. The bivector field $m_t^*(\pi)$ has a Laurent expansion of the following form
\[m_t^*(\pi)\simeq \frac{1}{t}\pi_{-1}+\pi_0+t\pi_1+\ldots,\]
where each term $\pi_i\in \X^2(E)$ is homogeneous of degree $i$, in the sense that
\[m_t^*(\pi_i)=t^{i}\pi_i.\]
Let $\pi_{k}$ be the first non-zero term of the sequence $\{\pi_i\}_{i\geq -1}$. Then $\pi_k$ is itself a Poisson structure and is independent of the chosen tubular neighborhood (see \cite{Aldo}). The pair $(E,\pi_k)$ is the \textbf{homogeneous local model} of $\pi$ around $S$. In most situations, the homogeneous local model is very different from ours. First of all, in general, $\pi_k$ depends on the $(k+2)$-th jet of the Poisson structure at $S$. 
Let us discuss the geometric structure encoded by the first non-zero term $\pi_k$. 

\underline{$k=-1$}. Poisson structures on the vector bundle $E$ such that $m_t^*{\pi_{-1}}=t^{-1}\pi_{-1}$ are in 1-to-1 correspondence with Lie algebroid structures on $E^*$. The condition that $S$ is a Poisson submanifold is equivalent to $\ka:=E^*$ being a bundle of Lie algebras. Then $\pi_S=0$ and $\pi_{-1}$ coincides with the linear Poisson structure on $\ka^*=E$. As we saw in Proposition \ref{prop:pi_S=0}, the partial split condition is satisfied precisely when $\ka$ is a Lie algebra bundle. In that case $\pi_{-1}$ coincides with our linear model. 

\underline{$k=0$}. The condition $\pi_{-1}=0$ imply that $S$ is a Poisson submanifold for which the bundle of Lie algebras $\ka:=\ker\mu_S$ is abelian. Therefore, homogeneous Poisson structures exclude many examples where our local model exists (see Example \ref{example:local:model:codim1:U=0} for the case of codimension one). Poisson structures on the vector bundle $E$ such that $m_t^*{\pi_{0}}=\pi_{0}$, called \emph{quadratic Poisson structures}, are discussed in the recent PhD thesis \cite{Mykola} (see also \cite[Proposition 5.1.55]{Aldo}). Several ``weakly degenerate'' structures admit quadratic normal forms around their singularities. As mentioned in Example \ref{example:log:symplectic:not:p:s}, log-symplectic manifolds admit a quadratic normal form around the singular locus, and these are not partially split. Elliptic Poisson structures \cite{CaGu,Aldo} also admit a quadratic normal form around the singular locus.


\underline{$k=1$}. In this case, the first order jet is very simple: $A_S\simeq T^*S\oplus\ka$ with bracket $[(\al,\xi),(\be,\eta)]=(0,\lambda(\al,\be))$, for some $\lambda\in \X^2(S,\ka)$.  This jet is partially split if and only if $\lambda\equiv 0$. However, also in this case interesting geometric structures can appear. For example, if the codimension of $S$ in $M$ is one, there is a 1-to-1 correspondence between Poisson structures $\pi_{1}$ on the line bundle $E$ satisfying $m_t^*(\pi_1)=t\pi_1$ and Jacobi structures $J$ on $E$ (see \cite{Vitagliano}). 

\underline{$k=2$}. The local model for scattering Poisson structures \cite{Lanius} is 2-homogeneous. For $k\geq 2$, the first order jet of $\pi$ at $S$ vanishes identically, so our local model is the zero Poisson structure.

\appendix


\section{The infinitesimal multiplicative Moser method}
\label{appendix}

Multiplicative forms and infinitesimal multiplicative (IM) forms play an important role in Poisson Geometry and, in particular, in this paper. In this appendix we give the main definitions and some basic results concerning such forms that are needed throughout the paper. We will use the notations and conventions from \cite{CFM21} and the Appendix of \cite{FM22}, and we refer to \cite{BuCa12,CrSaSt15,DrEg19} for proofs and details. In the last section of this appendix we give an IM-version of the well-known Moser method from symplectic geometry, which does not seem to be known, and which is used in the paper to prove uniqueness of the local model. 

\subsection{Multiplicative forms}

Let $\G\tto M$ be a Lie groupoid with source/target $\s,\t:\G\to M$ and multiplication $m:\G\timesst \G\to \G$. A differential form $\omega\in\Omega^k(\G)$ is called {\bf multiplicative} if
\[m^*\omega=\pr_1^*\omega+\pr_2^*\omega\in \Omega^{k}(\G\timesst\G),\]
where $\pr_i:\G\timesst \G\to \G$ are the projections on the factors. The differential of a multiplicative form is again a multiplicative form, so we have a complex of multiplicative differential forms $(\Omega^{\bullet}_\mult(\G),\d)$. 

A basic fact is that two multiplicative forms $\omega,\omega'\in \Omega^k_\mult(\G)$ such that $\omega|_M=\omega'|_M$ and $(\d\omega)|_M=(\d\omega')|_M$ actually coincide. Denote by $(A,[\cdot,\cdot],\rho)$ the Lie algebroid of $\G\tto M$. Given a multiplicative form $\omega\in\Omega^k_\mult(\G)$ one defines two vector bundle maps  $\mu:A\to \wedge^{k-1}T^*M$ and $\zeta:A\to \wedge^kT^*M$ by setting
\begin{equation}
\label{eq:M:IM:forms}
 \mu(a)=i_a\omega|_{TM},\quad \zeta(a)=i_a\d\omega|_{TM}.
\end{equation}
The pair $(\mu,\zeta)$ satisfies the following set of equations for any $\al,\be\in\Gamma(A)$
\begin{align}
\label{eq:mult:form}
i_{\rho(\be)} \mu(\al) &=-i_{\rho(\al)} \mu(\be),\notag\\
\mu([\al,\be])&=\Lie_{\rho(\al)}\mu(\be)-i_{\rho(\be)}\d \mu(\al)-i_{\rho(\be)}\zeta(\al),\\
\zeta([\al,\be])&=\Lie_{\rho(\al)}\zeta(\be)-i_{\rho(\be)}\d \zeta(\al).\notag
\end{align}

This leads to the notion of {\bf infinitesimal multiplicative $k$-form} on an arbitrary Lie algebroid $A\Ato M$ integrable or not: a pair $(\mu,\zeta)$, where $\mu:A\to \wedge^{k-1}T^*M$ and $\zeta:A\to \wedge^kT^*M$ are bundle maps satisfying \eqref{eq:mult:form}. 

The space of IM forms is denoted by $\Omega_\imult^k(A)$ and it becomes a cochain complex with differential given by
\begin{equation}
    \label{eq:IM:differential}
    \d_\imult:\Omega_\imult^k(A)\to \Omega_\imult^{k+1}(A), \quad \d_\imult(\mu,\zeta):=(\zeta,0).
\end{equation}
For a \emph{target 1-connected} Lie groupoid $\G\tto M$ with Lie algebroid $A\Ato M$ the assignment $\omega\mapsto (\mu,\zeta)$ given by \eqref{eq:M:IM:forms} is an isomorphism of complexes
\[ (\Omega^\bullet_\mult(\G),\d)\simeq (\Omega^\bullet_\imult(A),\d_\imult). \]

We denote closed IM forms $(\mu,0)$ simply by $\mu$. For us the most relevant IM-forms will be closed IM 2-forms $\mu:A\to T^*M$ which are {\bf non-degenerate}, meaning that they are bundle isomorphisms.

\begin{example}[IM exact forms]\label{example:IM:exact}
Any differential form $\omega\in \Omega^{k}(M)$ induces a IM-form \[\rho^*(\omega)\in \Omega_{\imult}^{k}(A),\]
with components $\rho^*(\omega)=(\mu,\zeta)$ given by
\[\mu(\al)= i_{\rho(\al)}\omega\quad \zeta(\al)=i_{\rho(\al)}\d\omega.\]
Forms of this type will be called \textbf{IM-exact}.
\end{example}

\begin{example}[IM forms and Poisson structures]
If  $\mu:A\to T^*M$ is a non-degenerate closed IM 2-form, we obtain a Poisson structure $\pi\in\X^2(M)$ by letting
\begin{equation}
\label{eq:Poisson:IMform} 
\pi^\sharp:=\rho\circ\mu^{-1}: T^*M\to TM.
\end{equation}
This in turn induces a cotangent Lie algebroid structure $(T^*M,[\cdot,\cdot]_\pi,\pi^\sharp)$, where
\[ [\al_1,\al_2]_\pi:=\Lie_{\pi^\sharp(\al_1)}\al_2-\Lie_{\pi^\sharp(\al_2)}\al_1-\d\pi(\al_1,\al_2). \]
The IM form $\mu:A\to T^*M$ then becomes a Lie algebroid isomorphism. 

Conversely, a Poisson structure $\pi\in\X^2(M)$ yields a Lie algebroid structure on the cotangent bundle $(T^*M,[\cdot,\cdot]_\pi,\pi^\sharp)$ such that the identity map $\id:T^*M\to T^*M$ is a non-degenerate closed IM 2-form.

The groupoid version of this is as follows: starting with a symplectic groupoid $(\G, \omega)$, i.e., a multiplicative non-degenerate closed 2-form $\omega\in \Omega^2_\mult(\G)$, we obtain
\begin{enumerate}[(i)]
\item a non-degenerate closed IM 2-form on its Lie algebroid $\mu:A\to T^*M$, and
\item a unique Poisson structure $\pi\in\X^2(M)$ for which the target map $\t:\G\to M$ is a Poisson map.
\end{enumerate}
Then $\mu$ and $\pi$ are related by \eqref{eq:Poisson:IMform}. In this way one recovers the well-known correspondence between target 1-connected symplectic groupoids and integrable Poisson manifolds. 
\end{example}

\begin{example}[IM forms and Dirac structures]\label{example:robust}
Generalizing the previous example, a closed IM 2-form $\mu:A\to T^*M$ is called \textbf{robust} if the map
\begin{equation}\label{eq:robust:rho:mu}
(\rho,\mu):A\to TM\oplus T^*M
\end{equation}
has constant rank equal to the dimension of $M$. Then the image of this map 
\[\L:=(\rho,\mu)(A) \subset TM\oplus T^*M\]
is a Dirac structure on $M$ \cite{BCWZ04}. 
 
If $\mu$ is robust and $(\rho,\mu)$ is injective, then we call $\mu$ \textbf{Dirac non-degenerate}. In this case, we have a Lie algebroid isomorphism $(\rho,\mu):A\diffto \L$. Conversely, for any Dirac structure $\L$, the projection $\mu:=\pr_{T^*M}:\L\to T^*M$ is Dirac non-degenerate. 
\end{example}

\begin{example}[The canonical IM 2-form]\label{ex:canonical:IM}
Let $\G\tto M$ be a Lie groupoid with Lie algebroid $A\Ato M$. The canonical symplectic structure $\omega_{\can}$ on the cotangent bundle $T^*\G\tto A^*$ is multiplicative. Moreover,  $(T^*\G,\omega_{\can})\tto A^*$ is the symplectic groupoid integrating the linear Poisson structure on $A^*$. At the infinitesimal level, we have the canonical ``reverse isomorphism'' $\mu_{\can}:T^*A \diffto TA^*$ (see \cite{Mackenzie05}), which is a closed, non-degenerate IM 2-form $\mu_{\can}\in \Omega^2_{\imult}(T^*A)$.
\end{example}

\subsection{Multiplicative forms with coefficients}
We also need to consider multiplicative forms and IM forms with coefficients in a representation.

Let $\G\tto M$ be a Lie groupoid and let $p:E\to M$ be a $\G$-representation. We will work with differential forms on $\G$ with coefficients in $\s^*E$ (instead of $\t^*E$, as in some of the references), which we denote by
\[\Omega^{\bullet}(\G;E):=\Omega^{\bullet}(\G;\s^*E).\]
We refer to these simply as multiplicative forms on $\G$ with coefficients in $E$. Similarly, we denote by $\Omega^{\bullet}(\G^{(k)};E)$ the space of differential forms on the manifold $\G^{(k)}$, of composable $k$-strings of arrows, with values in the vector bundle $(\s\circ \pr_k)^*(E)\to \G^{(k)}$, where $\pr_i:\G^{(k)}\to \G$ is the projection onto the $i$-th component.

\begin{definition}
\label{def:mult:forms:coeff}
A form $\omega\in\Omega^{\bullet}(\G;E)$ is called {\bf multiplicative} if it satisfies
\[m^*\omega|_{g_1g_2}=\pr_1^*\omega|_{g_1}+\pr_2^*\omega|_{g_2},\quad\forall\,  (g_1,g_2)\in \G^{(2)}.
\]
We denote by $\Omega^\bullet_\mult(\G;E)$ the space of $E$-valued multiplicative $k$-forms. 
\end{definition}


Let us pass to the infinitesimal level, so denote by $A\Ato M$ the Lie algebroid of $\G\tto M$ and $\nabla$ the induced representation of $A$ on $E$. 
Now, if $\omega\in\Omega^k_\mult(\G;E)$ is an $E$-valued multiplicative $k$-form, we define  a linear operator $L:\Gamma(A)\to \Omega^k(M; E)$ and a vector bundle map $l:A\to \wedge^{k-1}T^*M\otimes E$ by
\begin{align}
\label{eq:M:IM:E:forms}
L(\al)_x(v_1,\dots,v_k)&:=\frac{\d}{\d \epsilon}\Big|_{\epsilon=0} \phi_{\al^{L}}^{\epsilon}(x)\cdot \omega(\d_x \phi_{\al^{L}}^{\epsilon}(v_1),\dots,\d_x \phi_{\al^{L}}^{\epsilon}(v_k)),\\
l(a)&:=(i_a\omega)|_{TM}.\notag
\end{align}
where $\phi_{\al^{L}}^{\epsilon}$ denotes the flow of the left-invariant vector field $\al^L$.
The resulting operator $L$ is a kind of differential operator with symbol $l$, in the sense that for any $f\in C^\infty(M)$ and $\al\in\Gamma(A)$ it satisfies
\begin{equation}
\label{eq:symbol:IM:form}
L(f\al)=fL(\al)+\d f\wedge l(\al).
\end{equation}
Furthermore, the pair $(L,l)$ satisfies the following set of equations
\begin{align}
\label{eq:compatibility:IM:E:form}
i_{\rho(\al)}l(\be)&=-i_{\rho(\be)} l(\al),\notag \\
L([\al,\be])&=\Lie_{\al} L(\be)-\Lie_{\be} L(\al),\\
l([\al,\be])&=\Lie_{\al} l(\be)-i_{\rho(\be)}L(\al), \notag
\end{align}
where, for $\al\in\Gamma(A)$ and $\gamma\in\Omega^k(M; E)$, we denoted
\begin{equation*}\Lie_\al\gamma(X_1,\dots,X_k):=\nabla_\al(\gamma(X_1,\dots,X_k))-\sum_{i=1}^k\gamma(X_1,\dots,[\rho(\al),X_i],\dots,X_k).
\end{equation*}

These equations make sense for a general, possibly non-integrable Lie algebroid.

\begin{definition}
Let  $(E,\nabla)$ be a representation of a Lie algebroid $A\Ato M$. An {\bf $E$-valued IM $k$-form} is a pair $(L,l)$, where $L:\Gamma(A)\to \Omega^k(M,E)$ is a linear map and $l:A\to \wedge^{k-1}T^*M\otimes E$ is a vector bundle map satisfying \eqref{eq:symbol:IM:form} and \eqref{eq:compatibility:IM:E:form}. 
\end{definition}

We denote by $\Omega^\bullet_\imult(A;E)$ the space of $E$-valued IM $k$-forms. For historical reasons, these are also sometime called \emph{Spencer operators}. For a \emph{target 1-connected} Lie groupoid $\G\tto M$ with Lie algebroid $A\Ato M$ the assignment $\omega\mapsto (L,l)$ given by \eqref{eq:M:IM:E:forms} gives an isomorphism
\[ \Omega^\bullet_\mult(\G;E)\simeq \Omega^\bullet_\imult(A;E). \]

\subsection{Multiplicative Cartan calculus} 
 
Recall that a vector field $X\in\X(\G)$ is called {\bf multiplicative} if its flow $\phi^t_X:\G\to \G$ is by Lie groupoid automorphisms. Equivalently, $X:\G\to T\G$ is a groupoid morphism, and this amounts to the relation
\[ \d m(X_g, X_h)=X_{gh}, \quad \forall (g,h)\in \G\timesst \G. \]
We denote by $\X_\mult(\G)\subset \X(\G)$ the subspace of multiplicative vector fields.
A multiplicative vector field $X\in\X(\G)$ covers vector field $\sigma_X\in\X(M)$
\[ \s_*(X)=\t_*(X)=\sigma_X. \]
The Lie bracket of multiplicative vector fields is multiplicative.
If $X\in \X_\mult(\G)$ and $\omega\in\Omega^k_\mult(\G)$ then the contraction $i_X\omega$ is still a multiplicative form. In this way, all the usual formulas from Cartan calculus restrict to multiplicative objects.

At the infinitesimal level, multiplicative vector fields correspond the {\bf Lie algebroid derivations}, i.e., to $\R$-linear maps $D:\Gamma(A)\to\Gamma(A)$ for which there exists a vector field $\sigma_D\in\X(M)$, called {\bf symbol} of $D$, such that
\[ D(f\al)=fD(\al)+\Lie_{\sigma_D}(f) \al, \quad \forall f\in C^\infty(M), \al\in\Gamma(A), \]
and which act as derivations of the Lie bracket
\[ D([\al,\be])=[D(\al),\be]+[\al,D(\be)],\quad \forall \al,\be\in\Gamma(A). \]
For example, a section $\be\in\Gamma(A)$ defines the {\bf inner derivation}
\[ D_\be:=[\be,\cdot], \]
with symbol $\sigma_{D_\be}=\rho(\be)$.

Given a multiplicative vector field $X\in\X_\mult(\G)$ the corresponding derivation $D_X\in\Der(A)$ has symbol $\sigma_X$, and can be defined by
\[ (D_X(\al))^L:=[X,\al^L], \]
where $\be^L\in \X(\G)$ denotes the left-invariant vector field on $\G$ determined by $\be \in\Gamma(A)$. For a \emph{target 1-connected} Lie groupoid $\G\tto M$ with Lie algebroid $A\Ato M$ the assignment $X\mapsto D_X$ establishes an isomorphism between the space of multiplicative vector fields $\X_\mult(\G)$ and the space of Lie algebroid derivations $\Der(A)$. 

For a Lie algebroid $A$ the {\bf flow of a derivation} $D\in\Der(A)$ is the unique (local) 1-parameter group of Lie algebroid automorphisms $\varphi^t_D:A\to A$ satisfying
\[ \frac{\d}{\d t}(\varphi^t_D)^*(\al)=D((\varphi^t_D)^*(\al)), \quad \varphi^t_D=\id. \]
It covers the flow $\phi^t_{\sigma_D}$ of the symbol of $D$ and it is defined as long as the flow of the symbol is defined. If $D_X\in\Der(A)$ is a derivation corresponding to a multiplicative vector field $X\in\X_{\mult}(\G)$ then its flow is the 1-parameter group of Lie algebroid automorphisms induced by the flow of $X$
\[ \varphi^t_{D_X}=(\phi^t_X)_*:A\to A. \]

The assignments mapping  multiplicative forms and vector fields to IM forms and algebroid derivations can be used to transfer the Cartan Calculus on multiplicative forms to infinitesimal multiplicative forms. 

Given a IM form $(\mu,\xi)\in\Omega^k_\imult(A)$ and a Lie algebroid derivation $D\in\Der(A)$ with symbol $\sigma_D\in\X(M)$ one has
\begin{enumerate}[(i)]
\item the {\bf interior product} of $(\mu,\xi)$ by $D$ is the IM form of degree $k-1$
\[ i_D(\mu,\xi):=(-i_{\sigma_D}\mu,\Lie_D\mu+i_{\sigma_D}\xi). \]
\item the {\bf Lie derivative} of $(\mu,\xi)$ along $D$ is the IM form of degree $k$
\[ \Lie_D(\mu,\xi):=(\Lie_D\mu,\Lie_D\xi). \]
\end{enumerate}
Here, for a bundle map $\nu:A\to\wedge^kT^*M$, we have denoted by $\Lie_D\nu:A\to\wedge^kT^*M$ its Lie derivative, which is the bundle map defined by
\[ (\Lie_D\nu)(\al):=\frac{\d}{\d t}\Big|_{t=0} (\phi^t_{\sigma_D})^*\nu((\varphi^t_D)_*(\al))=\Lie_{\sigma_D}(\nu(\al))-\nu(D(\al)). \]
For example, for an inner derivation $D_\be=[\be,\cdot]$ and any IM form $(\mu,\xi)$ one obtains
\begin{equation*}
\Lie_{D_\be}(\mu,\xi)=\rho^*(\d\mu(\be)-\xi(\be)),
\end{equation*} 
i.e., the form is IM-exact, in the sense of Example \ref{example:IM:exact}.

\begin{proposition}
The operators $\d_\imult:\Omega_\imult^\bullet(A)\to \Omega_\imult^{\bullet+1}(A)$, $\Lie_D:\Omega_\imult^\bullet(A)\to \Omega_\imult^{\bullet}(A)$ and $i_D:\Omega_\imult^\bullet(A)\to \Omega_\imult^{\bullet-1}(A)$ satisfy
\begin{enumerate}[(i)]
\item $[\d_\imult,i_D]=\Lie_D$;
\item $[\Lie_{D},\Lie_{D'}]=\Lie_{[D.D']}$;
\item $[\Lie_{D},i_{D'}]=i_{[D,D']}$;
\item $[i_{D},i_{D'}]=0$.
\end{enumerate}
\end{proposition}

For a symplectic groupoid $(\G,\omega)$, the relation $i_X\omega=\eta$ gives a 1-to-1 correspondence between multiplicative vector fields $X$ and multiplicative 1-forms $\eta$, which sends multiplicative exact vector fields to multiplicative exact forms. In other words, we have the commutative diagram
\begin{equation}
    \label{eq:iso:symplectic:grpd}
    \vcenter{\xymatrix{
    \X_\mult(\G)\ar[r]^{\omega^\flat} & \Omega^1_\mult(\G)\\
    \Gamma(A)\ar[u]^{\delta}\ar[r]_{\mu} & \Omega^1(M)\ar[u]_{\delta}
    }}
\end{equation}
where the horizontal rows are isomorphisms and the vertical arrows are the simplicial differentials
\begin{align*}
    \delta:\Gamma(A)\to \X_\mult(\G),\qquad &\delta(\be)=\be^R-\be^L,\\
    \delta:\Omega^1(M)\to \Omega^1_\mult(\G),\qquad &\delta(\theta)=\t^*\theta-\s^*\theta.
\end{align*}

We give now the infinitesimal version of this result. 

\begin{lemma}\label{lemma:1:to:1:derivations:IM} 
A closed, non-degenerate IM 2-form $\mu:A\to T^*M$ gives a bijection
\[\Der(A)\diffto \Omega^1_{\imult}(A),\quad D\mapsto i_D(\mu,0),\]
which maps inner derivations to IM-exact 1-forms.
 \end{lemma}
\begin{proof}
For injectivity, $i_D(\mu,0)=0$ is equivalent to the relations
\[-i_{\sigma_D}\mu(\alpha)=0, \quad \Lie_{\sigma_D}(\mu(\al))-\mu(D(\al))=0, \quad \forall \al\in \Gamma(A).\]
Since $\mu$ is a bijection, the first relation implies that $\sigma_D=0$. From the second relation, we obtain then that also $D=0$.

For surjectivity, let $(\nu,\xi)\in \Omega^1_{\imult}(A)$. Then we need to find a derivation $D$ such that, for all $\al\in \Gamma(A)$, we have
\[-i_{\sigma_D}\mu(\al)=\nu(\al), \quad \Lie_{\sigma_D}(\mu(\al))-\mu(D(\al))=\xi(\al).\]
Since $\mu$ is a bijection, these relations have the unique solution
\[\sigma_D:=-\nu\circ \mu^{-1}, \quad D(\al):=\mu^{-1}\big(\Lie_{\sigma_D}(\mu(\al))-\xi(\al)\big).\]
The Leibniz rule for $D$ with respect to $\sigma_D$ follows because $\mu$ and $\xi$ are linear over $C^{\infty}(M)$. It remains to show that $D$ is a derivation. First, we compute
\begin{align*}
-i_{\rho(D(\alpha))}\mu(\beta)&=i_{\rho(\beta)}\mu( D(\alpha))\\
&=i_{\rho(\beta)}\big(\Lie_{\sigma_D}(\mu(\al))-\xi(\al)\big)\\
&=-i_{\rho(\be)}\d \nu (\al) - i_{\rho(\be)}\xi(\al) +i_{\rho(\beta)}i_{\sigma_D}\d 
\mu(\al)\\
&=\nu ([\al,\be]) - \Lie_{\rho(\al)}\nu(\be) - i_{\sigma_D}i_{\rho(\beta)}\d 
\mu(\al)\\
&=-i_{\sigma_D}\mu ([\al,\be]) + \Lie_{\rho(\al)}i_{\sigma_D}\mu(\be) - i_{\sigma_D}i_{\rho(\beta)}\d 
\mu(\al)\\
&=-i_{\sigma_D}\Lie_{\rho(\al)}\mu (\be) + \Lie_{\rho(\al)}i_{\sigma_D}\mu(\be)\\
&=i_{[\rho(\al),\sigma_D]}\mu(\be), 
\end{align*}
where we used the IM conditions for $(\mu,0)$ and $(\nu,\xi)$. Since $\mu$ is a bijection, we obtain the relation
\begin{equation}\label{eq:anchor:derivation}
\rho(D(\al))=[\sigma_D,\rho(\al)].
\end{equation}
Using this, we compute
\begin{align*}
\mu([D(\al),\be])&=-\mu([\be,D(\al)])\\
&=-\Lie_{\rho(\be)}\mu(D(\al))+i_{\rho(D(\al))}\d \mu(\be)\\
&=\Lie_{\rho(\be)}\xi(\al) - \Lie_{\rho(\be)}\Lie_{\sigma_D}\mu(\al)+i_{[\sigma_D,\rho(\al)]}\d \mu(\be)\\
\mu([\al,D(\be)])&=\Lie_{\rho(\al)}\mu(D(\be))-i_{\rho(D(\be))}\d \mu(\al)\\
&=-\Lie_{\rho(\al)}\xi(\be) + \Lie_{\rho(\al)}\Lie_{\sigma_D}\mu(\be)-i_{[\sigma_D,\rho(\be)]}\d \mu(\al)\\
\mu(D([\al,\be]))&=\Lie_{\sigma_D}(\Lie_{\rho(\al)}\mu(\be)-i_{\rho(\be)}\d\mu(\al))-\xi([\al,\be]).
\end{align*}
For the terms containing $\xi$, we have
\begin{align*}
\xi([\al,\be])-\Lie_{\rho(\al)}\xi(\be)+\Lie_{\rho(\be)}\xi(\al) &=\d i_{\rho(\be)} \xi (\al) \\
&=\d \big(i_{\rho (\beta)}\Lie_{\sigma_D}\mu(\al)-i_{\rho(\beta)} \mu(D(\al))\big)\\
&=\d \big(i_{\rho (\beta)}\Lie_{\sigma_D}\mu(\al)+i_{[\sigma_D,\rho(\al)]}\mu(\be)\big)
\end{align*}
So, we obtain
\begin{align*}
\mu\big(&D([\al,\be])-[D(\al),\be]-[\al,D(\be)]\big)=\\
&=\big(-\Lie_{\sigma_D}  i_{\rho(\be)}  \d
-\d   i_{\rho (\beta)}  \Lie_{\sigma_D}
+ \Lie_{\rho(\be)}  \Lie_{\sigma_D} +i_{[\sigma_D,\rho(\be)]}  \d\big)\mu(\al)\\
&+\big(\Lie_{\sigma_D}  \Lie_{\rho(\al)}-\d  i_{[\sigma_D,\rho(\al)]}
-i_{[\sigma_D,\rho(\al)]}  \d - \Lie_{\rho(\al)}  \Lie_{\sigma_D}\big) \mu(\be)=0,
\end{align*}
which, by injectivity of $\mu$, implies that $D$ is indeed a derivation. 

Let $D_\be=[\be,\cdot]$ be an inner derivation associated with $\be\in\Gamma(A)$, and denote $(\nu,\xi):=i_{D_{\beta}}(\mu,0)$.  Using the IM-conditions for $\mu$, we obtain
\begin{align*}
\nu(\al)&=-i_{\rho(\be)}\mu(\al)=i_{\rho(\al)}\mu(\be),\\
\xi(\al)&=\Lie_{\rho(\be)}(\mu(\al))-\mu([\be,\al])=i_{\rho(\al)}\d\mu(\be),
\end{align*}
for all $\al\in \Gamma(A)$. This is equivalent to
\begin{equation}\label{eq:contraction:derivation}
    i_{D_\be}(\mu,0)=\rho^*\mu(\be), 
\end{equation}
which proves that $i_{D_\be}(\mu,0)$ is an IM-exact 1-form.
\end{proof}

For Dirac structures the situation is similar,  but slightly more involved. At the groupoid level, for a presymplectic groupoid $(\G,\omega)$, in the commutative diagram \eqref{eq:iso:symplectic:grpd} the horizontal rows are no more isomorphisms. Instead they define a quasi-isomorphism, i.e., the induced map in cohomology is an isomorphism
\begin{equation}
    \label{eq:iso:VB:cohomology}
    \X_\mult(\G)/\delta(\Gamma(A))\diffto \Omega^1_\mult(\G)/\delta(\Omega^1(M)).
\end{equation}
This amounts to the fact that we have a short exact sequence of vector spaces
\begin{equation}
\label{eq:exact:sequence:grp:presymplectic}
    \xymatrix@R=5pt{
0\ar[r] & \Gamma(A)\ar[r] & \X_\mult(\G)\oplus \Omega^1(M)\ar[r] & \Omega^1_\mult(\G)\ar[r] & 0,
}
\end{equation}
where
\begin{align*}
     \Gamma(A)\to \X_\mult(\G)\oplus \Omega^1(M),\qquad &\be \mapsto (\delta(\be),\mu(\be)),\\
     \X_\mult(\G)\oplus \Omega^1(M)\to \Omega^1_\mult(\G),\qquad &(X,\theta) \mapsto i_X\omega-\delta(\theta).
\end{align*}

\begin{remark}
Given a closed multiplicative 2-form $\omega$ on a Lie groupoid $\G$ it is proved in \cite[Section 5.2]{dHO20} that the non-degeneracy condition $\ker\omega\cap\ker\d\t\cap\ker\d\s=0$ is equivalent to the contraction $\omega^\flat:T\G\to T^*\G$ being a VB Morita map. The authors show that such a map induces an isomorphism between the fiberwise linear VB cohomologies. The two quotients in  \eqref{eq:iso:VB:cohomology} can be identified with the fiberwise linear VB cohomologies in degree one of $T^*\G$ (left side) and $T\G$ (right side). Hence, the isomorphism \eqref{eq:iso:VB:cohomology} follows from the results in \cite{dHO20}.
\end{remark}



At the algebroid level, we have the corresponding statement.

\begin{lemma}\label{lemma:derivations:Dirac}
Let $\mu:A\to T^*M$ be a Dirac non-degenerate, closed IM 2-form. We have a short exact sequence of vector spaces 
\begin{equation}
\label{eq:exact:sequence:grp:Dirac}
\xymatrix@R=5pt{
0\ar[r] & \Gamma(A)\ar[r] & \Der(A)\oplus \Omega^1(M)\ar[r] & \Omega^1_\imult(A)\ar[r] & 0,
}
\end{equation}
where
\begin{align*}
     \Gamma(A)\to \Der(A)\oplus \Omega^1(M),\qquad &\be \mapsto ([\beta,\cdot],\mu(\be)),\\
     \Der(A)\oplus \Omega^1(M)\to \Omega^1_\imult(A),\qquad &(D,\theta) \mapsto i_D(\mu,0)-\rho^*(\theta).
\end{align*}
Moreover, the choice of a complement in $TM\oplus T^*M$ of the corresponding Dirac structure $\L$ induces a splitting of \eqref{eq:exact:sequence:grp:Dirac}.
\end{lemma}

\begin{proof}
That the sequence is indeed a cochain complex follows from \eqref{eq:contraction:derivation}
\[\beta\mapsto (D_{\beta},\mu(\beta))\mapsto  i_{D_{\beta}}(\mu,0)-\rho^*(\mu(\beta))=0.\]

We start by checking exactness at $\Gamma(A)$. Let $\be\in \Gamma(A)$ be in the kernel if the first map. Then $[\be,\cdot]=0$, which implies that $\rho(\be)=0$. Since also $\mu(\be)=0$ and the map $(\rho,\mu)$ is injective, we have that $\be=0$.


We show now surjectivity of the last map. For this, consider $(\nu,\xi)\in \Omega^1_{\imult}(A)$, and we try to solve
\begin{equation}\label{eq:to:find:derivation}
(\nu,\xi)=i_D(\mu,0)-\rho^*(\theta),
\end{equation}
with $D\in \Der(A)$ and $\theta\in \Omega^1(M)$. The first component can be written as
\begin{equation}\label{eq:used:for:uniqueness}
(\rho,\mu)^*(\theta, \sigma_{D})=-\nu.
\end{equation}
By assumption, $(\rho,\mu)$ is injective, therefore its dual map is surjective. So a solution to the above equation exists, and we fix such a solution $(\theta,\sigma_{D})$. The derivation $D$ will be determined by the following equations
\begin{equation}\label{eq:which:determine:D}
\left\{
\begin{array}{l}
\rho(D(\al))=[\sigma_{D},\rho(\al)]
\\
\\
\mu(D(\al))=\Lie_{\sigma_{D}}(\mu(\al))-i_{\rho(\al)}\d\theta-\xi(\al)
\end{array}
\right.
\end{equation}
where the first equation holds for any derivation and the second is the second from \eqref{eq:to:find:derivation}. By assumption, the map $(\rho,\mu)$ is an isomorphism from $A$ to the Dirac structure $\L$. Therefore, if we show that the right-hand side of the above set of equations defines an element in $\L$, then it follows that the system determines a unique $D(\al)\in \Gamma(A)$, for any $\al\in \Gamma(A)$. We need to check that
\[
[\sigma_{D},\rho(\al)]+\Lie_{\sigma_{D}}(\mu(\al))-i_{\rho(\al)}\d\theta-\xi(\al)\in \L
\]
for all $\al\in  \Gamma(A)$. Since $\L$ is Lagrangian, this is equivalent to
\[\langle [\sigma_{D},\rho(\al)],\mu(\be)\rangle+\langle \rho(\be), \Lie_{\sigma_{D}}(\mu(\al))\rangle= \langle \rho(\be),i_{\rho(\al)}\d\theta+\xi(\al)\rangle,\]
for all $\al,\be \in \Gamma(A)$. This is proven as follows
\begin{align*}
&\langle [\sigma_{D},\rho(\al)],\mu(\be)\rangle+\langle \rho(\be), \Lie_{\sigma_{D}}(\mu(\al))\rangle=-\Lie_{\rho(\al)}\langle \sigma_{D},\mu(\be)\rangle\\
&+\langle \sigma_{D},\Lie_{\rho(\al)}\mu(\be)\rangle+i_{\rho(\be)}\d \langle \sigma_{D},\mu(\al)\rangle - \langle \sigma_{D}, i_{\rho(\be)}\d \mu(\al)\rangle\\
&=(i_{\sigma_{D}}\mu)([\al,\be])-\Lie_{\rho(\al)} (i_{\sigma_{D}}\mu)(\be)+i_{\rho(\be)}\d (i_{\sigma_{D}}\mu)(\al)\\
&=i_{\rho(\be)}\big( \xi(\al)+i_{\rho(\al)}\d \theta\big) 
\end{align*}
where we have used that $(\mu,0)$ is an IM 2-form, that $i_{\sigma_D}\mu=-\nu-i_{\rho(\cdot)}\theta$, and so it is the first component of the IM 1-form 
\[(i_{\sigma_D}\mu, -\xi-i_{\rho(\cdot)}\d\theta)=-(\nu,\xi)-\rho^*(\theta).\] Thus we obtain a well-defined operator $D$ on $\Gamma(A)$. The defining equations imply that $D$ satisfies the Leibniz rule with symbol $\sigma_{D}$. Since $D$ satisfies \eqref{eq:anchor:derivation}, we obtain
\[\rho\big(D([\al,\be])-[D(\al),\be]-[\al,D(\be)]\big)=0.\]
Applying the same steps as in the proof of Lemma \ref{lemma:1:to:1:derivations:IM}, we also obtain
\[\mu\big(D([\al,\be])-[D(\al),\be]-[\al,D(\be)]\big)=0.\]
Hence, by injectivity of the map $(\rho,\mu)$, $D$ is indeed a derivation. 

Note that the only choice we made in the construction of $D$ was the solution $(\sigma_D,\theta)$ of \ref{eq:used:for:uniqueness}. The solution is unique up to a section of the Dirac structure $\L=\mathrm{Im}(\rho,\mu)$, i.e., up to a section of the form $(\rho,\mu)(\be)$, for some $\be\in\Gamma(A)$. Changing the solution this way, we obtain a new derivation $D+F$. The relations \eqref{eq:which:determine:D} for $D$ and $D+F$, imply that $F$ has to satisfy \[
\left\{
\begin{array}{l}
\rho(F(\al))=[\rho(\be),\rho(\al)] = \rho([\be,\al])
\\
\\
\mu(F(\al))=\Lie_{\rho(\be)}\mu(\al)-i_{\rho(\al)}\d\mu(\be)=\mu ([\be,\al])
\end{array}
\right.
\]
Injectivity of $(\rho,\mu)$ shows that $F=D_{\be}$. This implies that the sequence \eqref{eq:exact:sequence:grp:Dirac} is exact also at middle point.

Finally, the choice of a complement $C\subset TM\oplus T^*M$ of $\L$ yields a unique solution of \eqref{eq:used:for:uniqueness} which lies in $C$. In other words, we obtain a splitting of the last map. 
\end{proof}

\subsection{The infinitesimally multiplicative Moser method}

We give now a version of Moser's Theorem for IM 2-forms on linear action algebroids
\[A_S\ltimes \Rep \Ato \Rep,\]
where $\Rep$ is a representation of $A_S\Ato S$.

\begin{proposition}
\label{prop:Moser:algbrd}
Consider two closed IM 2-forms 
\[\mu_k\in\Omega^2_\imult(A_S\ltimes M_k),\quad k=0,1,\]
defined on open neighborhoods $S\subset M_k\subset \Rep$, such that their pullbacks along the Lie algebroid embedding $i:A_S\hookrightarrow A_S\ltimes \Rep$ coincide\[ i^*\mu_0=i^*\mu_1\in \Omega^2_{\imult}(A_S).\]
\begin{enumerate}[(a)]
\item If $\mu_0$ and $\mu_1$ are non-degenerate, then there are open neighborhoods $S\subset M_k'\subset M_k$, $k=0,1$, and an isomorphism of Lie algebroids $\Phi:A_S\ltimes M_0'\diffto A_S\ltimes M_1'$, whose base map fixes $S$, and such that
\[ \Phi^*(\mu_1)=\mu_0. \]
\item If $\mu_0$ and $\mu_1$ are Dirac non-degenerate, then there are open neighborhoods $S\subset M_k'\subset M_k$, $k=0,1$, an isomorphism of Lie algebroids $\Phi:A_S\ltimes M_0'\diffto A_S\ltimes M_1'$, whose base map fixes $S$, and an exact 2-form $\d\lambda$, whose pullback to $S$ vanishes, such that
\[ \Phi^*(\mu_1)=\mu_0+\rho_{\ltimes}^*(\d\lambda). \]
Hence, the Dirac structures $\L_k:=\im (\rho_{\ltimes},\mu_k)$ are equivalent around $S$, in the sense of Definition \ref{def:iso:Dirac}.
\end{enumerate}
\end{proposition}

For the proof we need the analog of Lemma \ref{lem:primitive} for IM primitives.

\begin{lemma}\label{Poincare:for:action:algd} Let $S\subset M\subset \Rep$ be an open subset that is invariant under multiplication $m_{\lambda}$ by $\lambda\in [0,1]$. Then any closed IM form $\al\in\Omega^k_\imult(A_S\ltimes M)$ whose pullback along the zero section $i: A_S\hookrightarrow A_S\ltimes M$ vanishes
\[i^*\al=0\in \Omega^k_{\imult}(A_S),\]
is exact for the IM differential
\[(\al,0)=\d_\imult (\be,\al), \quad\text{with}\quad (\be,\al)\in\Omega^{k-1}_\imult(A_S\ltimes M). \]
Moreover, we can take $\be$ to vanish on $A_S$.
\end{lemma}

\begin{proof}
We use an IM version of the homotopy operator of the proof of Lemma \ref{lem:primitive}. For that, note that the scalar multiplication by $\lambda\in [0,1]$
\[ m_{\lambda}:A_S\ltimes M\to A_S\ltimes M,\quad (\al,v)\mapsto (\al,\lambda v), \]
is a family of Lie algebroid morphisms. This family is generated by the algebroid derivation $D:\Gamma(A_S\ltimes M)\to \Gamma(A_S\ltimes M)$ which vanishes on constant sections and has symbol the Euler vector field $X$ of $\Rep \to S$. Since $D$ vanishes along the subalgebroid $i:A_S\hookrightarrow A_S\ltimes M$, we have that:
$i^*(i_{D}(\mu,\xi))=0$, for any $(\mu,\xi)\in\Omega^{\bullet}_\imult(A_S\ltimes M)$. Hence $m_0^*(i_{D}(\mu,\xi))=0$, and so the family $\frac{1}{\lambda}m_\lambda^*(i_{D}(\mu,\xi))$ is smooth at $\lambda=0$. So we can define the map
\[H:\Omega^{\bullet}_{\imult}(A_S\ltimes M)\to \Omega^{\bullet-1}_{\imult}(A_S\ltimes V)\]
\[ H(\mu,\xi):=\int_0^1\frac{1}{\lambda} m_\lambda^*(i_{D}(\mu,\xi))\, \d \lambda= \Big(-\int_0^1\frac{1}{\lambda} m_\lambda^* i_{X}\mu\, \d \lambda, \int_0^1\frac{1}{\lambda} m_\lambda^*(\Lie_{D}\mu+i_{X}\xi)\, \d \lambda\Big).\]
This gives an homotopy operator
\[ \id-(i\circ P)^*=\d_\imult H+H\d_\imult, \]
where $P: A_S\ltimes M\to A_S$ is the projection.

So if $\al$ is a closed IM form whose pullback to $A_S$ vanishes, we find 
\[ (\al,0)=\d_\imult H(\al,0)=\d_\imult (\be,\al), \]
where
\begin{equation}\label{formula:for:beta}
 \be=-\int_0^1\frac{1}{\lambda} m_\lambda^* i_{X}\al\, \d \lambda. 
 \end{equation}
Note that $\be$ vanishes on $A_S$. 
\end{proof}

Next we give an analog of Lemma \ref{lemma:omega:along:zero}:

\begin{lemma}\label{lemma:how:to:fix:along:zero:section}
Let $\Rep \to S$ be a representation of $A_S$, $M\subset \Rep$ an open neighborhood of $S$, $\mu\in \Omega^2_{\imult}(A_S\ltimes M)$ be a Dirac non-degenerate closed IM 2-form, and denote $\mu_S:=i^*\mu\in \Omega^2_{\imult}(A_S)$. Then $\ka:=\ker \mu_S\cap \ker\rho_S\subset A_S$ is a bundle of ideals, and
\[\psi:\ka \to \Rep^*,\quad \xi\mapsto \pr_{\Rep^*}(\mu(\xi))\]
is an isomorphism of $A_S$-representations, where $\pr_{\Rep^*}$ is the projection corresponding to the canonical decomposition $T_S^*\Rep=T^*S\oplus \Rep^*$. 
\end{lemma}

\begin{proof}Since the zero section $S\subset \Rep$ is an invariant submanifold for $A_S\ltimes \Rep$, it follows that $S\subset M$ is an invariant submanifold also for the Dirac structure $\L:=\im(\rho_{\ltimes},\mu)$. As explained in Section \ref{sec:Dirac}, this implies that $\mu_S$ is robust and $\ka$ is a bundle of ideals. Also it is easy to see that $\ka$ and $\Rep$ have the same rank. We show now that $\psi$ is injective. If $\xi \in \ka$ and $\psi(\xi)=0$, it follows that $\mu(\xi)=(\mu_S(\xi),\psi(\xi))=0$ and $\rho_{\ltimes}(\xi)=\rho_S(\xi)=0$. Because $\mu$ is Dirac non-degenerate, $\xi=0$. 

It remains to show that $\psi$ is $A_S$-equivariant, i.e., for all $\al\in \Gamma(A_S)$ and $\be\in \Gamma(\ka)$
\begin{equation}
\label{eq:AS:equivariant}
 \psi([\al,\be]_{A_S})=\nabla_{\alpha}^{\Rep^*} \psi (\beta).   
\end{equation}
To see this, for a section $\alpha\in \Gamma(A_S)$, we denote by $\overline{\alpha}\in \Gamma(A_S\ltimes \Rep)$ its pullback along the projection $\Rep\to S$. If $\alpha\in \Gamma(A_S)$ and $\beta\in \Gamma(\ka)$, from the definition of the bracket $[\cdot,\cdot]_{\ltimes}$ of $A_S\ltimes \Rep$, we obtain
\[\mu([\overline{\alpha},\overline{\beta}]_{\ltimes})|_{S}=\mu(\overline{[\alpha,\beta]_{A_S}})|_S=\psi([\al,\be]_{A_S}).\]
On the other hand, since $\mu$ is a closed IM 2-form \eqref{eq:mult:form}, we obtain
\[\mu([\overline{\alpha},\overline{\beta}]_{\ltimes})\big|_{S}=
\big(\Lie_{\rho_{\ltimes}(\overline{\alpha})}\mu(\overline{\beta})-
i_{\rho_{\ltimes}(\overline{\beta})}\d \mu(\overline{\alpha})\big)\big|_{S}.
\]
The last term vanishes because $\rho_{\ltimes}(\overline{\beta})|_{S}=\rho_{A_S}(\beta)=0$. To calculate the first term, consider a section $\xi\in \Gamma(\Rep)$, and denote the corresponding vertical vector field on $\Rep$ by $\overline{\xi}\in \mathfrak{X}(\Rep)$. Then we have that
\begin{align*}
\big\langle \xi, \Lie_{\rho_{\ltimes}(\overline{\alpha})}\mu(\overline{\beta})\big|_{S} \big\rangle = & i_{\overline\xi }\Lie_{\rho_{\ltimes}(\overline{\alpha})}\mu(\overline{\beta})\big|_{S}\\
= & \Lie_{\rho_{\ltimes}(\overline{\alpha})}i_{\overline\xi }\mu(\overline{\beta})\big|_{S}+
 i_{[\overline\xi, \rho_{\ltimes}(\overline{\alpha})]}\mu(\overline{\beta})\big|_{S}\\
 = &
 \Lie_{\rho_{A_S}(\alpha)}\big\langle \xi , \psi (\beta) \big\rangle -
 \big\langle \nabla_{\alpha}^{\Rep}\xi , \psi (\beta)\big\rangle \\
 = &\big\langle \xi ,\nabla_{\alpha}^{\Rep^*} \psi (\beta) \big\rangle,
\end{align*}
where in the last equality we have used the definition of the dual connection, and in the one before we have used the following properties of the anchor of $A_S\ltimes \Rep$
\[\rho_{\ltimes}(\overline{\alpha})|_S=\rho_{A_S}(\alpha),\quad [\rho_{\ltimes}(\overline{\alpha}),\overline\xi]= \overline{\nabla_{\alpha}^{\Rep^*}\xi}.\]
This proves \eqref{eq:AS:equivariant} and concludes the proof.
\end{proof}

We have now all the ingredients to conclude the proof of Proposition \ref{prop:Moser:algbrd}.

\begin{proof}[Proof of Proposition \ref{prop:Moser:algbrd}]

We will prove part (b), and show along the way what needs to be changed to obtain part (a).

Denote by $\mu_S:=i^*(\mu_k)\in \Omega^2_{\imult}(A_S)$, and denote $\ka:=\ker\rho_S\cap \ker\mu_S$. Lemma \ref{lemma:how:to:fix:along:zero:section} gives linear, $A_S$-equivariant isomorphisms $\psi_k:\ka\diffto \Rep^*$, which yield Lie algebroid isomorphisms
$(\id,(\psi_k)^*):A_S\ltimes \Rep\diffto A_S\ltimes \ka^*$. By pushing $\mu_k$ forward along these maps, we may assume that $\Rep=\ka^*$, and that $\mu_k$ satisfy $\mu_k|_{\ka}=\id_{\ka}$. This implies that the entire path
\[\mu_t:=(1-t)\mu_0+t\mu_1\in \Omega^2_{\imult}(A_S\ltimes M_0\cap M_1),\quad t\in [0,1] \]
satisfies $\mu_t|_{\ka}=\id_{\ka}$ and $i^*(\mu_t)=\mu_S$. This implies that there exists an open neighborhood $S\subset M$ so that $\mu_t|_{A_S\ltimes M}$ is Dirac non-degenerate for all $t\in [0,1]$. 

Next, we adapt the Moser-type argument to the multiplicative setting and to account for gauge transformations. We look for a path of Lie algebroid automorphisms $\Phi^t:A_S\ltimes M\to A_S\ltimes M$, starting at the identity at $t=0$, and a path of 1-forms $\lambda_t$ on $M$, with $\lambda_0=0$, such that
\begin{equation}\label{eq:moser:with:form}
 (\Phi^t)^*(\mu_t)=\mu_0+\rho_{\ltimes}^*(\d\lambda_t),\quad \forall t\in [0,1].
\end{equation}
Here we have denoted by $\rho_{\ltimes}^*(\d\lambda_t)$ is the closed IM 2-form 
$\alpha\mapsto (i_{\rho_{\ltimes}(\alpha)}\d\lambda_t,0)$, which is also the pullback via the Lie algebroid map $\rho_{\ltimes}$ of the (closed IM) 2-form $\d\lambda_t$. This IM 2-form is exact, with primitive $(\rho_{\ltimes}^*(\lambda_t),\rho_{\ltimes}^*(\d \lambda_t))$. In the non-degenerate case (a), we will see that one can take $\lambda_t\equiv 0$. The time-dependent algebroid derivation $D_t$ generating the isotopy $\Phi^t$ will satisfy the infinitesimal equation
\[ \d_\imult\big( i_{D_t}(\mu_t,0)-(\rho_{\ltimes}^*(\theta_t),\rho_{\ltimes}^*(\d\theta_t)) \big)=(\mu_0-\mu_1,0),\]
where $\theta_t=(\phi^t)_*(\frac{\d}{\d t}\lambda_t)$ and $\phi^t$ it the base map of $\Phi^t$. By Lemma \ref{Poincare:for:action:algd}, after shrinking $M$ we find $\beta\in \Gamma(A_S^*\ltimes M)$ (given by \eqref{formula:for:beta}) so that $(\be,\mu_0-\mu_1)$ is an IM 1-form. Hence it suffices to solve the equation
\[ i_{D_t}(\mu_t,0)-(\rho_{\ltimes}^*(\theta_t),\rho_{\ltimes}^*(\d\theta_t)) =(\be,\mu_0-\mu_1). \]

In the non-degenerate case (a), Lemma \ref{lemma:1:to:1:derivations:IM} show that this equation has a unique solution $D_t$ with $\theta_t\equiv 0$, which depends smoothly on $t$. Because $\beta$ vanishes along $S$, the proof of Lemma \ref{lemma:1:to:1:derivations:IM} shows that $\sigma_{D_t}$ vanishes along $S$.

In the Dirac non-degenerate case (b), Lemma \ref{lemma:derivations:Dirac} shows solutions $(D_t,\theta_t)$ exists. If we fix a metric on $TM\oplus T^*M$, then we can use the orthogonal complement to $\L_t=\im(\rho_{\ltimes},\mu_t)$ to obtain a smooth family of solutions $(D_t,\theta_t)$. Moreover, it follows that $\sigma_{D_t}$ and $\theta_t$ vanish along $S$.

Since $\sigma_{D_t}$ vanishes along $S$, its flow $\varphi^t$ is defined up to time 1 on a possible smaller neighborhood $S\subset M'_0\subset M$, and fixes $S$. Also the flow $\Phi^t$ of $D_t$ is defined on $A_S\ltimes M'_0$ and satisfies \eqref{eq:moser:with:form} for $\lambda_t:=\int_0^t(\phi^s)^*\theta_s\d s$. Since $\theta_t$ vanishes along $S$ and $\phi^t$ fixes $S$, it follows that also $\lambda_t|_S=0$, and so the pullback of $\d \lambda$ to $S$ vanishes.

The required isomorphism is $\Phi:=\Phi^1$, which covers $\varphi:=\varphi^1$,  $M'_1:=\varphi(M'_0)$ and the 2-form is $\d\lambda:=\d\lambda_1$. The equivalence of Dirac structure around $S$ is
\[\varphi:(M'_0,e^{\d\lambda}(\L_0))\diffto (M'_1,\L_1).\qedhere\]
\end{proof}

\section{Linearization in Dirac geometry}
\label{sec:Dirac}

In this section, we show that many of the results and constructions of the paper can be extended naturally from Poisson geometry to the setting of Dirac manifolds (for a brief introduction to Dirac geometry see, e.g., \cite{Bu13}). 

For a Dirac manifold $(M,\L)$, the notion corresponding to a Poisson submanifold is that of an \textbf{invariant submanifold}, i.e., a submanifold  $S\subset M$ such that \[\pr_{TM}a\in T_xS,\quad\forall\ a\in \L_x, \ x\in S.\]
Equivalently, $S$ is an invariant submanifold for the Lie algebroid $\L$. Obviously, saturated submanifolds, i.e., submanifolds which are unions of presymplectic leaves of $\L$, are examples of invariant submanifolds. In fact, any closed invariant submanifold is saturated. One of the main results of this section is the following.

\begin{theorem}
\label{thm:normal:form:Dirac}
Let $(M,\L)$ be a Dirac manifold and $S\subset M$ be an embedded saturated submanifold. If $\L$ is integrable by a proper Lie groupoid, then $\L$ is linearizable around $S$.
\end{theorem}

Let us explain what is the local model underlying this result. If $S$ is an invariant submanifold of $\L$, we
have the restricted Lie algebroid $A_S:=\L|_S$. Moreover, $S$ inherits a Dirac structure which makes the inclusion 
\[i:(S,\L_S)\hookrightarrow (M,\L)\] simultaneously a forward Dirac map and a backward Dirac map. We obtain a short exact sequence of Lie algebroids
\begin{equation}\label{eq:short:exact:Dirac}
0\rmap \ka\rmap A_S\stackrel{i^{!}}\rmap \L_S\rmap 0,
\end{equation}
where $\ka=(TS)^{\circ}=\nu^*(S)$ is the conormal bundle of $S$ in $M$, and $i^!(v+\xi)=v+\xi|_{TS}$. This short exact sequence represents the \emph{first order jet} of $\L$ at $S$. As for Poisson structures, we will encode these data more efficiently, using the closed IM 2-form
\[\mu_S:A_S\to T^*S, \quad \mu_S(v+\xi)=\xi|_{TS}.\]



Then $i^!=(\rho_{A_S},\mu_S)$, and the short exact sequence \eqref{eq:short:exact:Dirac} becomes:
\begin{equation}\label{eq:short:exact:Dirac2}
0\rmap \ka\rmap A_S\stackrel{(\rho_{A_S},\mu_S)}\rmap \L_S\rmap 0.
\end{equation}

Hence, $\mu_S$ is a \textbf{robust} closed IM 2-form, as defined in Example 
\ref{example:robust}. The following notion describes the situation at hand.

\begin{definition}
A \textbf{first order jet} of a Dirac structure is a pair $(A_S,\mu_S)$, consisting of a Lie algebroid and a closed, robust IM 2-form.

If the bundle of ideals $\ka:=\ker\rho_{A_S}\cap \ker\mu_S$ is partially split (see Subsection \ref{sec:local:model:description}), then we say that the first order jet of a Dirac structure $(A_S,\mu_S)$ is \textbf{partially split}. 
\end{definition}



In the Dirac setting, the \textbf{local model} is defined similarly as in Proposition \ref{prop:local:model:algbrd}. Namely, let $(A_S,\mu_S)$ be a partially split first order jet of a Dirac structure. Given an IM connection 1-form $(L,l)\in \Omega^1(A_S;\ka)$, define the closed IM 2-form
\[\mu_0:=\pr^*\mu_S+\d_{\imult}\langle(L,l),\cdot \rangle\in \Omega^2_\imult(A_S\ltimes\ka^*).\]
In a neighborhood $M_0\subset \ka^*$ of $S$, $\mu_0$ is Dirac non-degenerate, i.e., the map
\begin{equation}\label{eq:map:anchor:IM:local:model} 
(\rho_{\ltimes},\mu_0):A_S\ltimes M_0 \to TM_0\oplus T^*M_0
\end{equation}
is injective and its image is a Dirac structure $\L_0$ on $M_0$ (as defined in Example 
\ref{example:robust}). Then $S$ is an invariant submanifold for $\L_0$, the map \eqref{eq:map:anchor:IM:local:model} restricts to a Lie algebroid isomorphism along $S$, $A_S\diffto \L_0|_S$, and the pullback of the canonical IM 2-form on $\L_0$ to $A_S$ is $\mu_S$. This justifies calling $(M_0,\L_0)$ the \textbf{local model} of $(A_S,\mu_S)$ corresponding to the IM connection 1-form $(L,l)$.

The Dirac-geometric version of Proposition \ref{prop:partially:split:algbrd:optimal} also holds. Its proof works with minimal adaptations. Namely, a first order jet $(A_S,\mu_S)$ of a Dirac structure with kernel $\ka:=\ker\rho_{A_S}\cap\ker\mu_S$ is partially split if and only if there exists a closed IM 2-form $\mu\in\Omega^2_\imult(A_S\ltimes M_0)$, with $S\subset M_0\subset \ka^*$ an open set, such that
\begin{itemize}
\item $\mu$ is Dirac non-degenerate;
\item $i^*(\mu)=\mu_S$, where $i:A_S\hookrightarrow A_S\ltimes \ka^*$ is the zero-section. 
\end{itemize}


%

For normal forms around submanifolds, or equivalences of local models, in Dirac geometry, one needs a more general notion of isomorphism.

\begin{definition}\label{def:iso:Dirac}
Let $i_0:S\hookrightarrow M_0$ and $i_1:S\hookrightarrow M_1$ be injective immersions. 
The Dirac structures $\L_0$ on $M_0$ and $\L_1$ on $M_1$ are said to be \textbf{equivalent around $S$}, if there is a diffeomorphism of Dirac structure
\[\varphi:\big(U_0,e^{\omega}\L_0\big)\diffto \big(U_1,\L_1\big),\]
where $U_k\subset M_k$ is an open neighborhood of $i_k(S)$, $\varphi\circ i_0=i_1$, and $\omega\in \Omega^2(U_0)$ is a closed 2-form such that $i_0^*\omega=0$.
\end{definition}

\begin{remark}\label{remark:gauge:Poisson}
After shrinking $U_0$, the 2-form $\omega$ admits a primitive $\lambda\in \Omega^1(U_0)$ which vanishes at points in $S$ (see the Relative Poincar\'e Lemma in \cite{WeinLSM}).

For Poisson manifolds, one can often ``absorb'' $\omega$ in the diffeomorphism $\varphi$, and so the relation reduces to that of being ``Poisson diffeomorphic around $S$". To see this, let $(M,\pi)$ be a Poisson manifold, $i:S\hookrightarrow M$ a submanifold and $\omega$ a closed 2-form such that $i^*\omega=0$. Let us assume that the maps
\begin{equation}\label{eq:map:gauge}
\id+t\, \omega^{\flat}\circ\pi^{\sharp}:T^*_SM\to T^*_SM\qquad (t\in [0,1])
\end{equation}
are invertible. This is equivalent to the Dirac structures $e^{t\omega}\L_{\pi}$ to correspond to Poisson structures $\pi_t$ around $S$. These satisfy $\pi_t^{\sharp}=\pi^{\sharp}\circ (\id+t\, \omega^{\flat}\circ\pi^{\sharp})^{-1}$. Let $\lambda$ be a primitive of $\omega$ which vanishes along $S$. By the Moser path-method for Poisson structures 
(see, e.g., \cite[Section 2.4]{Me18}), the isotopy $\phi^t$ generated by the time-dependent vector field $-\pi_t^{\sharp}(\lambda)$ sends $\pi=\pi_0$ to $\pi_t$ and fixes $S$ pointwise.

The assumption that \eqref{eq:map:gauge} is invertible is satisfied in any of the following cases:
\begin{itemize}
\item If $\omega$ vanishes at points in $S$;
\item If $S$ is a coisotropic submanifold of $(M,\pi)$ and $i_v\omega=0$ for all $v\in TS$;
\item If $S$ is a Poisson submanifold of $(M,\pi)$.
\end{itemize}
\end{remark}


We show that equivalent structures have isomorphic first order jets.

\begin{lemma}\label{lemma:equivalent:give:equivalent:jet}
Consider two Dirac manifolds $(M_k,\L_k)$, $k=0,1$, that are equivalent around the invariant, embedded submanifolds $i_k:S\hookrightarrow M_k$. Then the first order jets of $\L_0$ and $\L_1$ at $S$ are isomorphic. In particular, the first order jet of $\L_0$ at $S$ is partially split if and only if that of $\L_1$ at $S$ is partially split.
\end{lemma}

\begin{proof}
An equivalence $(\varphi,\omega)$ between $\L_0$ and $\L_1$ gives an isomorphism between the first order jets of $e^{\omega}\L_0$ and $\L_1$ at $S$. Therefore, it suffices to show that $\L_0$ and $e^{\omega}\L_0$ have isomorphic first order jets at $S$. Denote these by 
\[(A_S:=\L_0|_S, \mu_S:=i_0^*\pr_{T^*M_0})\quad  \textrm{and} \quad (A'_S:=(e^{\omega}\L_0)|_S, \mu_S':=i_0^*\pr_{T^*M_0}).\]
The map $e^\omega$ restricts to a Lie algebroid isomorphism between $A_S$ and $A'_S$, and the pullback of $\mu'_S$ under this isomorphism is given by
\[\tilde{\mu}_S(a)=\mu_S(a)+i_0^*\big(i_{\rho(a)}\omega\big).\]
Since $\rho(a)\in TS$ and $i_0^*\omega=0$, we have $\tilde{\mu}_S=\mu_S$. This completes the proof. 
\end{proof}

The following result, which is proven in the Appendix using an IM version of the Moser argument (see Proposition \ref{prop:Moser:algbrd}), shows that the local model is well-defined up to the equivalence relation from Definition \ref{def:iso:Dirac}.

\begin{proposition}\label{prop:moser:Dirac}
Let $(A_S,\mu_S)$ be a first order jet of a Dirac structure with kernel $\ka:= \ker \rho_{A_S}\cap\ker\mu_S$. Let $\mu_0,\mu_1\in\Omega^2_\imult(A_S\ltimes M)$ be 
closed IM 2-forms, defined on an neighborhood $M\subset \ka^*$ of $S$, that are Dirac non-degenerate and extend $\mu_S$
\[ i^*\mu_k=\mu_S, \quad (k=0,1).\]
Then the Dirac structures $\L_k:=\mathrm{Im}(\rho_{\ltimes},\mu_k)$ are equivalent around $S$. 
\end{proposition}

\begin{definition}
A Dirac structure $(M,\L)$ is called \textbf{linearizable} around an invariant submanifold $S\subset M$, if its first order jet $(A_S:=\L|_S,\mu_S:=\pr_{T^*S})$ at $S$ is partially split, and $\L$ is equivalent around $S$ with the local model $\L_0$ corresponding to some (hence, any) IM connection 1-form $(L,l)\in \Omega^1(A_S;\ka)$.
\end{definition}

Proposition \ref{prop:moser:Dirac} implies the analog of Theorem \ref{thm:normal:form:algebroid} for the Dirac setting.

\begin{theorem}\label{thm:normal:form:Dirac:algebroid}
A Dirac manifold $(M,\L)$ is linearizable around an invariant submanifold $S$ if and only if the  Lie algebroid $\L$ is linearizable around $S$.
\end{theorem}

This result and the linearization theorem for proper Lie groupoids from \cite{dHFe18} imply Theorem \ref{thm:normal:form:Dirac}, stated at the beginning of this section.

\medskip 

Finally, we briefly discuss the notions at the groupoid level which correspond to Dirac structures and their first order jets. We omit the discussion about partially split groupoids, because it can be found in a more general setting in \cite{FM22} and we omit the discussion about local models and equivalences, because these admit similar generalizations from the Poisson setting. 

At a global level, Dirac structures correspond to \textbf{presymplectic groupoids} \cite{BCWZ04}, i.e., a Lie groupoid $\G\tto M$ with a closed, multiplicative 2-form $\omega$ satisfying
\[\dim(\G)=2\dim(M),\quad \ker(\omega)\cap \ker(\d\s)\cap \ker(\d\t)=0.\]

The notion corresponding to first order jets of Dirac structures, i.e., closed, robust IM 2-forms, was also introduced in \cite{BCWZ04}.

\begin{definition}
A closed multiplicative 2-form $\omega_S$ on a Lie groupoid $\G_S\tto S$ is said to be \textbf{robust} if $\ker(\omega_S)\cap \ker(\d\s)\cap \ker(\d\t)$ has constant rank, equal to 
$\mathrm{dim}(\G_S)-2\dim(S)$. A Lie groupoid $\G_S\tto S$ endowed with a robust 2-form $\omega_S$ is called an \textbf{over-presymplectic} groupoid.
\end{definition}

Equivalent characterizations of these objects are given in \cite[Lemma 4.7]{BCWZ04}, which implies also the following result.

\begin{proposition}
Let $\G_S\tto S$ be a Lie groupoid with Lie algebroid $A_S$. Then
\begin{enumerate}[(i)]
    \item A closed multiplicative 2-form $\omega_S$ on $\G_S$ is robust if and only if the corresponding IM 2-form $\mu_S:A_S\to T^*S$ is robust. 
    \item If $\G_S$ has 1-connected target-fibers, then differentiation $\omega_S\mapsto \mu_S$ (see \eqref{eq:M:IM:forms}) gives a 1-to-1 correspondence between closed, robust, multiplicative 2-forms on $\G_S$ and closed, robust, IM 2-forms on $A_S$.
\end{enumerate}
\end{proposition}

\section{Linearization around coregular submanifolds}
\label{sec:coregular}

In this section, we consider local models and linearization around more general submanifolds, called coregular submanifolds. The splitting theorem around transversals of \cite{BLM19}, reduces the study of this class to the case of invariant submanifolds, discussed before. We will discuss the general setting of Lie algebroids, and then the setting for Dirac structures. For example, we will prove the following. 

\begin{theorem}\label{Theorem:coregular:normal:form}\mbox{}
\begin{enumerate}[(i)]
\item The Lie algebroid of a proper Lie groupoid is linearizable around any coregular submanifold.
\item A Dirac manifold $(M,\L)$, where $\L$ is the Lie algebroid of a proper groupoid, is linearizable around any coregular submanifold.
\end{enumerate}
\end{theorem}

In both settings, we take advantage of the existence of the pullback operation. The reader should bare in mind that now we work modulo more general equivalences, that include gauge transformations, as in Definition \ref{def:iso:Dirac} (see also Remark \ref{remark:gauge:Poisson}, where the passage of this equivalence relation to the setting of Poisson structures is described).
 

\subsection{Transversals}

An embedded submanifold $i:S\hookrightarrow  M$ is called a \textbf{transversal} for a Lie algebroid $A\Ato M$, if it is transverse to the anchor $\rho$ of $A$, i.e., if
\[T_xS+\rho_x(A_x)=T_xM,\quad \forall \, x\in S.\]
Transversals lie at the other end of the spectrum compared to invariant submanifolds. They admit a very simple local model. First of all, the pullback of $A$ to $S$ yields a Lie algebroid over $S$
\[A_S:=i^!(A)=\rho^{-1}(TS)\Ato S.\]
The \textbf{local model} of $A$ around $S$ is the pullback of the Lie algebroid $A_S$ to the normal bundle via the projection $p:\nu(S,M)\to S$, i.e.,
\[p^!A_S=T\nu(S,M)\times_{TS}A_S\Ato \nu(S,M),\]
where, in this subsection, we include the ambient manifold in the notation for the normal bundle. The normal form theorem, obtained in \cite{BLM19} (see also \cite{Fr19} for a different proof) holds without restrictions.
\begin{theorem}[\cite{BLM19}]\label{theorem:BLM}
Let $A\Ato M$ be a Lie algebroid, and $i:S\hookrightarrow M$ a transversal. Then $A$ is isomorphic around $S$ to the local model $p^!A_S\Ato \nu(S,M)$.
\end{theorem}

The analog of this result in the Dirac setting also holds. In order to explain this, recall that given a map $f:N\to M$ and a Dirac structure $\L$ on $M$, one has
\begin{itemize}
    \item the pullback of the \emph{Lie algebroid} $\L$ to $N$
    \[ (f^!\L)_{\textrm{alg}}=TN\times_{TM}\L;\]
    \item the pullback of the \emph{Dirac structure} $\L$ to $N$
    \[(f^!\L)_{\textrm{Dir}}:=\{v+(\d f)^*(\xi)\, :\, (\d f)(v)+\xi\in \L\}.\]
\end{itemize}
In general, these are not smooth objects. However, if $f$ is transverse to $\L$, i.e., if
\[\im \d_{x}f+\pr_{TM}\L_{f(x)}=T_{f(x)}M, \quad \forall x\in N,\]
then these objects are smooth and the Lie algebroids are canonically isomorphic. In this case, we will not distinguish between them and we will denote them by $f^!\L$.



Let $(M,\L)$ be a Dirac manifold. The local model of $\L$ around a transversal $i:S\hookrightarrow M$ is the Dirac pullback $p^!i^!\L$ to the normal bundle $p:\nu(S,M)\to S$. Again, the normal form theorem always holds (see \cite{BLM19,FrMa18B}, and \cite{FrMa17} for Poisson transversals).
\begin{theorem}[\cite{BLM19}]\label{theorem:BLM:Dirac}
If $i:S\hookrightarrow (M,\L)$ is a transversal in a Dirac manifold, then $\L$ and $p^!i^!\L$ are equivalent Dirac structures around $S$ (see Definition \ref{def:iso:Dirac}). 
\end{theorem}

\subsection{Coregular submanifolds of Lie algebroids}

For a Lie algebroid, we consider the following class of submanifolds which includes transversals and embedded invariant submanifolds as extremes. Moreover, any submanifold has an open dense subset whose components are of this type. 
 
\begin{definition}
Let $A\Ato M$ be a Lie algebroid. We call an embedded submanifold $i_S:S\hookrightarrow M$  \textbf{coregular} if the vector spaces
\[T_xS+\rho_x(A_x),\quad x\in S,\]
have the same dimension.
\end{definition}

Equivalently, $i_S:S \hookrightarrow M$ is a coregular submanifold if and only if the pullback of $A$ to $S$ is a smooth Lie subalgebroid of $A$
\[A_S:=i_S^!A=\rho^{-1}(TS).\]
In this case, the anchor of $A$ yields a short exact sequence of vector bundles
\begin{equation}\label{eqseq}
0\rmap A_S\rmap A|_S\rmap\tau_S\rmap 0,
\end{equation}
where we have denoted by $\tau_S$ the vector subbundle
\[\tau_S:=\im(\rho|_S) \ \mathrm{mod}\ TS \ \subset\ \nu(S,M).\]
Note that, when $\tau_S=\nu(S,M)$, we recover transversals, and when $\tau_S=0_S$, we recover invariant submanifolds.

%

The following will play an important role in the study of such submanifolds.
\begin{definition}
Let $A\Ato M$ be a Lie algebroid and let $i_S:S\hookrightarrow M$ be a coregular submanifold. A \textbf{minimal transversal around $S$} is a transversal $i_X:X\hookrightarrow M$ containing $S$ that yields a direct sum decomposition
\begin{equation}\label{eq:def:min:trans}
\nu(S,M)=\tau_S\oplus \nu(S,X).
\end{equation}
\end{definition}

Notice that minimal transversals always exist. To see this, choose a Riemannian metric and take $X:=\exp(U)$, where $U$ is a small enough neighborhood of $S$ in a complement $C$ of $\tau_S$ in $\nu(S,M)$. We obtain a submanifold satisfying \eqref{eq:def:min:trans}, which can then be shrunk to a minimal transversal. The following is also straightforward.

\begin{lemma}\label{lemma:coreg:is:inv:in:transv}
If $X$ is a minimal transversal around $S$, then $S$ is an invariant submanifold for the Lie algebroid $A_X:=i_X^!A\Ato X$.
\end{lemma}

In fact, coregular submanifolds can be characterized as follows (for similar results in Poisson geometry, see \cite[Theorem 8.44]{CFM21} and \cite[Lemma 4.1]{CZ09}).

\begin{proposition}
Given a Lie algebroid $A\Ato M$, a submanifold $S\subset M$ is coregular if and only if it is an invariant submanifold inside a transversal $X\subset M$.
\end{proposition}

Next, we observe that minimal transversals are essentially unique. We sketch a proof which is similar to that of the case of Poisson manifolds from \cite[Lemma 2.2]{Weinstein83} and \cite[Theorem 4.3]{CZ09}. 

\begin{lemma}\label{lemma:equiv:min:trans}
Let $A\Ato M$ be a Lie algebroid and $i_S:S\hookrightarrow M$ be a coregular submanifold. If $X_0$ and $X_1$ are minimal transversals around $S$, then there exists an inner Lie algebroid isomorphism 
$\Phi:A|_{U_0}\diffto A|_{U_1}$, defined between open sets $U_0$ and $U_1$ containing $S$, whose base map $\varphi:U_0\diffto U_1$ fixes $S$ pointwise and sends $X_0\cap U_0$ to $X_1\cap U_1$.
\end{lemma}
\begin{proof}
Possibly after shrinking $X_0$ and $X_1$, one can join them by a smooth family of transversals around $S$,  $X_t:=i_t(X_0)$, $t\in [0,1]$, where $i:[0,1]\times X_0\to M$ is a smooth map such that $i_0=\id_{X_0}$ and $i_t|_S=\id_S$. 
Transversality implies the existence of a  time-dependent section $\al_t\in \Gamma(A|_U)$, $t\in [0,1]$, defined on some neighborhood $U$ of $S$ in $M$, such that 
\[\frac{\d i_t}{\d t}(x)=\rho(\al_t(i_t(x)))\ \mathrm{mod}\ T_{i_t(x)}X_{t},\]
for all $(t,x)\in [0,1]\times X_0$. Moreover, we may assume that $\al_t|_S=0$. Then the flow of the time-dependent section $\al_t$ gives the isomorphism $\Phi$ from the statement. 
\end{proof}

We are ready to discuss local models and linearization of a Lie algebroid $A\Ato M$ around a coregular submanifold $i_S:S \hookrightarrow M$. 
Fix a minimal transversal $X$ around $S$. Note that $A_S=i^!_SA$ can be identified also with the restriction of $A_X$ to the invariant submanifold $S$, $A_S=i^!_SA_X$. Then for $A_X$ we have the local model around the invariant submanifold $S$ (see Definition \ref{defi:linearizable:LieAlg})
\[A_S\ltimes \nu(S,X)\Ato \nu(S,X).\]
On the other hand, Theorem \ref{theorem:BLM} shows that $A$ is isomorphic around $X$ to the pullback Lie algebroid $p^!A_X$, where $p:\nu(X,M)\to X$ it the projection. This motivates the following definition.
\begin{definition}\label{defi:local:model:linear:coreg}
The \textbf{local model} of $A$ around the coregular submanifold $S$ is the Lie algebroid
\[p^!(A_S\ltimes \nu(S,X))\Ato \nu(S,M),\]
where $p:\nu(S,M)\to \nu(S,X)$ is the projection with kernel $\tau_S$. 

We say that $A$ is \textbf{linearizable} around $S$ if there is a Lie algebroid isomorphism $\Phi:p^!(A_S\ltimes \nu(S,X))|_{U_0}\diffto A|_{U_1}$, where $U_0\subset \nu(S,M)$ and $U_1\subset M$ are open neighborhoods of $S$, with $\Phi|_{A_S}=\id_{A_S}$ and such that the differential of its base map $\varphi:U_0\diffto U_1$ along $S$ induces the identity on $\nu(S,M)$. 
\end{definition}

The conditions on $\Phi$ along $S$ imply that, in the short exact sequence \eqref{eqseq}, $\Phi$ preserves not only the inclusion of $A_S$ but also the projection to $\tau_S$. Lemma \ref{lemma:equiv:min:trans} shows that local models associated with different choices of transversals are isomorphic, and so being linearizable is independent of the chosen minimal transversal. 

\begin{proposition}\label{prop:linearizable:iff:transverse}
Let $i_S:S\hookrightarrow M$ be a coregular submanifold, and $X$ a minimal transversal around $S$. Then $A\Ato M$ is linearizable around $S$ if and only if $A_X\Ato X$ is linearizable around $S$.
\end{proposition}
\begin{proof}
Fix a tubular neighborhood, and identify $M\simeq \nu(S,M)$ and $X\simeq \nu(S,X)$.

Assume that $A_X$ is linearizable, i.e.\ $A_X\simeq A_S\ltimes \nu(S,X)$ around $S$. Then also their pullbacks via the projection $p:\nu(S,M)\to \nu(S,X)$ are isomorphic. From 
Theorem \ref{theorem:BLM}, we have that $A\simeq p^! A_X$ around $X$. So $A$ is linearizable. 

Conversely, assume that we have an isomorphism $\Phi$ between $p^!(A_S\ltimes \nu(S,X))$ and $A$, defined around $S$, as in Definition \ref{defi:local:model:linear:coreg}. Then $\Phi$ induces an isomorphism between the restriction of $p^!(A_S\ltimes \nu(S,X))$ to $\nu(S,X)\subset \nu(S,M)$ and the restriction of $A$ to the transversal $X'=\varphi(\nu(S,X))\subset M$, where $\varphi$ is the base map of $\Phi$. This map is a linearization for $A_{X'}$. Lemma \ref{lemma:equiv:min:trans}, implies that $A_X\simeq A_{X'}$ around $S$, thus also $A_X$ is linearizable.
\end{proof}

One can also approach the linearization problem around a coregular submanifold $S$ by looking at another invariant submanifold, namely the one arising from the saturation of $S$.


\begin{definition}
Let $A\Ato M$ be a Lie algebroid. The \textbf{saturation} of $S\subset M$, denoted by $\Sat(S,M)$, is the union of all orbits of $A$ that hit $S$. The \textbf{local saturation} of $S$ inside the open neighborhood $U\subset M$, denoted by $\Sat(S,U)$, is the saturation of $S$ with respect to the Lie algebroid $A|_U\Ato U$. 
\end{definition}

The following shows that the local saturation of a coregular submanifold is smooth (see \cite{Geudens20} for this result in the Poisson and Dirac setting).
\begin{theorem}\label{theorem:coregular:intersection}
Let $A\Ato M$ be a Lie algebroid, and let $i: S\hookrightarrow M$ be a coregular submanifold. 
Then $S=N\cap X$, where the intersection is transverse, and
\begin{enumerate}[(i)]
    \item $N=\Sat(S,U)$ is a local saturation of $S$ and an embedded submanifold;
    \item $X$ is a minimal transversal around $S$.
\end{enumerate}
\end{theorem}
\begin{proof}
Let $X$ be a minimal transversal around $S$. By applying Theorem \ref{theorem:BLM}, we find a neighborhood $U$ of $X$ in $M$, a submersion with connected fibers $p:U\to X$, which is the identity on $X$ (corresponding to the retraction of the normal bundle), and an isomorphism $(p^!A_X)|_U\simeq A|_U$. Then $N:=p^{-1}(S)$ is smooth submanifold, $S=X\cap N$, where the intersection is transverse and, moreover, $N=\Sat(S,U)$.
\end{proof}

We have the following analog of Lemma \ref{lemma:coreg:is:inv:in:transv}.
\begin{lemma}
In the setting of Theorem \ref{theorem:coregular:intersection}, $S$ is a transversal for the Lie algebroid $A_N\Ato N$.
\end{lemma}

For a pair $(N,X)$ as in Theorem \ref{theorem:coregular:intersection}, we have that $\nu(S,N)=\tau_S$ and
\[ \nu(S,M)=\nu(S,N)\oplus \nu(S,X).\]
In particular, $\dim N=\dim \tau_S$. In fact, the following holds. 

\begin{proposition}
The germ around $S$ of the embedded submanifold $N$ from Theorem \ref{theorem:coregular:intersection} is independent of choices.  
\end{proposition}


\begin{proof}
Let $N=\Sat(S,U)$ and $\tilde{N}=\Sat(S,\tilde{U})$ be as in the theorem. Then $\tilde{N}$ and $N$ have the same dimension, and we write $N=p^{-1}(S)$, where $p:U\to X$ is as in the proof. Consider an open $S\subset O\subset U\cap \tilde{U}$, such that $p|_O$ also has connected fibers. Then $N':=O\cap N$ satisfies $N'=\Sat(S,O)$, and it is open in $N$. So $N$ and $N'$ have the same germ at $S$. On the other hand, $N'\subset \tilde{N}$ is an embedded submanifold and $\dim N'=\tilde{N}$, so $N'$ is open in $\tilde{N}$. Hence, $N'$ and $\tilde{N}$ have the same germ at $S$. 
\end{proof}

\begin{proposition}\label{local:model:saturation:Lie:alg}
Let $i_S:S\hookrightarrow M$ be a coregular submanifold for $A\Ato M$. The local model of $A$ around $S$ is isomorphic to the local model of $A$ around the invariant submanifold $N=\Sat(S,U)$, for some small enough open neighborhood $U$ of $S$. Moreover, $A$ is linearizable around $S$ if and only if $A$ is linearizable around $N=\Sat(S,U)$, for some small enough open neighborhood $U$ of $S$. 
\end{proposition}

\begin{proof}
By choosing a tubular neighborhood, we may replace $M$ by $\nu(S,M)$. Fix a decomposition $\nu(S,M)=\tau_S\oplus E$. Then $E$ is a transversal and, after applying Theorem \ref{theorem:BLM} and changing the tubular neighborhood, we may assume that $A=p_{E}^!A_{E}$, where $p_{E}: \tau_S\oplus E\to E$ is the projection. Since $S$ is an invariant submanifold of $E$, we have that $\Sat(S,\nu(S,M))=\tau_S$, and the restriction of $A$ to $\tau_S$ is the pullback of $A_S$ along the projection $p:\tau_S\to S$
\[A_{\tau_S}=i_{\tau_S}^!p_{E}^!A_{E}=p^!A_S.\]
Moreover, the obvious isomorphism $p^{!}E\simeq \tau_S\oplus E$ is in fact an isomorphism between the pullback of the normal representation of $A_S$ in $A_{E}$ and the normal representation of $A_{\tau_S}$ in $A$. We obtain an isomorphism between the pullback of the action Lie algebroid and the action Lie algebroid
\[p_{E}^!(A_S\ltimes E)\simeq A_{\tau_S}\ltimes (\tau_S\oplus E).\]
In other words, the two local models are isomorphic. An argument similar to the proof of Proposition \ref{prop:linearizable:iff:transverse} shows that the linearization problem of $A$ at $S$ and the linearization problem of $A$ at $\tau_S$, in a neighborhood of $S$, are equivalent.
\end{proof}

Theorem \ref{Theorem:coregular:normal:form} (i) is a direct consequence of Proposition \ref{local:model:saturation:Lie:alg} and the linearization theorem for proper Lie groupoids from \cite{dHFe18}.


\subsection{Coregular submanifolds in Dirac manifolds}

Let $(M,\L)$ be a Dirac manifold and let $i_S:S\hookrightarrow M$ be a coregular submanifold for $\L$, in the Lie algebroid sense. Any minimal transversal $i_X:X\hookrightarrow M$ for $S$ comes with a Dirac structure $\L_X:=i_X^!\L$. By Lemma \ref{lemma:equiv:min:trans}, any other minimal transversal $X'$ is related to $X$ by an inner automorphism of $\L$, more precisely, by the flow of a time-dependent section of $\L$ with vanishes along $S$. Recall that such an automorphism of $\L$ is given by a diffeomorphism fixing $S$ composed with a gauge transformation by a closed 2-form whose pullback to $S$ vanishes (see, e.g., \cite{Gual11}). This implies the following result.

\begin{proposition}\label{proposition:transversal:well-defined}
For any two minimal transversals $X_0$ and $X_1$ around $S$ the Dirac manifolds $(X_i,\L_{X_i})$ are equivalent around $S$ (see Definition \ref{def:iso:Dirac}).
\end{proposition}

Proposition \ref{proposition:transversal:well-defined}, Lemma \ref{lemma:equivalent:give:equivalent:jet} and Proposition \ref{prop:moser:Dirac} imply that the following notions are independent of the chosen minimal transversal.

\begin{definition}
Let $S$ be a coregular submanifold of a Dirac manifold $(M,\L)$, and $X$ a minimal transversal through $S$. 
\begin{enumerate}[(i)]
    \item We say that $\L$ is \textbf{partially split} at $S$ if the first order jet at $S$ of the Dirac structure $\L_X$ is partially split.
    \item Let $\L_X$ be partially split at $S$, and let $(X_0,\L_{X_0})$ be the local model of $\L_X$ around $S$ for some IM connection 1-form. The \textbf{local model} for $\L$ around $S$ is a Dirac manifold \[(M_0,\L_0):=(p^{-1}(X_0),p^!\L_{X_0}),\] where $p:\nu(S,M)\to \nu(S,X)$ is the projection with kernel $\tau_S$.
    \item We say that $\L$ is \textbf{linearizable} around $S$ if it is partially split at $S$, and $(M,\L)$ is equivalent to the local model $(M_0,\L_0)$ around $S$. 
\end{enumerate}
\end{definition}

Theorem \ref{theorem:BLM:Dirac}, i.e., the local normal form theorem around transversals, implies the Dirac version of Proposition \ref{prop:linearizable:iff:transverse}.

\begin{proposition}
A Dirac manifold $(M,\L)$ is linearizable around the coregular submanifold $S$ if and only if for some (and any) minimal transveral $X$ through $S$, the induced Dirac structure $\L_X$ is linearizable around the invariant submanifold $S$.
\end{proposition}

Theorem \ref{Theorem:coregular:normal:form} (ii) follows by applying this proposition, Theorem \ref{thm:normal:form:Dirac} and the linearization theorem for proper Lie groupoids from \cite{dHFe18}.

As for Lie algebroids, the linearization problem can be expressed also in terms of the saturation. Namely, for a coregular submanifold $S$ of $(M,\L)$, Theorem \ref{theorem:coregular:intersection} implies the existence of an open neighborhood $U$ of $S$ for which $N:=\Sat(S,U)$ is an embedded submanifold with $T_SN=\tau_S$. The analog of Proposition \ref{local:model:saturation:Lie:alg} holds.

\begin{proposition}\label{local:model:saturation:Dirac}
The Dirac structure $\L$ is partially split at $S$ if and only if $\L$ is partially split at the invariant submanifold $N=\Sat(S,U)$, for some small enough open neighborhood $U$ of $S$. In this case, the two local models of $\L$ are equivalent around $S$, and $\L$ is linearizable around $S$ if and only if $\L$ is linearizable around $N=\Sat(S,U)$, for some small enough open neighborhood $U$ of $S$.
\end{proposition}

For the proof, we will need the following fact.

\begin{lemma}
Let $A\Ato M$ be a Lie algebroid, $I\subset A$ a bundle of ideals, and $f:N\to M$ a map transverse to the anchor of $A$. Then $f^!A\Ato N$ is a Lie algebroid which contains $f^!I=N\times_MI$ as a bundle of ideals. Moreover, if $I$ is partially split, then so is $f^!I$. 
\end{lemma}

\begin{proof}
The pullback of an IM form $(L,l)\in \Omega^1_{\imult}(A,I)$, with $l|_{I}=\id$, is also an IM 1-form $f^!(L,l)\in \Omega^1_{\imult}(f^!A,f^!I)$, with $(f^!l)|_{f^!I}=\id$. 
\end{proof}

\begin{proof}[Proof of Proposition \ref{local:model:saturation:Dirac}]
As in the proof of Proposition \ref{local:model:saturation:Lie:alg}, assume that
\[ M=\tau_S\oplus E, \] 
with $E$ a transversal and $\tau_S=\Sat(S,M)$. Applying Theorem \ref{theorem:BLM:Dirac}, we have an equivalence around $E$
\[ \L\simeq p^!_{E}\L_{E}. \] 

To prove the statement about partially split, observe that the first jet of the invariant submanifold $S\subset (E,\L_{E})$ is given by the Lie algebroid $A_S:=\L_{E}|_{S}$ and has bundle of ideals $\ka:=\nu^*(S,E)\simeq E^*$. For the invariant submanifold $\tau_S\subset (\tau_S\oplus E, p_{E}^!\L_{E})$, its corresponding Lie algebroid $A_{\tau_S}=\L|_{\tau_S}$ is isomorphic to the pullback $p^!A_S$, where $p:\tau_S\to S$ is the projection, and its bundle of ideals is $p^!\ka$. By the lemma above, if $\ka$ is a partially split for $A_S$ then  $p^!\ka$ is partially split for $p^!A_S$. Conversely, we can also write $A_S=i_S^!(p^!A_S)$, where $i_S:S\hookrightarrow \tau_S$ is the zero-section, which is a transversal for $p^!A_S$. For the bundle of ideals, we have $\ka=i_S^!(p^!\ka)$ so, by the lemma, if $p^!\ka$ is partially split for $p^!A_S$ then $\ka$ is a partially split for $A_S$. 

Assuming now that the partially split condition holds, let $(L,l)\in\Omega^1_\imult(A_S,\ka)$ be an IM form, and let $\L_0$ be the corresponding local model on $E$. Then the local model of $\L$ around $S$ is $p^!\L_{0}$. This coincides with the local model of $\L$ around $\tau_S$ constructed using the pullback IM form $p^!(L,l)$ for $p^!\ka\subset p^!A_S$.

The equivalence between the two linearization problems can be proven similarly to Proposition \ref{prop:linearizable:iff:transverse}. 
\end{proof}


\end{document}